\documentclass[12pt]{amsart}

\usepackage{amsmath,amssymb,amsthm,graphicx,color,amscd}
\usepackage[usenames,dvipsnames]{xcolor}
\usepackage{enumitem}

\usepackage[small,nohug,heads=vee]{diagrams}

\usepackage{hyperref}

\numberwithin{equation}{section}
\hypersetup{bookmarksdepth=3}

\theoremstyle{plain}
\newtheorem{theorem}[equation]{Theorem}

\newtheorem{corollary}[equation]{Corollary}

\newtheorem{proposition}[equation]{Proposition}
\newtheorem{lemma}[equation]{Lemma}

\theoremstyle{definition}
\newtheorem{definition}[equation]{Definition}

\newtheorem{question}[equation]{Question}

\theoremstyle{remark}
\newtheorem{remark}[equation]{Remark}

\newcommand{\A}{\mathcal{A}}

\newcommand{\al}{\alpha}
\newcommand{\ann}{\operatorname{Ann}}

\newcommand{\be}{\beta}
\newcommand{\ben}{\begin{enumerate}}
\newcommand{\bit}{\begin{itemize}}
\newcommand{\C}{\mathbb{C}}

\newcommand{\com}{\operatorname{com}}

\newcommand{\de}{\delta}

\newcommand{\dist}{\operatorname{dist}}
\newcommand{\een}{\end{enumerate}}
\newcommand{\eit}{\end{itemize}}

\newcommand{\eps}{\varepsilon}

\newcommand{\fg}{\mathfrak{g}}
\newcommand{\fh}{\mathfrak{h}}

\newcommand{\ga}{\gamma}

\renewcommand{\H}{\mathbb{H}}

\newcommand{\id}{\operatorname{id}}

\newcommand{\La}{\Lambda}

\newcommand{\Lip}{\operatorname{Lip}}
\newcommand{\loc}{\operatorname{loc}}
\newcommand{\lra}{\longrightarrow}

\newcommand{\N}{\mathbb{N}}

\newcommand{\osc}{\operatorname{osc}}

\newcommand{\om}{\omega}
\newcommand{\Om}{\Omega}

\newcommand{\perm}{\operatorname{Perm}}

\newcommand{\R}{\mathbb{R}}
\newcommand{\ra}{\rightarrow}

\definecolor{gray}{gray}{0.7}

\newcommand{\si}{\sigma}
\newcommand{\Si}{\Sigma}

\newcommand{\Span}{\operatorname{span}}
\newcommand{\spt}{\operatorname{spt}}

\renewcommand{\th}{\theta}

\newcommand{\vol}{\operatorname{vol}}

\newcommand{\we}{\wedge}
\newcommand{\weight}{\operatorname{weight}}
\newcommand{\wt}{\operatorname{wt}}

\def\Xint#1{\mathchoice
{\XXint\displaystyle\textstyle{#1}}%
{\XXint\textstyle\scriptstyle{#1}}%
{\XXint\scriptstyle\scriptscriptstyle{#1}}%
{\XXint\scriptscriptstyle\scriptscriptstyle{#1}}%
\!\int}
\def\XXint#1#2#3{{\setbox0=\hbox{$#1{#2#3}{\int}$ }
\vcenter{\hbox{$#2#3$ }}\kern-.6\wd0}}

\def\av{\Xint-}

\newcommand{\modp}{\mathrm{mod}_p}

\begin{document}

\title[Pansu pullback and exterior differentiation]{Pansu pullback and exterior differentiation for Sobolev maps on Carnot groups}

\author{Bruce Kleiner}
\thanks{BK was supported by NSF grants DMS-1711556 and DMS-2005553, and a Simons Collaboration grant.}
\email{bkleiner@cims.nyu.edu}
\address{Courant Institute of Mathematical Science, New York University, 251 Mercer Street, New York, NY 10012}
\author{Stefan M\"uller}
\thanks{SM has been supported by the Deutsche Forschungsgemeinschaft (DFG, German Research Foundation) through
the Hausdorff Center for Mathematics (GZ EXC 59 and 2047/1, Projekt-ID 390685813) and the 
collaborative research centre  {\em The mathematics of emerging effects} (CRC 1060, Projekt-ID 211504053).  This work was initiated during a sabbatical of SM at the Courant Institute and SM would like to thank  R.V. Kohn and the Courant Institute
members and staff for 
their  hospitality and a very inspiring atmosphere.}
\email{stefan.mueller@hcm.uni-bonn.de}
\address{Hausdorff Center for Mathematics, Universit\"at Bonn, Endenicher Allee 60, 53115 Bonn}
\author{Xiangdong Xie}
\thanks{XX has been supported by Simons Foundation grant \#315130.}
\email{xiex@bgsu.edu}
\address{Dept. of Mathematics and Statistics, Bowling Green State University, Bowling Green, OH 43403}

\begin{abstract}
We show that in an $m$-step Carnot group, a probability measure with finite $m^{th}$ moment has a well-defined Buser-Karcher center-of-mass, which is a polynomial in the moments of the measure, with respect to exponential coordinates.  Using this, we improve the main technical result of \cite{KMX1} concerning  Sobolev mappings between Carnot groups.  As a consequence, a number of rigidity and structural results from \cite{KMX1,KMX2,kmx_rumin,kmx_iwasawa}
  hold under weaker assumptions on the Sobolev exponent.  We also give applications to quasiregular mappings following \cite{reshetnyak_space_mappings_bounded_distortion,heinonen_holopainen,vodopyanov_foundations},  extending earlier work in the $2$-step case to general Carnot groups.
\end{abstract}

\maketitle

\tableofcontents

\section{Introduction}

This is part of a series of papers  on geometric mapping theory in Carnot groups, in which we establish regularity, rigidity, and partial rigidity results for bilipschitz, quasiconformal, or more generally, Sobolev mappings, between Carnot groups \cite{KMX1,KMX2,kmx_rumin,kmx_iwasawa}. In \cite{KMX1} we showed that Reshetnyak's theorem on pullbacks of differential forms has a partial generalization to mappings between Carnot groups (see also \cite{dairbekov_morphism_property_bounded_distortion,vodopyanov_bounded_distortion,vodopyanov_foundations}). 
Our aim here is to strengthen the pullback theorem from \cite{KMX1} by relaxing the assumptions on the Sobolev exponent.  This yields new applications to quasiregular mappings, in addition to stronger versions of results from \cite{KMX1,KMX2,kmx_rumin,kmx_iwasawa}.
We expect further applications to geometric mapping theory in Carnot groups, in particular to understanding the threshold between flexibility and rigidity.  
 We refer the interested reader to \cite{KMX1} for more background.

Before stating our results, we briefly recall some facts and notation; see Section~\ref{sec_prelim} for more detail.  

Let $G$ be a Carnot group with Lie algebra $\fg$, grading $\fg=\oplus_{j\geq 1}^sV_j$, and  dilation group $\{\de_r:G\ra G\}_{r\in (0,\infty)}$.  The exponential map $\exp:\fg\ra G$ is a diffeomorphism, with inverse $\log:G\ra \fg$.   Without explicit mention, in what follows all Carnot groups will be equipped with Haar measure and a Carnot-Caratheodory metric denoted generically by $d_{CC}$.  If $f:G\supset U\ra G'$ is a Sobolev mapping between Carnot groups, where $U$ is open, then $f$ has a well-defined approximate Pansu differential $D_Pf(x):G\ra G'$ for a.e. $x\in U$, which is a graded group homomorphism (Theorem~\ref{th:Lp*_pansu_differentiability_new}); 
 By abuse of notation, we also denote the associated homomorphism of graded Lie algebras by $D_Pf(x):\fg\ra \fg'$; furthermore, for the sake of brevity we will often shorten ``approximate Pansu differential'' to ``Pansu differential''.
  If $\om$ is a differential form defined on the range of $f$, then the Pansu pullback $f_P^*\om$ is given by $f_P^*\om(x)=(D_Pf(x))^*\om(f(x))$ for a.e. $x\in U$. 

Let $G$ be a Carnot group with Lie algebra $\fg$.  For every $x\in G$ we let $\log_x:G\ra \fg$ be the logarithm map ``centered at $x$'', i.e. $\log_x(y):=\log(x^{-1}y)$.  We recall \cite{karcher_buser_almost_flat_manifolds,KMX1} that if $\nu$ is a compactly supported probability measure in  $G$, then $\nu$ has a well-defined Buser-Karcher center of mass $\com_\nu$, which is characterized as the unique point $x\in G$ such that $\nu$ is ``balanced'' with respect to logarithmic coordinates centered at $x$:  $\int_G\log_x\,d\nu=0$.  Our first result is a generalization of this center of mass to the case of measures with noncompact support.

\medskip
\begin{theorem}[Theorem~\ref{le:C_nu_diffeomorphism}]~
\label{thm_com_finite_moment_intro}
Suppose $G$ is an $m$-step Carnot group, and $\nu$ is a probability measure on $G$ with finite $m^{th}$-moment, i.e. for some $x\in G$ we have
$$
\int_Gd_{CC}^m(x,y)\,d\nu(y) <\infty\,.
$$
Then for every $x\in G$ the map $\log_x$ is integrable w.r.t. $\nu$, and there is a unique point $\com_\nu\in G$ such that $\int_G\log_{\com_\nu}\,d\nu=0$.  Moreover, $\log(\com_\nu)$  is a polynomial in the polynomial moments of the pushforward measure $(\log_{x_0})_*\nu$, for any $x_0\in G$.
\end{theorem}

\begin{remark}
With minor modifications, the same proof works for general simply connected nilpotent groups.
\end{remark}

Applying Theorem~\ref{thm_com_finite_moment_intro} in a standard way, one may define a mollification process for $L^m_{\loc}$-mappings into an $m$-step Carnot group, which yields a family of smooth approximations.

For a Sobolev mapping $f$ with mollification $f_\rho$, our main result relates the ordinary pullback $f_\rho^*\om$ of a differential form $\om$ with the Pansu pullback $f_P^*\om$, as defined above.  To state the result, we require the notion of the weight $\wt(\al)$ of a differential form $\al$; this is defined using the decomposition of $\La^*\fg$ with respect to the diagonalizable action of the Carnot scaling, see Subsection~\ref{subsec_differential_forms_on_carnot_groups}.

\begin{theorem}[Approximation theorem]  \label{thm_approximation_theorem_intro}~
Let $G$, $G'$ be Carnot groups, and $f: U \to G'$ be a map in $W^{1,p}_{loc}(U, G')$, where $U\subset G$ is open.  Suppose:
\bit
\item  $\eta\in \Om^k(G)$, $\om\in\Om^\ell(G')$ are differential forms, where $k+\ell=N:=\dim G$.
\item $\eta$ is left-invariant.
\item  $\om$ is continuous and bounded.
\item $\wt(\om)+\wt(\eta)\leq -\nu$, where $\nu$ is the homogeneous dimension of $G$.
\item $p\geq -\wt(\om)$.
\item $\frac1p\leq \frac1m+\frac{1}{\nu}$, where $G'$ has step $m$.
\eit
Then
\begin{equation*}  
 f_\rho^*\omega \wedge \eta \to f_P^* \omega \wedge \eta   \quad   \text{in $L^s_{\loc}(U)$ with $s = \frac{p}{-\wt(\omega)}$,}
 \end{equation*}
where $f_\rho$ is the mollification of $f$ at scale $\rho$,  see Section~\ref{subsec_mollifying_maps_between_carnot_groups}.    In particular, when $\om\in\Om^N(G')$ and $\wt(\omega)\leq -\nu$, then
 \begin{equation*}  
 f_\rho^*\omega  \to f_P^* \omega   \quad  
  \text{in $L^{\frac{p}{\nu}}_{\loc}(U)$.}
 \end{equation*}

\end{theorem}

\bigskip
  We refer the reader to Section~\ref{sec_pansu_pullback_mollification} for more refined statements.  

Although the overall outline of the proof of Theorem~\ref{thm_approximation_theorem_intro} is the same as for \cite[Theorem 1.18]{KMX1}, the fact that  $p\leq \nu$ creates several complications: a Sobolev mapping $f\in W^{1,p}_{\loc}(U,G')$ as in the theorem need not be either  (classically) Pansu differentiable almost everywhere or continuous; in particular, the argument cannot be localized in the target.   

\bigskip
As immediate consequences of Theorem~\ref{thm_approximation_theorem_intro}, the rigidity and partial rigidity results from \cite{KMX1} hold under weaker assumptions on the exponent.  For instance:
\ben
\item    Let $\{G_i\}_{1 \le i \le n}$, $\{G'_j\}_{1 \le j \le n'}$ be collections on Carnot groups where each $G_i, G'_j$ is nonabelian and does not admit
a nontrivial decomposition as a product of Carnot groups. Let $G = \prod_i G_i$, $G' = \prod_j G'_j$. Set 
$ K_i : = \{ k \in \{1, \ldots, n\} : G_k \simeq G_i \}$
and assume that 
\begin{equation}  \label{eq:p_rigidity_intro}
 p \ge \max\{ \nu_i - 1:   |K_i| \ge 2 \}
 \end{equation}
 where $\nu_i$ denotes the homogeneous dimension of $G_i$. 
Assume that  $f: G \supset U \to G'$ is a $W^{1,p}_{\rm loc}$ mapping,
$U = \prod_i U_i$ is a product of connected open sets $U_i \subset G_i$,
and the (approximate) Pansu differential $D_P f(x)$   is an isomorphism for a.e.\ $x \in G$.
Then $f$ coincides almost everywhere with a product mapping, modulo a permutation of the factors, 
 see Theorem~\ref{th:product_rigidity}  below. In \cite[Theorem 1.1]{KMX1} the result was proved under the stronger hypothesis 
 $p > \sum_i \nu_i$.  
If each $G_i$ is either a higher Heisenberg group  $\H_{m_i}$  (with $m_i \ge 2$) or any complex Heisenberg group $\H^\C_{m_i}$ (with $m_i \ge 1$)
then the condition 
  \eqref{eq:p_rigidity_intro} can be  improved to $p \ge 2$,
 see Corollary~\ref{co:product_rigidity} below.

 \item If $\H^\C_m$ is the complexification of the $m^{th}$ Heisenberg group $\H_m$, $U \subset \H^\C_m$ is open and connected 
and $f: U\ra\H^\C_m$ is a $W^{1,2m+1}_{\loc}$-mapping such that  the (approximate) Pansu differential $D_Pf(x)$   is an isomorphism for a.e. $x$, then $f$ coincides almost everywhere with a holomorphic or antiholomorphic mapping, cf. \cite[Theorem 1.6]{KMX1} 
for the same result under the stronger condition $p > 4m +2$.
\een

\bigskip
Another application of Theorem~\ref{thm_approximation_theorem_intro} is to quasiregular mappings between Carnot groups, addressing questions originating in \cite{rickman_quasiregular_mappings,heinonen_holopainen}.    We recall that a fundamental step in Reshetnyak's approach to quasiregular mappings in $\R^n$ is showing that the composition of an $n$-harmonic function with a quasiregular mapping is a solution to a quasilinear elliptic PDE;  this ``morphism property'' depends crucially on  the fact that pullback commutes with exterior differentiation \cite{reshetnyak_space_mappings_bounded_distortion}.  Using Theorem~\ref{thm_approximation_theorem_intro}, we are able to extend earlier work of \cite{vodopyanov_foundations} (see also \cite{heinonen_holopainen}), so as to generalize a portion of Reshetnyak's theory to all Carnot groups.   In particular, if a Carnot group $G$ has homogeneous dimension $\nu$ and $f:G\supset U\ra G$ is a quasiregular mapping, then (see Section~\ref{sec_quasiregular_mappings} for more details):
\bit
\item  The ``morphism'' property, which was first shown by Reshetnyak in the $\R^n$ case,  holds for locally Lipschitz $\nu$-harmonic functions: if $u:G\ra \R$ is a locally Lipschitz $\nu$-harmonic function then the composition $u\circ f$ is $\A$-harmonic.
\item $f\in W^{1,\nu'}$ for some $\nu'>\nu$.
\item $f$ is H\"older continuous, Pansu differentiable almost everywhere, and maps null sets to null sets.
\eit

\bigskip\bigskip
We conclude with some open questions.  

\begin{question}
What is the exponent threshold for rigidity/flexibility in the results mentioned above?  
\end{question}
For instance,  suppose $f:\H\times\H\ra\H\times\H$ is a $W^{1,p}$-mapping whose Pansu differential is an isomorphism almost everywhere.  For which $p$ must $f$ agree with a product mapping almost everywhere?  Are there counterexamples when $p=1$?  Optimal Sobolev exponents were obtained for an analogous product rigidity question in the Euclidean setting in \cite{kmsx_infinitesimally_split_globally_split,kmsx_counterexample}.

\medskip

We recall that a Carnot group $G$ is rigid in the sense of Ottazzi-Warhurst if for any connected open subset $U\subset G$, the family of smooth contact embeddings $U\ra G$ is finite dimensional.  We conjectured \cite[Conjecture 1.10]{KMX1} that quasiconformal homeomorphisms of rigid Carnot groups are smooth.  One may ask if there is a rigidity/flexibility threshold for these groups.
\begin{question}
Let $f:G\supset U\ra G$ be a $W^{1,p}$-mapping, where $U$ is an open subset of an Ottazzi-Warhurst rigid Carnot group, and  $D_Pf(x)$ 
  is an isomorphism for a.e. $x\in U$ (recall that $D_Pf(x)$ denotes the approximate Pansu differential).  For which $p$ can we conclude that $f$ is smooth?  What if $f$ is a homeomorphism?\footnote{Ottazi-Warhurst 
  showed that $C^2$ quasiconformal homeomorphisms of rigid Carnot groups are smooth.
Recently  Jonas Lelmi improved this result,  replacing the $C^2$ regularity assumption with $C^1$ (or even Euclidean bilipschitz); 
the same result was shown 
by Alex Austin for  the $(2,3,5)$ distribution \cite{lelmi,austin_235}.}
\end{question}

Motivated by \cite{iwaniec_martin_quasiregular_even_dimensions,iwaniec_p_harmonic_tensor_quasiregular_mappings} one may ask about minimal regular requirements for quasiregular mappings.  
\begin{question}
Suppose $f:G\supset U\ra G$ is  a weakly quasiregular mapping, i.e. $f\in W^{1,p}$ and for some $C$ we have $|D_hf|^\nu\leq C\det D_Pf$ almost everywhere.  For which $p<\nu$ can we conclude that  $f\in W^{1,\nu}$?
\end{question}

\bigskip
\subsection*{Organization of the paper}~
We review some background material on Carnot groups and Sobolev mappings in Section~\ref{sec_prelim}.  Section~\ref{sec_center_of_mass_mollification} establishes  existence and estimates for the center of mass for measures which satisfy a moment condition, and establishes bounds for the associated mollification procedure.  The proof of the main approximation theorem and some applications to the exterior derivative are proven in Section~\ref{sec_pansu_pullback_mollification}.  Section~\ref{sec_quasiregular_mappings} gives applications to quasiregular mappings.  For the convenience of the reader, we have included proofs of some background results in the appendices.  In Appendix~\ref{se:W1p_differentiability}, we give a new direct proof of the $L^p$ Pansu differentiability of Sobolev mappings; see Subsection~\ref{subsec_sobolev} and Appendix~\ref{se:W1p_differentiability} for a comparison with the original proof by Vodopyanov.  In Appendix~\ref{se:app_compact} we prove the compact Sobolev embedding,  and in Appendix~\ref{se:sobolev} we discuss Sobolev spaces defined using weak upper gradients, collecting some results from the literature, and comparing with with distributional approach of Reshetnyak and Vodopyanov.

\section{Preliminaries}
\label{sec_prelim}
\subsection{Carnot groups}~
In this subsection  we recall  some  standard  facts about Lie groups, in particular nilpotent Lie groups and  Carnot groups, and prove a simple estimate for the 
nonlinear term in the Baker-Campbell-Hausdorff (BCH) formula.  This will be useful  to define  the center of mass for probability measures  which do not necessarily have compact support, but only satisfy bounds on certain moments.

Our main interest is in Carnot groups and the reader may focus on this case. Since the construction of the center of mass extends
to connected, simply connected  nilpotent groups without additional effort, we include a short discussion of nilpotent groups as well. Since the facts mentioned below will be standard for most geometers, our discussion is more calibrated for analysts.

\medskip

\medskip

Let $G$ be a Lie group  of dimension $N$. In this paper we will only consider connected, simply connected
Lie groups.
We give  the tangent space $T_e G$ at the identity the structure of a Lie algebra $\fg$  in the usual way: each  tangent vector  $X \in T_e G$
can be extended to a left-invariant vectorfield $\tilde X$ through push-forward by left translation $\ell_a (g) = a g$, i.e  $\tilde X(a) f = X (f \circ \ell_a)$. 
For two left-invariant vector fields $\tilde X$ and $\tilde Y$ one easily sees that the commutator $[\tilde X, \tilde Y] = \tilde X \tilde Y - \tilde Y \tilde X$ is
a left-invariant vectorfield. We define the Lie bracket on $T_e G$ by $[X,Y] = [\tilde X, \tilde Y](e)$. In the following we do not distinguish between
$X$ and $\tilde X$. Similarly we identify left-invariant differential $k$-forms on $G$ with $\Lambda^k \fg$.

The descending series of the Lie algebra is defined   by $\fg_1 = \fg$ and $\fg_{i+1} = [\fg, \fg_i]$ where the right hand side denotes the linear space generated by  of all brackets of the form $[X, Y]$
with $X \in \fg$, $Y \in \fg_i$. 
We say that $G$ is a nilpotent group of step $m$ if $[\fg]_m \ne \{0\}$ and $[\fg]_{m+1} = \{0\}$. 
We say that $G$ is a Carnot group  of step $m$ if, in addition,  $\fg$ is graded, i.e.  
if  we are given a direct sum decomposition $\fg = \oplus_{j=1}^m V_j$ (as a vectorspace) with $V_{j+1} = [V_1, V_j]$ for $1 \le j \le m-1$.  
In the  general nilpotent case it will be convenient to introduce   subspaces $W_1, \ldots, W_m$ such that $\fg_{i} = W_i \oplus  \fg_{i+1}$. 
There is no canonical choice of the spaces $W_i$ (except for $W_m$),  but for our analysis any choice will do (see also
Remark~\ref{re:norms_and_complement}  below).
In the Carnot and general cases, respectively, we have
\begin{equation}   \label{eq:bracket_grading}
[V_i, V_j]   \subset \fg_{i+j}  \quad \text{and}  \quad  [W_i, W_j] \subset  \oplus_{k=i+j}^{m} W_k.
\end{equation}

By uniqueness of  solutions of ordinary differential equations the integral curve $\gamma_X: \R \to G$  of a left invariant vectorfield $X$ with $\gamma_X(0) = e$     is a subgroup.
We define the exponential map $\exp: \fg \to G$ by $\exp X(e) = \gamma_X(1)$. Thus $\exp: \fg \to G$ is smooth. 
By the Baker-Campbell-Hausdorff (BCH) Theorem for sufficiently small
$X$ and $Y$ we have $\exp X \exp Y =  \exp (X +  Y + P(X,Y))$ where $P(X,Y)$ is a series in iterated Lie brackets of $X$ and $Y$, see, e.g.
\cite[eqn. (2), p.\ 12]{corwin_greenleaf_book} or
\cite[Thm. 4.29]{Michor} .

 For a nilpotent Lie group of step $m$ 
the Lie brackets of order $m+1$ and higher vanish. Then  the expression $P(X,Y)$ is a polynomial, 
the exponential map is a diffeomorphism  and the BCH formula holds for all $X$ and $Y$  \cite[Thm. 1.21]{corwin_greenleaf_book}.  We often write $\log = \exp^{-1}$,  
and denote the induced group action on $\fg$ by
$$
X \ast Y :=\log(\exp X\exp Y)= X + Y + P(X,Y)\,.
$$
 One can use $\log: G \to \fg$ as a global chart for $G$ with the group operation given by $\ast$, but we will usually not do this. 
We denote by  $\pi_i$ the projection  $\fg \to V_i$ (or $\fg \to W_i$ for nilpotent groups). It follows from \eqref{eq:bracket_grading}
that $\pi_i [X,Y]$ depends only on $\pi_1(X), \ldots, \pi_{i-1}(X)$
and $\pi_1(Y), \ldots, \pi_{i-1}(Y)$. Thus the differential of $P$ with respect to the first or second variable is block lower triangular with 
respect to the decompositions  $\fg = \oplus_{i=1}^m V_i$ or $\fg = \oplus_{i=1}^m  W_i$, with zero entries on the block diagonal. 
 It follows  that  the Lebesgue measure $\mathcal L^N$ on $\fg$ is invariant under the left and right  group operation $\ast$. 
Thus the push-forward measure $\exp_* \mathcal L^N$ is the biinvariant Haar measure on $G$ (up to a factor). 

The horizontal bundle $\mathcal H \subset TG$ is the span of the left-invariant vectorfields $X$ which satisfy $X(e) \in V_1$ (or $X(e) \in W_1$ in the nilpotent case). 
We fix a scalar product on $\fg$. This induces a   left-invariant metric on $G$ by left-translation.  The Carnot-Carath\'eodory distance on $G$ as the shortest length of horizontal curves, i.e.
\begin{equation}
 d_{CC}(x, y) =  \inf \{ \int_a^b |\gamma'(t)| \, dt :  \text{  $\gamma: [a,b] \to G$ rectifiable, $\gamma'(t) \in \mathcal H$}  \}.
 \end{equation}
Push-forward by left translation preserves the horizontal bundle. Thus the left translation of a horizontal curve is horizontal
and  the metric $d_{CC}$ is left-invariant. By Chow's theorem every two points in $G$ can be connected by a horizontal curve
of finite length so that $d_{CC}(x,y) < \infty$ for all $x, y \in G$. Moreover $d_{CC}$ induces the usual manifold topology on $G$
\cite[Thm 2.1.2 and Thm 2.1.3 ]{montgomery_book}.

\bigskip

We now focus on Carnot groups.
 We define a one parameter group of dilations $\delta_r: \fg \to \fg$ by $\delta_r X = r^j X$ for $X \in V_j$ and linear extension. Then $\delta_r [X, Y] = [\delta_r X, \delta_r Y]$
so that $\delta_r$ is a Lie algebra homomorphism. Since $P(X,Y)$ is a sum of iterated Lie brackets it follows that $\delta_r (X \ast Y) =( \delta_r X) \ast (\delta_r Y)$. 
Thus $\exp \circ \delta_r \circ \exp^{-1}: G \to G$ is a group homomorphism which we also denote by $\delta_r$.  
Then $\delta_r (\ell_a x) = \ell_{\delta_r a} \delta_r x$.  Since $\delta_r$ as a map on $\fg$ preserves $V_1$ and is scaling by $r$ on $V_1$ it follows that 
$\delta_r$ maps horizontal curves to horizontal curves and 
\begin{equation}
d_{CC}(\delta_rx, \delta_r y) = r d_{CC}(x,y).
\end{equation}
Since $d_{CC}$ is also left-invariant, the bi-invariant measure of a ball $B(x,r)$ in the $d_{CC}$   metric is given by
\begin{align}
& \, \mu(B(x,r)) = \mu(B(e,r)) =  \mu(\delta_r B(e,1))  \\
 =& \,  \mathcal L^N(\delta_r \log B(e,1)) = r^\nu \mu(B(e,1))
\end{align}
where 
\begin{equation}
\nu := \sum_{j=1}^m  j \dim \fg_j  \quad \text{is the homogeneous dimension of $G$.}
\end{equation}

We define a Euclidean norm  $| \cdot|_e$ on $V_j$ by restriction of  the scalar product on $\fg$ to $V_j$.
Recall that $\pi_j$ denotes the projection from $\fg$ to $V_j$.   
To reflect the action of $\delta_r$ on $\fg$ we introduce the `homogeneous norm'
\begin{equation}  \label{eq:homogeneous_norm}
|X| := \left(  \sum_{i=1}^m  | \pi_i  X|_e^{2m!/i}\right)^{1/ 2m!}.  \quad 
\end{equation}
Then
\begin{equation}  \label{eq:scaling_homogeneous_norm}
|\delta_r X| = r |X|.
\end{equation}
Note that $| X|$ is comparable to $\sum_{j=1}^m  |\pi_j X|_e^{1/j}$ and that $| \cdot|$ does not satisfy the triangle inequality
but only the weaker estimate $| X + Y| \le C |X| + C |Y|$.
It follows from the  ball-box theorem, see e.g.  \cite[Theorem 2.4.2]{montgomery_book},   that there exists constant $C_1$ and $C_2$ such that
\begin{equation} \label{eq:ball_box}
   C_1 d_{CC}(e, \exp X) \le     |X| \le C_2 d_{CC}(e, \exp X)  \quad \forall X \in \fg.
\end{equation}
In fact, in Carnot groups the ball-box theorem follows immediately from the seemingly weaker statement that
the Riemannian distance $d$ and $d_{CC}$ induce the same topology on $G$. Indeed, together with the fact
that $\exp$ is a homeomorphism from $\R^n$ to $G$ equipped with $d$ this implies that the set
$S := \{ X \in \fg : d_{CC}(e, \exp X) = 1\}$ is compact. Thus $|X|$ attains its minimum and maximum on $S$ and
the inequality  \eqref{eq:ball_box} follows since all terms scale by $r$ if we replace $X$ by $\delta_r X$. 

Note also that 
\begin{equation}  \label{eq:norm_euclidean}
   C^{-1} \sum_{j=1}^m |\pi_j X|_e^2       \le   |X|_e^2  \le C \sum_{j=1}^m |\pi_j X|_e^2 \quad \forall X \in \fg
\end{equation}
since all norms on a finite-dimensional vector space are equivalent. 
In fact we can take $C=1$ if we choose a scalar product on $\fg$ such that the subspaces $V_j$ 
are orthogonal.

One of our main goals  is to construct a center of mass for probability measures    $\nu$ on $G$
which is invariant under left-translation and group homomorphisms. Equivalently, we want to construct
a center of mass for probability measures on $\fg$ which is compatible with the group action $\ast$. 
Since we want to allow measures which do not have compact support but only satisfy suitable moment bounds
we need good control of  the nonlinear term $P(X,Y)$ in the BCH formula in terms of $|X|$ and $|Y|$. 
To write the estimate we use the following notation for a multiindex
 $J = {j_1, \ldots, j_k}$ with $k \ge 1$ and $j_i \in \N \setminus \{0\}$.  We set
$\# J= k$ and $| J| = \sum_{i=1}^k j_i$. 

\begin{proposition}   \label{pr:bounds_P}
 Let $E_1,   \ldots, E_N$ be a basis of $\fg$. There exist multilinear forms $M_I$ and $L_J$ such that
\begin{align}  \label{eq:decomposition_P}
 & \qquad  P(X,Y) = \\
&   \sum_{i=1}^N     \sum_ { \underset{ j \ge 1, k \ge 1}{j, k, j+k  \le m}  }
\sum_{  \underset{ \# I =j, \# J = k}{I,J} }
 M^i_I(\pi_{i_1} X,  \ldots, \pi_{i_j}  X)   \, L^i_J(\pi_{i_1} Y,   \ldots, \pi_{i_k} Y)  \, E_i
 \nonumber 
\end{align}
and
\begin{equation}  \label{eq:bound_P_multilinear}
 M^i_I(\pi_{i_1} X,  \ldots, \pi_{i_j} X)  \le C  |X|^{| I |},  \quad 
  L^i_J(\pi_{i_1} Y,   \ldots, \pi_{i_k} Y)  \le C  |Y|^{| J |}.
\end{equation}
In particular
\begin{equation}     \label{eq:bound_P}
|P(X,Y)|_e  \le C(R)   (1 + |Y|^{m-1})  \quad \text{for all $X$ with $|X| \le R$}
\end{equation}
and the derivatives of $P$ with respect to the first variable satisfy
\begin{align}    \label{eq:bound_derivatives_P}
& |D^k_1 P(X,Y) (\dot X, \ldots, \dot X)| \le C(R) (1 + |Y|^{m-1})    \\
 & \quad \text{for all $X, \dot X$ with $|X| \le R$ and $|\dot X| \le 1$.}  \nonumber 
\end{align}
for $1 \le k \le m-1$ and $D^m_1 P = 0$. 
Moreover
\begin{equation} \label{eq:bound_P_homogeneous_norm}
| [X, Y]| \le C (|X| + |Y|)  \quad \text{and} \quad  |P(X,Y)| \le C (|X| + |Y|).
\end{equation}
\end{proposition}

\begin{proof} Since $X = \sum_i  \pi_i X$ and $Y = \sum_i \pi_i Y$ and since $P(X,Y)$ is a multilinear expression 
in $X$ and $Y$ it is clear that $P$ can be expanded into sums of products of multilinear terms as in   \eqref{eq:decomposition_P}. To show 
\eqref{eq:decomposition_P}
 we only have to show that the terms corresponding to $ | I | + |J | > m$ vanish.
 This follows immediately  from the fact that $P(X,Y)$ is a sum of iterated Lie brackets and 
 the first inclusion in   \eqref{eq:bracket_grading}

  Since $M^i_I$ is a multilinear form it follows that
 $M^i_I(\pi_{i_1} X,  \ldots, \pi_{i_j} X)  \le C \prod_{k=1}^j  |\pi_{i_k} X|_e$.
 Now by the definition of the homogeneous norm we have $   |\pi_{i_k} X|_e  \le |X|^{i_k}$.
 This implies the estimate for $M^i_I$ and the same argument applies to $L^i_J$. 
  The estimate  \eqref{eq:bound_P} is an immediate consequence of 
  \eqref{eq:decomposition_P},  \eqref{eq:bound_P_multilinear}, the  fact that 
  only terms with $|J| \le m-1$ appear in \eqref{eq:decomposition_P} 
  and  the estimate $a^k \le 1 + a^{m-1}$ for $1 \le k \le m-1$. 
  Since  the terms $M_I^i$ are multilinear, their derivatives are uniformly bounded for  $|X| \le R$
  and thus   \eqref{eq:bound_derivatives_P} follows in the same way.
  
The second estimate in  \eqref{eq:bound_P_homogeneous_norm} follows from the first  since $P(X,Y)$ is a linear combination of
  iterated Lie brackets. To show the first inequality, assume first that $X \in V_j$, $Y \in V_k$. Then $[X, Y] \in V_{j+k}$.
  Thus
  $$|[X, Y]| = |[X,Y]|_e^{\frac1{k+j}} \le C |X|_e^{\frac1{k+j}} \,   |Y|_e^{\frac1{k+j}} \le  C |X|^{\frac{j}{k+j}} \,   |Y|^{\frac{k}{k+j}}$$
  and the estimate follows by Young's inequality. For general $X, Y$ the estimate follows by bilinearity of the Lie bracket.
\end{proof}

\begin{remark}   \label{re:BCH_bounds_nilpotent}
The estimates   \eqref{eq:bound_P_multilinear}, \eqref{eq:bound_P} and \eqref{eq:bound_derivatives_P}
also  hold for nilpotent groups if $\pi_i$ denotes the projection
to the spaces $W_i$ and the homogeneous norm is defined with this definition of $\pi_i$. 
Indeed, we can use the second inclusion in   \eqref{eq:bracket_grading} to see that also in the nilpotent
  case  the sum in  \eqref{eq:decomposition_P} only contains terms
with $|I| + |J| \le  m$. The rest of the argument is the same.  
 Instead of \eqref{eq:bound_P_homogeneous_norm}
we have the slightly weaker estimates
\begin{equation} \label{eq:bound_P_nilpotent_homogeneous_norm}
| [X, Y]| \le C (1 + |X| + |Y|)  \quad \text{and} \quad  |P(X,Y)| \le C ( 1+ |X| + |Y|).
\end{equation}
Again the second estimate follows from the first. For the first estimate we first consider $X \in W_j$, $Y \in W_k$.
Then $[X, Y] \in  \oplus_{i=j+k}^m W_i$    
 and thus by Young's inequality and the previous estimate for 
$|[X,Y]|_e^{\frac1{j+k}}$
$$ | [X, Y] | \le \sum_{i=j+k}^m  |\pi_i [X,Y]|_e^{\frac1i}  \le  C(1 + |[X,Y]|_e^{\frac1{j+k}}) \le C (1 + |X| + |Y|).$$
For general $X$, $Y$ the estimate follows by bilinearity. 
\end{remark}

\bigskip\bigskip
\begin{remark}  \label{re:norms_and_complement}
Note that in the nilpotent case the homogeneous norm does not just depend on the group and the
scalar product on $\fg$, but also on the choice of the complementing spaces $W_1, \ldots, W_m$. 
Different choices lead, however, to  norms which are essentially equivalent in the following sense. Let $\tilde W_i$ be different spaces with   
$\fg_i = \tilde W_i \oplus \fg_{i+1}$, let $\tilde \pi_i$ be the corresponding projections
and let $|\cdot |_{\sim}$ be the corresponding homogeneous norm.
Then there exists a constant $C$ such that
\begin{equation}  \label{eq:equivalent_norms_nilpotent}
|X|_{\sim} \le C(  |X|^{\frac1m} +   |X| )   \quad   \text{and} \quad |X| \le C(  |X|_{\sim}^{\frac1m} +   |X|_{\sim} ). 
\end{equation}
It suffices to prove the first inequality, the second follows by reversing  the roles of $W_i$ and $\tilde W_i$. 
We have $\tilde \pi_i  |_{\fg_{i+1}} = 0$.  Since $\pi_{i+1} X, \ldots, \pi_{m} X \in \fg_{i+1}$ there exist linear maps
$L_i : \oplus_{k=1}^i W_k  \to \tilde W_i$ 
   such that
$ \tilde \pi_i X = L_i(\pi_1 X,  \ldots, \pi_i X)$.
Thus 
$$ | \tilde \pi_i X|_e  \le  C \sum_{k=1}^i |\pi_k X|_e  \le  C  \sum_{k=1}^i |X|^k$$
and 
hence 
$$ |X|_\sim \le C \sum_{i=1}^m   |\tilde \pi_i X|_e^{\frac1i}  \le C \sum_{i=1}^m   |X|^{\frac1i}.$$
  From this the assertion easily follows by Young's inequality.
\end{remark}

\bigskip\bigskip
\subsection{Differential forms on Carnot groups}\label{subsec_differential_forms_on_carnot_groups}
Let $G$ be a Carnot group with graded Lie algebra $\fg=\oplus_i V_i$.  The grading defines a simultaneous eigenspace decomposition for the dilations $\{\de_r\}_{r\in (0,\infty)}$.  Therefore the action of $\{\de_r\}_{r\in (0,\infty)}$ on $\La^k\fg$ also has a simultaneous eigenspace decomposition 
\begin{equation}
\label{eqn_weight_eigenspace_decomposition}
\La^k\fg=\oplus_w \La^{k,w}\fg
\end{equation}
 where $\de_r$ acts on $\La^{k,w}\fg$ by scalar multiplication by $r^w$.  In particular, for any  $\al\in\La^k\fg$, we have a canonical decomposition
\begin{equation}
\label{eqn_weight_decomposition}
\al=\sum_w\al_w
\end{equation}
where $(\de_r)_*\al_w=r^w\al_w$ for every $w$.  Concretely, if $X_1,\ldots,X_N$ is a graded basis of $\fg$,  and  $\th_1,\ldots,\th_N$ is the dual basis, then the actions of $\de_r$ on $\fg$ and $\fg^*$ are diagonal with respect to these bases, and the action on $\La^k\fg$ is diagonal with respect to the basis given by exterior powers of the $\th_i$s.

\begin{definition}
An element  $\al\in \La^k(\fg)$  is {\bf homogeneous with weight $w$}
if $\al\in\La^{k,w}\fg$; it has {\bf weight $\leq w$} if $\al\in \La^{k,\leq w}$ where
\begin{equation}
\label{eqn_def_wt_leq_w}
\La^{k,\leq w}:=\oplus_{\bar w\leq w}\La^{k,\bar w}\,.
\end{equation}
If $U\subset G$ is open, then a $k$-form $\al\in \Om^k(U)$  is {\bf homogeneous of weight $w$} or has {\bf weight $\leq w$} if $\om(x)\in \La^{k,w}$  or $\om(x)\in \La^{k,\leq w}$, for every $x\in U$, respectively.  We let $\Om^{k,w}(U)$ and $\Om^{k,\leq w}$ denote the homogeneous forms of weight $w$ and the forms of weight $\leq w$, respectively, so  $\Om^k(U)=\oplus_w\Om^{k,w}(U)$.  Note that $0\in\La^k(\fg)$ has weight $w$ for every $w\in \R$.  
\end{definition}

\begin{lemma}
\label{lem_weight_facts}

\mbox{}
\ben
\item If $\al_i\in\Om^{k_i,w_i}$ for $1\leq i\leq 2$, then $\al_1\we \al_2\in\Om^{k_1+k_2,w_1+w_2}$.  
\item $\th_{i_1}\wedge\ldots\wedge\th_{i_k}$  is homogeneous of weight $\sum_j\weight(\th_{i_j})$.   In particular, such wedge products give a basis for the left invariant 
$k$-forms. 
\item  \label{it:weight_facts_pullback} 
If $\be\in \Om^{k,w}(G')$ and
$\Phi:G\ra G'$ is a graded group homomorphism, then $\Phi^*\be=\Phi_P^*\be$ belongs to $\Om^{k,w}(G)$.
\een
\end{lemma}
\begin{proof}
(1) is immediate from the definitions, and (2) follows from (1).

(3). Since $\Phi$ is a graded group homomorphism, the Pansu derivative
and the ordinary derivative coincide.  Therefore
$$
(\de_r)_*\Phi_P^*\be= (\de_{r^{-1}})^*\Phi^*\be
=\Phi^*(\de_{r^{-1}})^*\be=\Phi^*(\de_r)_*\be
=r^{w_\be}\Phi_P^*\be\,.
$$ 

\end{proof}

\bigskip\bigskip
\subsection{Sobolev spaces on Carnot groups}  \label{subsec_sobolev}
In this subsection we discuss $L^p$ and Sobolev spaces for maps between Carnot
groups; in addition to the definitions, we cover two key properties needed for the proof of the approximation theorem -- the Poincar\'e-Sobolev
inequality and  almost everywhere  Pansu differentiability (in an $L^p$ sense).  
In the literature there are  different approaches to Sobolev mappings between Carnot groups -- some are based on distribution derivatives, and  others on (weak) upper gradients (see \cite{vodopyanov_monotone_1996,vodopyanov_bounded_distortion,HKST}).    In this subsection we use the distributional definition of Sobolev mappings, and cover the upper gradient version in Appendix~\ref{se:sobolev}. 
In fact, the two definitions are equivalent in our setting (see Appendix~\ref{se:sobolev}),  so one could work equally well work with either.

We treat the case where the domain is an open set in a Carnot group; however most statements and proofs apply without modification to equiregular subriemannian manifolds satisfying a suitable Poincar\'e inequality.
 
In this subsection we let $G$ denote a Carnot group with graded Lie algebra $\fg = \oplus_{i=1}^m V_i$. Let $X_1, \ldots, X_K$ be an orthonormal  basis of the  first layer  $V_1$.
As usual we identify the vectors   $X_i \in V_1$ with  left-invariant vectorfields on $G$. Then $X_1(p), \ldots, X_K(p)$ is a 
basis of the horizontal subspace at $p$.

\begin{definition}   \label{de:sobolev_scalar_new} Let $U \subset G$ be open.  We say that  $u: U \to \R$  is in  the Sobolev space $W^{1,p}(U)$
if $u \in L^p(U)$ and  
if the distributional derivatives  $X_1 u, \ldots X_K u$ are in $L^p(U)$, i.e., if there exist $g_i \in L^p(U)$ such that
$$ \int_U  u   \,  X_i \varphi \, d\mu =  - \int_U g_i  \,  \varphi \, d\mu \quad     \text{for all $\varphi \in C_c^\infty(U)$.}
 $$
We say that $u \in W^{1,p}_{loc}(U)$ if $u \in W^{1,p}(V)$ for every open set $V$ whose closure is compact and contained in $U$. 
\end{definition}

\bigskip
We write $X_i u$ for the weak derivatives $g_i$ and  we define
\begin{equation}
D_h u = (X_1 u, \ldots, X_K u), \quad  |D_h u| := \left( \sum_{i=1}^K |X_i f|^2\right)^{1/2}.
\end{equation}

We recall some basic properties of Sobolev functions.
\begin{proposition}  \label{pr:sobolev_basic} Let $U \subset G$ be open. Then the following assertions hold.
\ben
\item  \label{it:basic_density}  $C^\infty(U)$ is dense in $W^{1,p}(U)$;
\item  \label{it:basic_chain}    if $u\in W^{1,p}(U)$ and $\psi: \R \to \R$ is $C^1$ with bounded derivative then $\psi \circ u -  \psi(0) \in  W^{1,p}(U)$ and 
the weak derivatives satisfy the chain rule;
\item  \label{it:basic_absolute}   if $u \in W^{1,p}(U)$ then $|u| \in W^{1,p}(U)$ and the weak derivatives satisfy $X_i |u| = \pm X_i u$ a.e. in the set $\{ \pm u > 0\}$
while $X_i |u| = 0$ a.e.\ in the set $\{u = 0\}$;
\item  \label{it:basic_min}  if $u,v \in W^{1,p}(U)$ then $\min(u,v) \in W^{1,p}(U)$ and $|D_h \min(u,v)| \le \max(|D_h u|, |D_h v|)$ a.e.;
\item    \label{it:basic_inf}    if $u_k \in W^{1,p}(U)$ for $k \in \N$ and there exist functions $g, h \in L^p(U)$ such that $ |D_h u_k| \le g$ a.e. and 
$u_k \ge h$ a.e., for all $k \in \N$ then $\underline u := \inf_k u_k \in W^{1,p}(U)$ and $|D_h  \underline u| \le g$ a.e.
\een
\end{proposition}

\begin{proof} For   assertion~\eqref{it:basic_density}  see
 Friedrichs \cite{friedrichs_1944} or  Thm.\  1.13 and Thm.\  A.2 in 
   \cite{garofalo_nhieu_1996}.  Assertions~\eqref{it:basic_chain} and   \eqref{it:basic_absolute}  follow from   \eqref{it:basic_density} 
   in the same way as in the Euclidean case (see, for example   \cite[Sec. 7.4]{gilbarg_trudinger} for the Euclidean setting). Indeed, for    \eqref{it:basic_absolute}
   one applies   \eqref{it:basic_chain} with $\psi_\eps(t)  = \sqrt{t^2 + \eps^2} - \eps$ and takes the limit $\eps \to 0$. 
   Assertion~\eqref{it:basic_min} follows from  \eqref{it:basic_absolute} since $\min(u,v) = \frac12( u+v) - \frac12 |u-v|$. To prove
   assertion~\eqref{it:basic_inf} set $w_k = \inf_{j \le k} u_j$. It follows from \eqref{it:basic_min} that $w_k \in W^{1,p}(U)$ and $|D_h w_k| \le g$. 
   Moreover $k \mapsto w_k$ is non-increasing. Since  $w_k \ge h$ and $h \in L^p(U)$, the monotone convergence theorem implies that 
$w_k \to \underline u$ in $L^p(U)$. Moreover a subsequence of the weak derivatives $X_i w_k$ converges weakly in $L^p(U)$ to a limit
$h_i$ (for $p=1$ use the Dunford-Pettis theorem). Thus $\underline  u$ is weakly differentiable with weak derivatives $h_i$. By weak lower semicontinuity
of the norm we deduce that $(\sum_i  h_i^2)^{1/2} \le |g|$.
\end{proof}

\bigskip
We now consider spaces of $L^p$ functions and Sobolev functions with values in a metric space. We will later only consider  a Carnot group $G'$
with the Carnot-Caratheodory  metric as the target space, but we  state the results for general targets to emphasize that they do 
not use the structure of a Carnot group.
The following definition is due to Reshetnyak \cite{reshetnyak_1997} for open subsets of $\R^n$ or a Riemannian manifold
and has been extended
by Vodopyanov  \cite[Proposition 3, p.\ 674] {vodopyanov_bounded_distortion} to the setting of Carnot groups. 

\begin{definition}   \label{de:sobolev_carnot_new} 
Let $(X', d')$ be a complete separable metric space and let $U \subset G$ be open.
\ben
\item We say that a map $f: U \to X'$ is in $L^p(U,X')$ if $f$ is measurable and if there exist an $a \in X'$ such that 
 the map $x \mapsto d(f(x), a)$ is in $L^p(U)$.
\item We say that $f \in L^p(U,X')$ is in 
the Sobolev space $W^{1,p}(U;X')$
if for all $z \in X'$ the functions  $u_z(\cdot) := d'(f(\cdot), z) - d'(a,z)$ are in $W^{1,p}(U)$
and if there exists a function $g \in L^p(U)$ such that
\begin{equation}   \label{eq:bound_D_h_metric} |D_h u_z| \le g  
\end{equation}
almost everywhere.
\een
The spaces $L^p_{loc}(U;X')$ and $W^{1,p}_{loc}(X')$ are defined as usual. 
\end{definition}

Note that by the triangle inequality the map  $x \mapsto d'(f(x), z)$ is in $L^p_{loc}(U)$ for all $z \in X'$
if it is in $L^p_{loc}(U)$ for one $z \in X'$; if $\mu(U) < \infty$ then the same assertion holds for $L^p(U)$.   Note however that the  assertion fails for $L^p(U)$ when $\mu(U)=\infty$.

Definition~\ref{de:sobolev_carnot_new} imposes estimates on the weak derivatives of $f$ composed with  the distance functions $d(z, \cdot)$.  These imply similar estimates on the composition with general Lipschitz functions from $G'$ to a finite-dimensional linear space:

\begin{proposition}  \label{pr:composition_by_lip}
 Let $(X', d')$ be a complete separable metric space and let $U \subset G$ be open.
Let $f \in W^{1,p}(U; G')$ and let $a \in G'$ be such that $x \mapsto d'(a, f(x))$ is in $L^p(U)$. 
Let $Y$ be a finite-dimensional  $\R$-vector space and $v : G' \to Y$ be Lipschitz. Then $v \circ f - v(a)
\in W^{1,p}(U;Y)$. 
\end{proposition}

\begin{proof} 
It suffices to show the assertion for $Y= \R$. So let $u :X' \to \R$ be $L$-Lipschitz and let $D \subset X'$ be a countable dense set. 
Then $v \circ f$ in $L^p(U)$ and 
$$ v(z) = \inf_{z' \in D}  v(z') + d'(z', z).$$
Thus 
the assertion follows from the definition of $W^{1,p}(U;X')$ and  Proposition~\ref{pr:sobolev_basic}~\eqref{it:basic_inf}.
\end{proof}

\bigskip
We will use the following version of the Lebesgue point theorem. Here and elsewhere in the paper we use the standard notation $\av$ to denote an average.

\begin{lemma}  \label{le:lebesgue_point_new} 
  Let $f \in L^p_{loc}(U, X')$. Then for a.e. $x \in U$ we have
\begin{equation}  \label{eq:p_lebesgue_point_new}
\lim_{r \to 0}  
\av_{B(x,r)} [ d'(f(y), f(x))]^p \, d\mu(y) = 0.
\end{equation}
\end{lemma} 

\begin{proof}This follows easily by applying the usual  Lebesgue point theorem to the scalar functions
$u_z(y) = d'(f(y), z)$ where $z$ runs through a countable dense subset of $X'$.\end{proof}

 To recall the Poincar\'e-Sobolev inequality for metric-space-valued maps we define the $L^p$-oscillation as follows.
 
\begin{definition}  Let $X'$ be a metric space. Let $A \subset G$ be a measurable set and let $f \in L^p(A,X')$. 
The $L^p$ oscillation on $A$ is defined by 
\begin{equation}
 \osc_p(f,  A)  := \inf_{a \in X'}  \left(   \int_{A} {d'}^p(f(x), a)  d\mu(x)   \right)^{1/p}.
 \end{equation}
\end{definition}

\bigskip
There is a general strategy for deducing the  Poincar\'e-Sobolev inequality for metric-space-valued maps
from the  Poincar\'e-Sobolev inequality for scalar valued functions. It is based on the derivation of a pointwise estimate for a.e.\ pair of points
and a chaining argument, see Theorem 9.1.15 in  \cite{HKST}  for an implementation of this approach in the context of the upper gradient definition.
Since we are only interested in Carnot groups as targets we use a more pedestrian approach.
Recall that a metric space is doubling
if every ball of radius $r >0$  can be covered by a fixed  number $M$
of balls of radius $\frac{r}{2}$. Carnot groups are doubling since  by compactness $B(e,1)$ can be covered  by $M$   balls
of radius $\frac12$. By translation and scaling every ball of radius $r$ can be covered by $M$ balls of radius $\frac{r}{2}$.

\begin{theorem} Let $G$ be a Carnot group of homogeneous dimension $\nu$ and  let $X'$ be a complete, separable metric space which is  doubling. Let $1 \le p < \nu$ and define $p^*$ by
$$\frac{1}{p*}  = \frac1p - \frac1\nu.$$ There exists a constant $C = C(G,p)$ with the following property.
Let  $B(x,r)$ be a ball in $G$,  let 
$f \in W^{1,p}(B(x,r),X')$ and let $g \in L^p(B(x,r))$  be    the function in Definition~\ref{de:sobolev_carnot_new}. Then
\begin{equation}  \label{eq:sobolev_poincare_new}
 \osc_{p^*}(f,   B(x,r)) \leq C \|g\|_{L^p(B(x,r))}
\end{equation}
and
\begin{equation} \label{eq:p_poincare_new}
 \osc_{p}(  f, B(x,r)) \le C r  \|g\|_{L^p(B(x,r))}.
\end{equation}
\end{theorem}

\begin{proof} If suffices to  prove the first estimate,  since the second follows from the first by H\"older's inequality.
For the Poincar\'e-Sobolev inequality for scalar functions see  \cite[Corollary 1.6.]{garofalo_nhieu_1996} or  \cite[Theorem 2.1]{lu_poincare_1994}.
By scaling and translation in $G$ it suffices to show the Poincar\'e inequality for  $X'$-valued maps
 for the set $B = B(e,1)$ and we may assume that the Haar measure $\mu$ is normalized so that
$\mu(B) = 1$. 

Let $f \in W^{1,p}(B, X')$ and for $z \in X'$ define $u_z(x) = d'(z, f(x))$.
By Definition~\ref{de:sobolev_carnot_new}  and the Poincar\'e inequality for 
scalar functions we see  that $u_z \in L^{p*}(B)$ for all $z \in D$. Thus  $ L:= \osc_{p*}( f, B) = \inf_z \|u_z\|_{p^*,B} < \infty$. 
Let $\bar a \in X'$ be such that the infimum is achieved.
Then 
$$ \mu \{ x \in B:  f(x) \in B(\bar a, 2L) \} \ge (1- 2^{-p^*}) \mu(B) \ge \frac12.$$
Since $X'$ is doubling there exist $M^2$ balls of radius $\frac{L}{2}$ which cover $B(\bar  a, 2L)$. Thus there exist $z \in X'$ such  that  
$$ \mu(E)  \ge \frac12 M^{-2} \quad \text{where $E= f^{-1}(B(z, \frac{L}{2}))$}.$$
By the Sobolev-Poincar\'e inequality for scalar-valued
functions and the triangle inequality we have 
\begin{equation} \label{eq:poincare_double}  \int_B \int_B  |u_z(x) - u_z(y) |^{p^*}   \, d\mu(x) \, d\mu(y)  
  \le C   \|g \|_{p, B}^{p^*} \,.
  \end{equation}
 
   Let $$v(x) =   \max( u_z(x) - \frac{L}{2}, 0).$$  For $y \in E$ we have $u_z(y) \le \frac{L}{2}$ and hence   $v(x) \le |u_z(x) - u_z(y)|$.  
Restricting the outer  integral on the left hand side of  \eqref{eq:poincare_double} to the set $E$ we  get
$$   \int_B  v^{p^*}   \, d\mu     \le 2  C M^2      \|g \|_{p, B}^{p^*}.$$
   Thus $\| v \|_{p^*, B} \le  C(M, p) \| g\|_{p, B}$.
By the definition  of the oscillation and the definition of $v$   we have  
$$  \| v \|_{p^*,B} \ge  \| u_z\|_{p^*, B} - \frac{L}{2}  \ge \frac{L}{2}.$$
Hence $ \osc_{p^*}(f, B) = L \le 2 C(M,p)   \| g\|_{p, B}$. 
\end{proof}

\bigskip\bigskip
We finally discuss differentiability results. It is well known that locally Lipschitz maps from $\R^n$ to $\R^m$
are differentiable a.e. In fact  maps
in  $f \in W^{1,p}_{loc}(\R^n; \R^m)$  are differentiable a.e.\ if $p> n$.  For $1 \le p < n$ the map  $f$ is differentiable in an $W^{1,p}$ sense. 
For maps between Carnot groups, Pansu \cite{pansu} showed that Lipschitz maps (with respect to the Carnot-Caratheodory metrics) on open sets
are a.e. differentiable in the following sense, now known as Pansu differentiability. For a.e. $x  \in U$
there exist a  graded group homomorphism $\Phi: G \to G'$ such that 
the rescaled maps
\begin{equation}  \label{eq:pansu_rescaling_new}
 f_{x,r} := \delta_{r^{-1}} \circ \ell_{f(x)^{-1}} \circ f  \circ \ell_x \circ \delta_r.
\end{equation}
converge locally uniformly to as $r\to 0$.
The corresponding Pansu differentiability results for Sobolev maps between Carnot groups
have been obtained by Vodopyanov \cite[Theorems 1 and 2, Corollaries 1 and 2]{vodopyanov_differentiability_2003},  see also \cite{vodopyanov_carnot_manifolds} for extensions to 
maps between Carnot manifolds.
We will use the following result.

\begin{theorem}[$L^{p*}$ Pansu differentiability a.e.,  \cite{vodopyanov_differentiability_2003}, Corollary 2]
\label{th:Lp*_pansu_differentiability_new}
 Let $ U \subset G$ be open, let $1 \le p < \nu$ 
and define $p^*$ by $\frac1{p^*} = \frac1p - \frac1\nu$.
Let $f \in W^{1,p}(U;G')$. For $x \in U$ consider the rescaled maps
$$ f_{x,r} = \delta_{r^{-1}} \circ \ell_{f(x)^{-1}} \circ f \circ \ell_x \circ \delta_r.$$
Then, for a.e. $x \in U$,  there exists  a group homomorphism $\Phi: G \to G'$ such that
\begin{equation}  \label{eq:Lpstart_pansu_differentiability_new}
 f_{x,r} \to \Phi  \quad \text{in $L^{p^*}_{loc}(G;G')$ as $r \to 0$.} 
 \end{equation}
 \end{theorem}

 \begin{remark}  \label{re:g_equal_norm_DPf}
  It follows easily from  Theorem~\ref{th:pansu_diff_Lpstar}  that for all $z \in G'$ 
 the functions $u_z:= d'(z, f(\cdot))$ satisfy 
 \begin{equation}  \label{eq:optimal_g}
 \text{$|D_h u_z|  \le |D_P f|$  a.e., where} 
  \end{equation}
 \begin{equation}  
 |D_P f(x)|  = \max \{ |D_P f(x) X|_{V'_1} : \, X \in V_1, \, |X|_{V_1} \le 1 \}.
 \end{equation}
Here  $| \cdot |_{V_1}$ and $| \cdot |_{V'_1}$ denote the norms induced by the scalar product on the first layer of $\fg$ and $\fg'$,
respectively. Thus the condition   \eqref{eq:bound_D_h_metric}  
in Definition~\ref{de:sobolev_carnot_new} holds with $g = |D_P f|$. 
The short proof  of   \eqref{eq:optimal_g}can be found in Appendix~\ref{se:W1p_differentiability}. 
 \end{remark}

\bigskip
 \begin{remark} \label{re:weak_derivative_abel}
   Let $\pi_{G'} : G' \to G'/[G',G']$ denote the abelianization map. Since $\pi_{G'}$ is globally Lipschitz,
 the map $\pi_{G'} \circ f$ is in $W^{1,p}(U, G'/[G',G'])$. It easily follows from 
 Theorem~\ref{th:Lp*_pansu_differentiability_new} that the weak (or distributional) derivative of 
 $\pi_{G'} \circ f$ in direction of a horizontal vectorfield $X$ satisfies
 \begin{equation} \label{eq:weak_derivative_abel}
 X (\pi_{G'} \circ f)(x) = D_P f(x) X \quad \text{for a.e. $x$,}
  \end{equation}
  see Remark~\ref{rem_step_2_of_proof} below.
 In  \eqref{eq:weak_derivative_abel} we have identified the abelian group $G'/[G',G']$ with the first layer $V'_1$ of $\fg'$.
 \end{remark}

\bigskip

 Vodopyanov's proof  in \cite{vodopyanov_differentiability_2003} combines work from a series of earlier papers \cite{vodopyanov_ukhlov_approximate_differentiable,vodopyanov_monotone_1996,vodopyanov_bounded_distortion,vodopyanov_P_differentiability}.  His argument is based on Lipschitz approximation on sets of almost full measure, 
 an extension of Pansu's result to Lipschitz maps defined on sets $E \subset G$
  which are not open and a careful estimate of the remainder terms at points in $E$ of  density one.
  
  In Appendix~\ref{se:W1p_differentiability} we give a direct alternative  proof of $L^p$ differentiability which 
   is based on blow-up, the Poincar\'e-Sobolev inequality,  the compact Sobolev embedding  (which is an immediate consequence of the Poincar\'e-Sobolev inequality)
  and the following observation:  if  $F : G \to G'$  is a Lipschitz map with $F(e) = e$ and the abelianization $\pi_{G'} \circ F$ is affine
  (i.e. has constant weak  horizontal derivatives) then $F$  is a  graded group homomorphism.

\section{Center of mass and  mollification}
\label{sec_center_of_mass_mollification}

In this section we first define a center of mass for measures on a Carnot group which satisfy  a moment condition.  We then use this to define a mollification procedure for maps with finite  $L^m$-oscillation taking values in an $m$-step Carnot group. 

We will be using the notation and results from Section~\ref{sec_prelim}, in particular the `homogeneous norm'  $|\cdot|$ on $\fg$ defined in (\ref{eq:homogeneous_norm}), and the Euclidean norm $|\cdot|_e$.  

\subsection{Center of mass in Carnot groups}
\label{subsec_center_of_mass}
Let $G$ be a  $m$-step Carnot group (for an extension to connected, simple connected nilpotent Lie groups see 
Remark~\ref{re:com_nilpotent}  and Remark~\ref{re:naturality_com_nilpotent} below).
Let $\nu$ be a Borel probability measure on $G$. We say that $\nu$ has finite $p$-th moment if  
\begin{equation} \label{eq:p_th_moment_nu}
\int_G  d^p_{CC}(e, y) \,  d\nu(y) < \infty.
\end{equation}

In view of   \eqref{eq:ball_box} this is  equivalent to

\begin{equation} 
\label{eq:p_th_moment_nu_bis}
\int_\fg  |Y|^p   \,  d(\log_*\nu)(Y) < \infty.
\end{equation}
In this subsection we define,  for probability measures with finite $m$-th moment,  a center of mass 
which is compatible with left translation and group homomorphisms.
For probability measures $\nu$ with compact support our notion of center of mass agrees with the one by Buser and Karcher \cite[Example 8.1.8] {karcher_buser_almost_flat_manifolds}. 
Their proof of the existence of the center of mass is different. 
They use the bi-invariant flat connection $D$ such that left-invariant vector fields are $D$-parallel, and base their proof on some estimates for the convexity radius of $D$ with respect to some auxiliary left invariant Riemannian metric.
Here we argue directly on the Lie algebra
 and also show that there is an explicit recursive formula for (the logarithm of) the center of mass
and that the logarithm of the center of mass is a polynomial in certain polynomial moments of 
$\log_* \nu$. 

The extension from compactly supported measures to measures with finite $m$-th moment
will  be crucial in the next subsection where we use the center of mass to define a group compatible
mollification for (Sobolev) functions which may be unbounded.

We define
\begin{equation} \label{eq:log_x}
\log_x = \log \circ \ell_{x^{-1}}
\end{equation}

\begin{theorem}  \label{le:C_nu_diffeomorphism}
 Let $G$ be an $m$-step Carnot group and let $\nu$ be a Borel probability measure on $G$
with finite $m$-th moment. Then $\log_x$ is $\nu$ integrable and the map $C_\nu: G \to \fg$ defined
by 
\begin{equation}
C_\nu(x) := \int_G  \log_x \, d\nu     
\end{equation}
is a diffeomorphism. Moreover $C_\nu \circ \exp: \fg \to \fg$ is a polynomial of degree not larger than $m-1$. 

For any $Z\in \fg$, the  equation $C_\nu(\exp X) = Z$ can be solved recursively 
and  $\log (C_\nu)^{-1}(Z)$ is a polynomial in $Z$ and certain  polynomial moments of $\log_* \nu$. 
In particular   there exist a  $\fg$-valued polynomial $Q$ 
 with $Q(0, \ldots, 0) = 0$  and $\fg$-valued
multilinear forms   $L_1, \ldots L_K: \fg \to \fg$ such that
\begin{equation}   \label{eq:formula_com1}
\log (C_\nu)^{-1}(0)   = Q(A_1,  \ldots A_K),
\end{equation}
where 
\begin{equation}   \label{eq:formula_com2}
A_i = \int_{\fg}  L_i(Y,  \ldots, Y) \, d(\log_*\nu)(Y). 
\end{equation}
Moreover
\begin{equation}  \label{eq:bound_Qi}
|L_i(Y,  \ldots, Y)|_e \le C_i (1 + |Y|^m)    \quad \text{for $1 \le i \le K$.}
\end{equation}
\end{theorem}

 We call 
\begin{equation}  \label{eq:define_com}
\com_\nu := (C_\nu)^{-1}(0)
\end{equation}
the center of mass of $\nu$.

\bigskip\bigskip
\begin{remark}
The proof shows the condition  that $\nu$ has finite $m$-th moment can be slightly weakened. It suffices to assume that
\begin{equation}  \label{eq:weakened_moment_condition}
\int_{\fg}   (|Y|^{m-1} +  |Y|_e )  \, d\log_* \nu < \infty.
\end{equation}
The reason is that elements of $V_m$ do not appear in the BCH term $P(X,Y)$ and that $P(X,Y)$ 
is polynomial in $Y$ of degree
not exceeding $m-1$, see also   \eqref{eq:bound_P}.  Recall from  \eqref{eq:homogeneous_norm} that the homogeneous norm $|X|$ is equivalent to $\sum_{i=1}^m |\pi_j X|_e^{\frac1j}$ where $\pi_j : \fg \to V_j = \fg_j$
is the projection to the $j$-th layer of the algebra. Thus condition \eqref{eq:weakened_moment_condition} is equivalent to the condition that 
$|Y|^{m-1} + |\pi_m Y|_e$ is  $\log_* \nu$ integrable.
\end{remark}

\bigskip\bigskip
\begin{proof}[Proof of Theorem~\ref{le:C_nu_diffeomorphism}]  
It is easier to work on the algebra $\fg$ rather than the group $G$.
We thus define
$$ \tilde C_\nu(X) =  C_\nu(\exp X).$$
Then, using the BCH formula,  we get  
\begin{equation}
\label{eqn_c_nu_tilde_formula}
\begin{aligned}
 \tilde C_\nu(X) =  & \, \int_G  \log( \exp(-X) y)  \, d\nu(y)  \\
 = & \, \int_{\fg}   \log( \exp(-X) \exp Y) \,  d\log_*\nu(Y) \\
 = & \, - X +   \int_{\fg} ( Y + P(-X,Y))  \,  d\log_*\nu(Y).
 \end{aligned}
 \end{equation}
  The integrand is a polynomial  of degree at most  $m-1$ in $X$. 
 By the definition of the homogeneous norm we have $|Y|_e \le C( |Y| + |Y|^m)$.
 Moreover by   \eqref{eq:bound_P}  we have $|P(X,Y)| \le C(X) ( 1 + |Y|^{m-1})$. 
 Since $\nu$ is a probability measure with finite $m$-th moment the integral in (\ref{eqn_c_nu_tilde_formula}) exists.
 The bounds  \eqref{eq:bound_derivatives_P} on the derivatives imply that differentiation with respect
 to $X$ and integration commute. Thus $\tilde C_\nu$ is a polynomial of degree at most $m-1$.

We now show that for every $Z\in \fg$, the equation
 \begin{equation}  \label{eq:equation_inverse_Cnu} \tilde C_\nu(X) = Z
 \end{equation}
 has a unique solution which is a polynomial in $Z$, and moreover depends only on certain polynomial moments of $\log_* \nu$ of degree at most $m-1$.
 Recall that $\pi_i$  denotes  the projection from $\fg = \oplus_{i=j}^m V_j$ to $V_i$. 
 Define functions $P^i$ by $P^i(X,Y) = \pi_i  P(X,Y)$ and set $X^i = \pi_i X$, $Y^i = \pi_i Y$.
 Applying $\pi_i$ to   \eqref{eq:equation_inverse_Cnu} we get a system of $m$ equations, namely
 \begin{equation}  \label{eq:inverse_C_nu}
  - X^i +  \int_{\fg}  Y^i + P^i(-X, Y) d\log_*\nu(Y) = Z^i  \quad \text{for $i \in \{1, \ldots, m\}$.}
  \end{equation}
 Since $P$ consists of iterated commutators one easily sees that $P_i(X,Y)$ depends on $X$
 only through $(X^1, \ldots, X^{i-1})$. 
 Thus the system can be solved recursively starting with 
 $$ X^1 = -Z^1 + \int_{\fg} Y^1 \, d\log_*\nu(Y).$$
  Moreover the solution is polynomial in $Z$ and in particular  smooth.

 We finally discuss the  dependence of the solution on $\log_* \nu$. Since $[V_i, V_j] \subset V_{i+j}$ we see as
 in  Lemma~\ref{pr:bounds_P} that
  $P^i(X,Y)$ can be written as
\begin{align}  \label{eq:decomp_Pi}
 &   \, \quad  P^i(X,Y) \\
  =  & \sum_{p=1}^{d_i} \, \sum_{j=2}^{i}  \sum_{k=1}^{j-1}   \sum_{ \sum_{\ell = 1}^j i_\ell = i}
   M^p_{i_1 i_2 \ldots i_k}(X^{i_1},   \ldots    X^{i_k}) \,  \, 
   L^p_{i_{k+1} \ldots i_j}(Y^{i_{k+1}},    \ldots Y^{i_j})  E_p  \nonumber 
   \end{align}
   where $E_1, \ldots, E_{d_i}$ is a basis of $V_i$ and where 
   $M_{I}^p$ and $L_{I}^p$ are multilinear forms.

     Let $\bar X = \log \com_\nu$. Then it follows from  \eqref{eq:inverse_C_nu}
   that $\bar X$ can be recursively computed as
    \begin{equation}  \label{eq:formula_com}
  \bar X^i =   \int_{\fg} \Big( Y^i + P^i(-\bar X, Y)  \Big) \, d\log_*\nu(Y)   \quad \text{for $i \in \{1, \ldots, m\}$, }
  \end{equation}
 starting with $\bar X^1 = \int_{\fg} Y^1 \,  d\log_*\nu(Y)$.
  Hence $\bar X$ is a polynomial expression in    
  $$ \bar Y^i =    \int_{\fg}   Y^i   \, d\log_*\nu(Y)$$
  and 
  the polynomial  moments
  $$ A^i_{i_{k+1}, \ldots, i_j}   :=   \int_{\fg}   
   L^i_{i_{k+1} \ldots i_j}(Y^{i_{k+1}},    \ldots Y^{i_j})  \, d\log_*\nu(Y).
  $$
  Since  $s := \sum_{\ell = k+1}^j i_\ell \le i-k \le m-1$ we have by multilinearity
  \begin{align}   \label{eq:bound_Li_com}
 & \,  | L^i_{i_{k+1} \ldots i_j}(Y^{i_{k+1}},    \ldots Y^{i_j}) |_e 
 \le    C  \prod_{\ell = k+1}^j |Y^{i_\ell}|_e  \\
 \le  & \,   C  \prod_{\ell = k+1}^j |Y^{i_\ell}|^{i_\ell} 
 \le   C  |Y|^s  \le C (1 + |Y|^m).  \nonumber
 \end{align}
 Moreover $|Y^i|_e \le |Y|^i \le 1 + |Y|^m$.
 \end{proof}

\bigskip\bigskip
\begin{remark}  \label{re:com_nilpotent} The conclusion of Theorem~\ref{le:C_nu_diffeomorphism} continues to hold if we consider a nilpotent group instead of a Carnot group, and use (\ref{eq:p_th_moment_nu_bis}) rather than   (\ref{eq:p_th_moment_nu}) to define the $p$-th moment, where the homogenous norm is defined as
in 
Remark~\ref{re:BCH_bounds_nilpotent}  using a (noncanonical) decomposition $\fg = \oplus W_j$ with $ \fg_j= W_j \oplus \fg_{j+1}$ and denoting by $\pi_j$ the projection from $\fg$ to $W_j$.

Indeed, by Remark~\ref{re:BCH_bounds_nilpotent}  the bounds in Lemma~\ref{pr:bounds_P} also hold in the nilpotent case. Thus $C_\nu \circ \exp$ is well-defined
and a polynomial of degree at most $m-1$. In view of Remark \ref{re:norms_and_complement} the condition that the probability measure  $\log_* \nu$ has finite $m$-th moment
is independent of the choice of the auxiliary spaces $W_i$ since different choices lead to  homogeneous norms $|\cdot|_\sim$ and $|\cdot|$
which satisfy $|X|_\sim \le C( 1 + |X|^m)$ and $|X| \le C(1 + |X|_{\sim}^m)$.

 To see that the equation  \eqref{eq:equation_inverse_Cnu} can be solved recursively, consider the 
projections $\pi_j : \fg \to W_j$ and $\tilde \pi_j = \sum_{i=1}^j  \pi_i$.
Since $[W_j, W_k] \subset \oplus_{i=j+k}^m W_i$, the expression $\tilde \pi_j [X,Y]$ depends only on $\tilde \pi_{j-1}(X)$ and $ \tilde \pi_{j-1}(Y)$ 
and thus $\tilde \pi_j P(X,Y)$  depends only on $\tilde \pi_{j-1}(X)$ and $ \tilde \pi_{j-1}(Y)$. Hence 
  \eqref{eq:equation_inverse_Cnu} can be again solved recursively by successively  applying the projections $\tilde \pi_1, \ldots, \tilde \pi_m = \id$.
  The projection $\tilde \pi_i P(X,Y)$ can again be expressed as a sum of products of multilinear terms in $X$ and $Y$. 
  The only difference is that  the condition    $  \sum_{\ell=1}^j  = i$   is replaced by   $  \sum_{\ell=1}^j  \le i$.
  Nonetheless the bound  \eqref{eq:bound_Li_com} still holds and this implies
   \eqref{eq:bound_Qi}.
   
We note in passing that one can  show the recursive solvability of     \eqref{eq:equation_inverse_Cnu}    without introducing 
the spaces $W_i$,  by considering the abstract projections $\bar \pi_1, \ldots, \bar \pi_m = \id$ given by
$\bar \pi_i: \fg \to  \fg / \fg_{i+1}$. 
\end{remark}

\bigskip

\begin{remark}    \label{re:com_two_step}  For a step-$2$ group $C_\nu \circ \exp$ is an affine function and thus 
\begin{equation} \label{eq:mollification_2step}
 \log \com_\nu = \int_{\fg}  Y  \, d\log_* \nu(Y)   \quad \text{for step-$2$ groups.}
 \end{equation}

\end{remark}

\bigskip\bigskip
  We now show that the center of mass defined by $\com_\nu = C_\nu^{-1}(0)$ commutes with left translations,
  inversion and group homomorphisms. 
  
\begin{lemma} \label{le:center_behaves_natural} 
 Let $G$ and $G'$ be Carnot groups of step $m$ and $m'$, respectively.  Let $\Phi: G \to G'$ be a  group homomorphism (with derivative $D\Phi: \fg \to \fg'$), let  $I: G \to G$ be   the inversion map given by $I(x) = x^{-1}$, and $\nu$ be a Borel probability measure on $G$.

If $\nu$ has finite $m$-th moment then $(\ell_z)_* \nu$ and $I_* \nu$ are Borel  probability measures on $G$ with finite $m$-th moment and 
\begin{eqnarray}
\com_{(\ell_z)_* \nu} &=& \ell_z(\com_\nu),  \label{eq:com_translation}\\
\com_{I_* \nu} &=& I(\com_\nu).  \label{eq:com_inversion}
\end{eqnarray}
In particular, if $\nu$ is reflection symmetric, i.e., if $I_* \nu = \nu$,  then 
\begin{equation}  \label{eq:com_inversion_invariant}
\com_{(\ell_z)_*\nu} = z.
\end{equation}

If $\nu$ has finite $p$-th moment for $p\geq \max(m,m')$, then $\Phi_* \nu$ is a Borel probability measure on $G'$ with finite   $p$-th moment, and
\begin{equation}
\com_{\Phi_* \nu} = \Phi(\com_\nu)\,.   \label{eq:com_Phi}
\end{equation}

\end{lemma}

\begin{remark}   \label{re:naturality_com_nilpotent} The assertion and the proof  immediately extend to nilpotent groups. To bound the $m$-th moment of the measures
$I_* \nu$ and $(\ell_x)_*\nu$ one  uses \eqref{eq:bound_P_nilpotent_homogeneous_norm}  instead of
\eqref{eq:bound_P_homogeneous_norm}.
\end{remark}

\begin{proof} 
Assume that $\nu$ has finite $p$-th moment. We have  $|\log \circ I(y)| = |- \log y| = |\log y|$. Thus $I_* \nu$ has finite $p$-th moment. 
Similarly $\log \circ \ell_{x}(y) = \log x  + \log y + P(\log x, \log y)$. By 
\eqref{eq:bound_P_homogeneous_norm} we have  $|P(\log x, \log y)| \le C(|\log x| + |\log y|)$. Thus $(\ell_x)_*\nu$ has finite $m$-th moment. 

To control the moment of $\Phi_* \nu$ we first note that $\Phi$ preserves the one-parameter subgroups and thus
$  \log_{G'} \circ \Phi \circ \exp_G = D \Phi$.  
    It follows that
$ (\log_{G'})_* \Phi_* \nu = (D\Phi)_* (\log_{G})_*\nu$. 
    It thus suffices to show that 
\begin{equation}  \label{eq:bound_DPhi_natural}
 |D \Phi (X)|  \le C (|X|^\frac1m + |X|).
\end{equation}
To show this,  observe that in a Carnot algebra the elements $\fg_j$ of the descending series are given by
$\fg_j = \oplus_{i=j}^m V_i$. Since $D\Phi$ is a Lie algebra homomorphism we have $D\Phi(\fg_j) \subset \fg'_j$.
Note also that $D\Phi$ is linear and hence bounded with respect to the Euclidean norms. Thus we have
for $X \in V_j$ 
$$ |D \Phi (X)| \le   C  \sum_{i=j}^m |\pi_i D\Phi (X)|_e^{\frac1i}
\le  C  \sum_{i=j}^m |X|_e^{\frac1i}  =  C \sum_{i=j}^m  |X|^\frac{j}{i}.$$
By Young's inequality we have $|X|^\frac{j}{i} \le C (|X|^\frac1m + |X|)$ whenever $i \ge j$. 
By linearity we get  \eqref{eq:bound_DPhi_natural} for all $X$.

To prove   \eqref{eq:com_translation}--\eqref{eq:com_Phi}
  it suffices to verify the corresponding
transformation rules for $C_\nu$. We have
$$ \log[ (\Phi(x))^{-1} \Phi(y) ]= \log[ \Phi(x^{-1} y)] = D\Phi(\log(x^{-1}y))$$
and thus
$$ C_{\Phi_* \nu}(\Phi(x)) = D\Phi (C_{\nu}(x)).$$
Setting $x = \com_\nu$ we get $C_{\Phi_* \nu}(\Phi(\com_\nu)) = 0$ and thus  \eqref{eq:com_Phi}.
Similarly the relation 
$$ \log [I(x^{-1} I(y)] = \log [ I(x^{-1}y) ] = - \log(x^{-1}y)$$
gives \eqref{eq:com_inversion} while the identity
$$ (\ell_z x)^{-1} (\ell_z y) = x^{-1} z^{-1} z y = x^{-1} y$$
gives  \eqref{eq:com_translation}.

Now assume that $\nu$ is reflection symmetric. Then \eqref{eq:com_inversion} implies that
$\com_{\nu} = e$. In combination with 
 \eqref{eq:com_translation} we obtain   \eqref{eq:com_inversion_invariant}.
\end{proof}

\bigskip\bigskip
\subsection{Mollifying maps between Carnot groups}
\label{subsec_mollifying_maps_between_carnot_groups}
In this subsection we define a mollification procedure for $L^p_{\loc}$-mappings into a Carnot group.  Traditional mollification of mappings into a linear target is based on averaging; since Carnot groups are not linear spaces, we replace averaging with the center of mass from Subsection~\ref{subsec_center_of_mass}.  

Let $\si_1$ be a smooth probability measure on a Carnot group $G$ with $\spt(\si_1)
\subset B(e,1)$. Thus $\sigma_1 = \alpha \mu$ where $\mu$ is the biinvariant measure on $G$
and $\alpha \in C_c^\infty(B(e,1))$.  We also assume that $\si_1$ is symmetric under inversion:
$I_*\si_1=\si_1$, where $I(x)=x^{-1}$.
For $x\in G$, $\rho\in (0,\infty)$, let $\si_\rho$, $\si_x$, and $\si_{x,\rho}$
be the pushforwards of $\si_1$ under the the corresponding Carnot scaling
and left translation:
\begin{equation}    \label{eq:transported_measures}
\si_\rho=(\de_\rho)_*\si_1\,,\quad \si_x=(\ell_x)_*\si_1\,,\quad
\si_{x,\rho}=(\ell_x\circ\de_\rho)_*\si_1=(\ell_x)_*(\si_\rho)\,.
\end{equation}

Let $G$ be any Carnot group and let $G'$ be an $m$-step Carnot group.
  Recall that  $f:G\ra G'$ is in $L^m_{loc}(G,G')$ if $f$ is   measurable 
and  
\begin{equation}  \label{eq:f_in_Lm_loc}
 \hbox{$ y \mapsto  d_{CC, G'} (f(y), e) $   
 belongs to $L^m_{loc}(G)$}.
\end{equation}
In particular every continuous map belongs to $L^m_{loc}(G, G')$. 
We will see shortly  that $f \in L^m_{loc}(G,G')$ implies that the push-forward measure
 $f_* \sigma_z$ has finite $m$-th moment.
We may then  define a mollified map   $f_1:G\ra G'$ by
\begin{equation}   \label{eq:mollification}
f_1(x)=\com(f_*(\si_{x}))\,,
\end{equation}
and maps $f_\rho:G\ra G'$ by 
\begin{equation}     \label{eq:mollification_scaled}
f_\rho=\de_\rho\circ(\de_{\rho^{-1}}\circ f\circ \de_\rho)_1\circ \de_{\rho^{-1}}\,.
\end{equation}

Recall that for $p \in [1, \infty)$ the $L^p$ oscillation on a set $A$ is defined by
\begin{equation} \label{eq:osc_Lp}
\osc_p(f,A) := \inf_{a \in G'} \left(    \int_A d_{CC, G'}^p(f(y), a) \, \mu(dy)  \right)^{1/p}
\end{equation}
where $\mu$ is the biinvariant measure on $G$. 

\begin{lemma}
\label{le:moll_prop_Lm} 
Let $G$  be a Carnot group, let  $G'$  be  an  $m$-step Carnot group, let  $ p \in [m, \infty)$ and $f \in L^p_{loc}(G,G')$.
As above  let $\sigma_1 = \alpha \mu$ with $\alpha \in C_c^\infty(B(e,1))$ and $I_* \sigma_1 = \sigma_1$. Define $\sigma_z$
by \eqref{eq:transported_measures} and $f_1$ and $f_\rho$ by  \eqref{eq:mollification} and
 \eqref{eq:mollification_scaled}. 
 Then: 
\ben
\item For all $z \in G$ the measures $f_* \sigma_z$ have finite $p$-th  moment. 
\item For all $\rho\in(0,\infty)$, 
$$
\de_{\rho^{-1}}\circ f_\rho\circ\de_\rho
=(\de_{\rho^{-1}}\circ f\circ \de_\rho)_1\,.
$$
\item For all $a \in G$ and $b \in G'$ 
\begin{equation*} \label{eq:smoothing_commutes_translation}
(\ell_b \circ f \circ \ell_a)_1 = \ell_b \circ f_1  \circ \ell_a.
\end{equation*}
\item   For all $\rho > 0$ 
$$ (\delta_\rho \circ f)_1 = \delta_\rho  \circ f_1.$$
\item \label{it:bound_oscillation_f1}  
Assume that for some $x_0 \in G$ and some  $a \in G'$
$$ \int_{B(x_0,1)}   d^m_{CC, G'}(f(x), a)  \, d\mu \le R^m.$$
Then 
$$ d_{CC, G'}(f_1(x_0), a) \le C R  \quad \text{where $C= C(G,G', \sigma_1)$.}   
$$   
\item \label{it:convergence_f_rho}
We have 
$$ f_\rho \to f  \quad \text{a.e.  and in $L^p_{loc}(G)$.}    $$
If $f$ is continuous, then  $f_\rho \to f$ locally uniformly. 
\item   \label{it:bound_derivatives_f1} 
 If $\osc_{m}(f, B(x_0,1)) \le R$
then the (Riemannian) norms of the  derivatives of 
$f_1$  are
controlled at  $x_0$, i.e.  $\|D^i( f_1)(x_0)\|<C=C(i,R, G, G', \sigma_1)$ and
$\| D^i\big((\delta_{R^{-1}} \circ f_1)(x_0)\big)\| \le C(i,G,G', \sigma_1)$. 
\een
\end{lemma}

In assertion   \eqref{it:bound_derivatives_f1} 
 the 'Riemannian' derivatives are computed with respect to the charts $\varphi=   \log_{G} \circ \ell_{x_0^{-1}}: G \to \fg$
and $\psi = \log_{G'} \circ \ell_{f_1(x_0)^{-1}}: G' \to \fg'$, i.e. we estimate the derivatives of the map
$$  \psi \circ f_1 \circ \varphi^{-1} = \log \circ \ell_{f_1(x_0)^{-1}} \circ f_1 \circ \ell_{x_0} \circ \exp  : \fg \to \fg'$$
at $0$.

\bigskip
\begin{remark}  One can consider more general domains. First,  if  $U \subset G$ is open,  then  it 
 follows from the proof that the results extend to maps in $L^m_{loc}(U;G)$, whenever the expressions
make sense. In particular we need that $\spt \sigma_z \subset U$. Taking $G = \R^N$ we in particular obtain a smoothing operation
for maps $U \subset \R^N \to G'$. If $U$ is an open subset in a metric measure space $X$ we can  abstractly define the mollification $f_1(z)$ 
using a general family of  compactly supported Borel probability  measures $\sigma_z$.
 In this case there is no notion of left translation to define $\sigma_z$, but the measures
$\sigma_z$ should be in a suitable sense concentrated near $z$. Then (1) still holds and it is  easy to prove counterpart of (5). Moreover the
proof of (6) shows that $z \mapsto f_1(z)$ is locally Lipschitz (with bounds on the local Lipschitz constant in terms of the $L^m$-oscillation of $f$),
provided that the measures are such that for each function $h \in L^1_{loc}(X)$ the map $z  \mapsto \int_X h \sigma_z$ is Lipschitz.
\end{remark}

\bigskip

\begin{proof}[Proof of Lemma~\ref{le:moll_prop_Lm}]  
We will sometimes denote the Carnot-Caratheodory distance $d_{CC,G'}$ generically by $d$ for brevity.
  
(1). We have
$$
\int_{G'}  d_{CC, G'}^p(y',e) \, df_* \sigma_z(y') \\
= \int_G d_{CC, G'}^p(f(y),e)\, d\sigma_z(y).
$$
Now $\spt \sigma_z = z \spt \sigma$ is compact. Thus the right hand side is finite
since by assumption  $y \mapsto d_{CC, G'}^p(f(y), e)$ is integrable over compact sets.
By   \eqref{eq:ball_box} this implies  that $f_*\si_z$  has finite $p$-th moment.

(2). This is immediate from the definition.

(3). We have  $(\ell_a)_* \sigma_z =(\ell_a)_* ((\ell_z)_* \sigma) = \sigma_{az}$ and thus 
$$(f \circ \ell_a)_1(z) = \com_{(f \circ \ell_a)_*\sigma_z} = \com_{f_*\sigma_{az}} = f_1 (az) = (f_1 \circ \ell_a)(z).
$$
The identity $(\ell_b \circ f)_1 = \ell_b \circ f_1$ follows from \eqref{eq:com_translation}.

(4). This follows from  \eqref{eq:com_Phi}.

 \eqref{it:bound_oscillation_f1}.
 Since mollification commutes with pre- and postcomposition by  left translation we may assume that $a=e$
 and $x_0 = e$. 
 Since mollification commutes with dilation we may also assume $R=1$. 
 By   \eqref{eq:formula_com1} and  \eqref{eq:formula_com2} 
 $$ \log f_1 = \log \com f_* \sigma$$
 is a polynomial of  the polynomial moments  
  $$ A_i= \int_{\fg'}   L_i(Y,  \ldots, Y) \, d\log_* f_*\sigma_{1}(Y).$$
 Here and in the following we just write $\log$ instead of $\log_{G'}$ for ease of notation.
  It just suffices to prove bounds on the $A_i$ (which only depend on $G$, $G'$ and $\sigma$). 
 Now $\sigma = \alpha \mu$ where $\mu$ is the biinvariant measure on $G$ and $\alpha \in C_c^\infty(B(e,1))$
 and thus 
 $$ A_i = \int_{G}  L_i( \log f(z),  \ldots, \log f(z)) \,  \alpha(z)  \,  \mu(dz).$$
 Now by 
 \eqref{eq:bound_Qi} we have
 $ |L_i(Y, \ldots, Y)| \le C (1 + |Y|^m)$.  
 By  \eqref{eq:ball_box} we have $|\log f(z)| \le C d_{CC, G'}(f(z), e)$. 
 
 Hence
 $$ |A_i|  \le C  \int_{B(e,1)}  (1 +  d_{CC, G'}^m(f(z), e)) \, \| \alpha\|_\infty  \, \mu(dz)  \le C,$$
 as desired.

\eqref{it:convergence_f_rho}. Taking $a= f(x_0)$ in assertion  \eqref{it:bound_oscillation_f1}
and unwinding definitions we see that

\begin{equation}
\label{eqn_unwinding_5}
d^m(f_\rho(x_0), f(x_0)) \le C   \av_{B(x_0, \rho)} d^m( f(x), f(x_0)) \, d\mu(x).
\end{equation}
Thus at every Lebesgue point $x_0$  we have $f_\rho(x_0) \to f(x_0)$,  and by Lemma~\ref{le:lebesgue_point_new}  we therefore have $f_\rho\ra f$  almost everywhere.
If $f$ is continuous, then it is uniformly continuous on compact sets and thus $f_\rho \to f$ locally uniformly.
To get convergence in $L^p_{\loc}(G)$  for $f \in L^p_{\loc}(G)$ with $p\ge m$, it suffices to  show that the restriction of $y \mapsto d^p_{CC,G'}(f_\rho(y), f(y))$
to any  ball $B(e,R)$ is equi-integrable for $0 < \rho \le 1$, see Proposition~\ref{pr:dominated_equiintegrable}   below.  
To that end, we first observe that (\ref{eqn_unwinding_5}) gives, by Jensen's inequality and the triangle inequality:
\begin{equation}
\label{eqn_d_m_jensen_triangle}
\begin{aligned}
d^p(f_\rho(y),f(y))&=\left(d^m(f_\rho(y),f(y))  \right)^{\frac{p}{m}}\\
&\leq  \left(C   \av_{B(y, \rho)} d^m( f(x), f(y)) \, d\mu(x)  \right)^{\frac{p}{m}}\\
&\leq C_p\av_{B(y,\rho)}d^p(f(x),f(y))\,d\mu(x)\\
&\leq C_p\av_{B(y,\rho)}\left(d^p(f(x),f(e))+d^p(f(y),e)\right)\,d\mu(x)\\
&=C_p\left(\av_{B(y,\rho)}h(x)\,d\mu(x)    + h(y)  \right)
\end{aligned}
\end{equation}
where $h(x) := d^p_{CC, G'}(f(x), e)$ and $C_p$ denotes a generic constant depending on $p$.  The right hand side of \eqref{eqn_d_m_jensen_triangle} can be written as $C_p(  h \ast \varphi_\rho + h)(y)$ with 
$ (f\ast g)(y) :=  \int_G   f(x) g(x^{-1} y) \, \mu(dx)$ and 
$\varphi_\rho(z)=   \rho^{-\nu} 1_{B(0,1)}(\delta_{\rho^{-1}} z)$. 
Since $h \in L^1_{\rm loc}(G)$,  the mollifications  $h \ast \varphi_\rho$ converge to $h$ in $L^1_{loc}(G)$
as $\rho \to 0$. Thus the right hand side of \eqref{eqn_d_m_jensen_triangle} is equi-integrable  on each ball $B(e,R)$ for $0 < \rho \le 1$.

  \eqref{it:bound_derivatives_f1}.  It suffices to prove the estimate for the derivatives of  $\delta_{R^{-1}} \circ f_1$
  at $x_0$. Then the other estimate follows from the chain rule since $\delta_R$ is smooth. In view of assertion 
 (4) we may  in addition assume $R=1$.
    In view of assertion (3) we may assume  without loss of generality that $x_0 = e$ and $f_{1}(x_0) = e$.
Thus we have to estimate the derivatives of the map
$$g =  \log_{G'} \circ f_1  \circ \exp_G.$$

 We begin with the following observation. Let $h \in L^1_{loc}(G)$ and define
    $$ \bar h(x) :=    \int_G  h \, \sigma_x.$$
  Then $\bar h$ is smooth and the derivatives are uniformly controlled.  In particular
  \begin{equation}   \label{eq:smoothness_in_sigma_z}
  \sup_{1 \le j \le k} |D^k (\bar h  \circ \exp)(0)|  \le C(G, k, \| \alpha\|_{C_k}, \| g\|_{L^1(B(0,1))}).
  \end{equation}
  Indeed, using the definition of $\sigma_x$   and the fact that $\log_*\mu = \mathcal L^N$ we get 
  \begin{align*}  & \, 
   (\bar h \circ \exp)(X) = \int_G h(y) \alpha((\exp X)^{-1}  y) \, d \mu(y) \\
    =   & \, \int_{\log B(e,1)} (h \circ \exp)(Y) (\alpha \circ \exp)( (-X)  \ast Y) \, d\mathcal L^N(Y).
    \end{align*}
  Here $(-X) \ast Y = -X + Y + P(-X,Y)$ is the induced  group operation on the Lie algebra.
  Since $\alpha$ has compact support in $B(e,1)$  and the group operation $\ast$ is continuous,  it follows
  that  $Y \mapsto (\alpha \circ \exp) ((-X) \ast Y)$ is supported in a  fixed compact subset of  $\log B(e,1)$  for all sufficiently small $X$. 
  Hence differentiation and integration commute and the assertion follows since $\| h \circ \exp\|_{L^1(\log(B(e,1)))} = \| h \|_{L^1(B(e,1))}$.

By   \eqref{eq:formula_com1} and  \eqref{eq:formula_com2} the quantity
$ \log f_1 (x) = \log  \com_{f_*(\sigma_x)}$
 is a polynomial in the polynomial
moments  
\begin{align*}  & \,  A_i(x) = \int_{g'}   L_i(Y,  \ldots, Y) \, d\log_* f_*(\sigma_{x})(Y)
=  \,  \int_G L_i (\log f,  \ldots, \log f)  \, d\sigma_x, 
\end{align*}
and by \eqref{eq:bound_Qi} we have
 $$ |L_i(Y, \ldots, Y)| \le C (1 + |Y|^m).$$
 By  \eqref{eq:ball_box} we have   $|Y| \le C d_{CC, G'}(\exp Y)$. 
 Thus the function $h$ defined by  $h(y) := L_i (\log f(y),  \ldots, \log f(y))$ is in $L^1_{loc}(G)$.
 Hence   by    \eqref{eq:smoothness_in_sigma_z}, the map  $x \mapsto A_i(x)$ is smooth with uniform bounds in terms of 
 $G$, $\alpha$ and $\| d_{CC, G'}^m(f( \cdot), e)\|_{L^1}$. Since $\log f_1$ is a polynomial (depending on $G'$) in the $A_i$ it is
 also smooth and the derivatives are controlled in terms of the same quantities and $G'$. 
 It only remains to show that  $\| d_{CC, G'}^m(f( \cdot), e)\|_{L^1}$ is controlled by a constant, taking into account the normalisations
 $R=1$ and $f(e) =e$.

 By assumption there exists an $a \in G'$ such that 
$$ \int_{B(e,1)} d_{CC, G'}^m(f(x), a)  \, \mu(dx)   \le 1.$$
Thus assertion  \eqref{it:bound_oscillation_f1} yields
$$ d_{CC, G'}(f_1(e), a) \le C(G, G', \sigma_{1}).$$
Since $f_1(e) = e$  it follows from the triangle inequality that
$$ \int_{B(e,1)} d_{CC, G'}^m(f(x), e)  \,  \mu(dx)   \le (1 +  C(G, G', \sigma))^m.$$
This concludes the proof of  assertion
 \eqref{it:bound_derivatives_f1}.
 \end{proof}

For the proof of assertion \eqref{it:convergence_f_rho} we used the following standard extension of the dominated convergence
theorem. Let $(A,  \mathcal A,  \mu)$ be a measure space with $\mu(A) < \infty$. We say that a  family 
of integrable functions $f_\alpha: E \to \R$ is equi-integrable if for every $\eps >0$ there exists a $\delta > 0$
such that $\mu(E) < \delta$ implies $\int_E |f_\alpha| \, d\mu < \eps$ for all $k$. Clearly every finite family of integrable
functions is equi-integrable and thus every $L^1$-convergent sequence of functions is equi-integrable.

\begin{proposition}    \label{pr:dominated_equiintegrable} Let $(A,  \mathcal A,  \mu)$ be a measure space with $\mu(A) < \infty$.
Assume that $f_k \to f$ a.e. in $A$ and  that for some $s \in [1, \infty)$   the family $\{|f_k|^s\}$ is equi-integrable. Then
$f_k \to f$ in $L^s(A)$.
\end{proposition}

\begin{proof} An equi-integrable sequence is in particular bounded in $L^1(A)$. Thus by Fatou's lemma $f \in L^s(A)$. 
Since $|f_k - f|^s \le 2^s (|f_k|^s + |f|^s)$ and since the right hand side is integrable it suffices to consider the case $s=1$, 
$f=0$, $f_k \ge 0$.  Pick $\rho>0$. Let $F_{k, \rho} : \{ x \in A : f_k >  \rho \}$.  Then $\mu(F_{k, \rho}) \to 0$ as $k \to \infty$, so by equi-integrability of $\{f_k\}$ we have 
$$ \limsup_{k \to \infty} \int_A f_k \le \rho \mu(A) + \limsup_{k \to \infty}  \int_{F_{k, \rho}} f_k \le  \rho \mu(A)\,.$$
Since $\rho >0$ was arbitrary, we get $\| f_k \|_{L^1(A)} \to 0$. 
\end{proof}

\bigskip\bigskip
\begin{lemma}  \label{le:moll_prop_Lm_bis}
Let $G$ be a Carnot group and let $G'$ be an $m$-step Carnot group.
\ben
\item If $f:G\ra G'$ is a group homomorphism, then $f_1=f$.
\item 
If $\{f_k:G\ra G'\}$ is a sequence of continuous maps,
and $f_k\ra f_\infty$ in $L^m_{loc}(G,G')$, i.e. 
$d_{CC,G'}(f_k, f_\infty) \to 0$
in $L^m_{loc}(G)$, 
then the sequence  of 
mollified maps $\{(f_k)_1\}$ converges in $C^j_{loc}$  (with respect to the Riemannian structure)
to $(f_\infty)_1$, for all $j$.
\een
\end{lemma}

\begin{proof} (1). This follows directly from \eqref{eq:com_Phi} and 
 \eqref{eq:com_inversion_invariant}. 

(2). The main point is to show that 
\begin{equation} \label{eq:pointwise_convergence_fk1}
(f_k)_1(x_0) \to (f_\infty)_1(x_0) \quad \forall x_0 \in G. 
\end{equation}
Then $C^i_{loc}$ convergence will follow  from  the uniform bounds in 
Lemma \ref{le:moll_prop_Lm}.

 Since mollification commutes with pre- and postcomposition by left-translation we may assume that
 $x_0 = e$ and $(f_\infty)(e) = e$. 
 To prove the pointwise convergence $(f_k)_1(e) \to e$
 we use the following fact. 
 Suppose that 
 $$  \text{$\varphi: G' \to \R$ 
 is continuous and $\varphi(y) \le C d^m_{CC, G'}(y, e)$.}$$
 Then 
 \begin{equation}  \label{eq:dominated_convergence_continuous}
  \varphi \circ f_k \to \varphi \circ f_\infty \quad \text{in $L^1_{loc}(G)$.}
  \end{equation}
  This follows easily  from Proposition~\ref{pr:dominated_equiintegrable} by first passing to an a.e.
  converging sequence and then using uniqueness of the limit. 
 
 Now  recall that  $\log (f_k)_{1}(e)$ 
  is a polynomial $P$  in the polynomial moments
 $$ A_i^k:=    \int_{\fg'} L_i(Y, \ldots, Y) (\log \circ f_k)_*\sigma_1(Y) = \int_{G} Q_i(\log f_k,   \ldots, \log f_k) \, \sigma_1.$$
  It thus suffices to show that $\lim_{k \to \infty} A^k_i = A^\infty_i$. 
In view of (\ref{eq:bound_Qi}),   this follows from 
  \eqref{eq:dominated_convergence_continuous}
  applied to the function $\varphi(y) = L_i( \log y, \ldots, \log y)$ since $\sigma_1 = \alpha \mu$ and
  $\alpha$ is bounded and compactly supported.
  \end{proof}

\section{Pansu pullback and mollification}
\label{sec_pansu_pullback_mollification}

We now consider the behavior of pulling back using
a mollified map between Carnot groups $G$ and $G'$. 
For an open set $U \subset G$ define
$$ U_\rho := \{ x \in U : \dist(x, G \setminus U) > \rho\}.$$
Assertion   \eqref{it:pullback_pre_pansu} of the following lemma
provides a key connection between convergence of 
the mollified pullback $f_\rho^*\alpha$ and Pansu differentiability.

\begin{lemma}  \label{lem_moll_calc_forms}
 Let $U \subset G$ be open and let $f \in L^m_{loc}(U, G')$.
 Suppose that $\alpha \in \Omega^{k,w_\al}(G')$ and $\gamma \in \Omega^{N-k,w_\ga}(G)$ are left-invariant forms.
 In particular, if $k=N$ then $\gamma$ is a constant zero-form, i.e. a constant function and $w_\gamma=0$.
Then 
\ben
\item  \label{it:pullback_alpha_rho}    For every $x\in U_\rho$, 
$$
(f_\rho^*\al  \wedge \gamma)(x)=\rho^{-(\nu + w_\al+w_\gamma)}(h_1^*\al \wedge \gamma)(\de_{\rho^{-1}}(x))\,,
$$
where $h=\de_{\rho^{-1}}\circ f\circ \de_\rho$.
\item    \label{it:pullback_pre_pansu}  For every $x \in U_\rho$,
$$
(f_\rho^*\al \wedge \gamma)(x)=\rho^{-(\nu + w_\al+w_\gamma)}
\big((\de_{\rho^{-1}}\circ f_x\circ \de_\rho)_1^* \al   \wedge \gamma\big)(e)\,,
$$
where $f_x=\ell_{f(x)^{-1}}\circ f\circ \ell_x$.
\item   \label{it:bound_pullback} If $x \in U_\rho$ and  $\osc _m (f,B(x,\rho)) \le C  \rho^{1 + \frac{\nu}{m}}$, then 
\begin{equation*}
\|(f_\rho^*\al)  \wedge \gamma)(x) \| \leq 
C'\,   C^{-w_\al}\rho^{-( \nu + w_\al+w_\gamma)}
\| \alpha\| \, \| \gamma\|.
\end{equation*}
\een
\end{lemma}

\begin{proof} The proof of the first two assertions is exactly the same
as the proof of the corresponding assertions in  Lemma 6.4   in  \cite{KMX1}.
We include the short calculation for the convenience of the reader.

\eqref{it:pullback_alpha_rho}.
Note that  $\{ z : B(z, 1) \subset \delta_{\rho^{-1}} U\} = \delta_{\rho^{-1}} U_\rho$
and thus
$$h: \delta_{\rho^{-1}} U \to G', \quad h_1:  \delta_{\rho^{-1}} U_\delta \to G'.   $$ 
For $x \in U_\rho$ we have
\begin{align*}
(f_\rho^*(\al) \wedge \gamma)(x) 
=& \big(   (\de_\rho \circ h_1 \circ \de_{\rho^{-1}})^* \alpha \wedge \gamma\big)(x)\\
=& (\de_{\rho^{-1}}^* h_1^* \de_\rho^* \alpha \, \wedge\,  \de_{\rho^{-1}}^* \de_\rho^* \gamma)(x)\\
=& \rho^{-(w_\al + w_\gamma)}  (\de_{\rho^{-1}}^* h_1^* \alpha \wedge \de_{\rho^{-1}}^* \gamma)(x)\\
=&  \rho^{-(w_\al + w_\gamma)}  \big(\de_{\rho^{-1}}^* (h_1^* \alpha \wedge  \gamma)  \big)(x)\\
=&  \rho^{-(\nu + w_\al + w_\gamma)} (h_1^* \alpha \wedge  \gamma)(\de_{\rho^{-1}} x).
\end{align*}
In the last step we used that $h_1^* \alpha \wedge  \gamma$ is a multiple of the volume form, which has weight $-\nu$. 

\bigskip

  \eqref{it:pullback_pre_pansu}.
 With $h$ as in \eqref{it:pullback_alpha_rho} we get
\begin{align*}
h=&\de_{\rho^{-1}}\circ f \circ \de_\rho\\
=&(\de_{\rho^{-1}}\circ \ell_{f(x)}\circ \de_\rho)\circ 
\de_{\rho^{-1}}\circ \ell_{f(x)^{-1}}\circ f\circ \ell_x\circ\de_\rho
\circ(\de_{\rho^{-1}}\circ \ell_{x^{-1}}\circ \de_\rho)\\
=&\ell_{\de_{\rho^{-1}}f(x)}
\circ \de_{\rho^{-1}}\circ f_x\circ \de_\rho
\circ \ell_{\de_{\rho^{-1}}x^{-1}}
\end{align*}
and so
$$
h_1=\ell_{\de_{\rho^{-1}}f(x)}
\circ (\de_{\rho^{-1}}\circ f_x\circ \de_\rho)_1
\circ \ell_{\de_{\rho^{-1}}x^{-1}}\,.
$$
Since $\al$ and $\gamma$ are left invariant we have for $x \in U_\rho$
\begin{align}   \label{eqn_h1_al}
&(    h_1^*\al    \wedge \gamma)
(\de_{\rho^{-1}}(x) ) \notag \\  
=&\ell_{\de_{\rho^{-1}}x^{-1}}^*[(\de_{\rho^{-1}}\circ f_x\circ \de_\rho)_1^*\al
\wedge \gamma] (\de_{\rho^{-1}}(x))    \\
=& [ (\de_{\rho^{-1}}\circ f_x\circ \de_\rho)_1^*\al
\wedge \gamma](e).  \notag
\end{align}
Combining \eqref{it:pullback_alpha_rho}
 with (\ref{eqn_h1_al}) gives (2).

\eqref{it:bound_pullback}. Note that our assumptions imply 
that $\osc_m (h,B(\delta_{\rho^{-1}} x,1)) \le C$.
Thus Lemma~\ref{le:moll_prop_Lm}~\eqref{it:bound_derivatives_f1} 
 implies that
$$ \| D( \delta_{C^{-1}} \circ h_1)(\delta_{\rho^{-1}} x) \| \le C'.$$
Using   assertion~\eqref{it:pullback_alpha_rho}  we get
\begin{align*}
 &   \| (f_\rho^* \alpha \wedge \gamma)(x) \| \\
 = & \| \rho^{-(\nu + w_\alpha + w_\gamma)}
 \big(         \big(  \delta_C \circ (\delta_{C^{-1}} \circ h_1) \big)^*\alpha
  \wedge \gamma \big)(\delta_{\rho^{-1}}(x)) \|    \\
  = & C^{-w_\alpha}   \| \rho^{-(\nu + w_\alpha + w_\gamma)}
 \big(          (\delta_{C^{-1}} \circ h_1)^*\alpha
  \wedge \gamma \big)(\delta_{\rho^{-1}}(x)) \|    \\
  \le & C'   \,  C^{-w_\alpha}  \rho^{-(\nu + w_\alpha + w_\gamma)}  \| \alpha\| \, \| \gamma\|.
\end{align*}
\end{proof}

\begin{theorem}[Approximation theorem]  \label{th:weight_controlled_pullback}
Let $G$ be a Carnot group of topological dimension $N$ and homogeneous 
dimension $\nu$ and let $G'$ be an $m$-step Carnot group.   Let $U \subset G$ and $U' \subset G'$ be open.  Suppose that $\omega \in \Omega^{k,w_\om}(U')$ has continuous and bounded  coefficients
and  $\gamma \in \Omega^{N-k,w_\ga}(G)$ is a left-invariant form
such that
\begin{equation}   \label{eq:condition_minimal_weight} w_\om+w_\ga \le - \nu.
\end{equation}
Assume that $p \ge - w_\om$  and $\frac{1}{p} \le \frac1m + \frac{1}{\nu}$.
Let $f: U \to U'$ be a map in $W^{1,p}_{loc}(U, G')$. 
Let $\overline \omega$ denote the extension of $\omega$ to $G' \setminus U'$ by zero. 
Then 
\begin{equation}  \label{eq:L1_convergence_pullback}
 f_\rho^*\overline{\omega} \wedge \gamma \to f_P^* \omega \wedge \gamma   \quad   \text{in $L^s_{\loc}(U)$ with $s = \frac{p}{-w_\om}$.}
 \end{equation}
 Equivalently, we have convergence of weight $w$ components
$$
(f_\rho^*\overline\om)_w\ra (f_P^*\om)_w  
$$
for $w\geq w_\om$, see Remark~\ref{rem_w_component_convergence} below.

    In particular  we have
 \begin{equation}  \label{eq:L1_convergence_pullback_top_degree}
 f_\rho^* \overline{\omega}  \to f_P^* \omega   \quad  
  \text{in $L^{\frac{p}{\nu}}_{\loc}(U)$  \quad
   if $\omega \in \Omega^N(U')$}    
 \end{equation}
\end{theorem}

\bigskip\bigskip
\begin{remark} The mollifications $f_\rho$ may take values outside $U'$. This is why we need to extend $\omega$ outside $U'$ to define the pull-back by $f_\rho$. 
The proof shows that  convergence in  \eqref{eq:L1_convergence_pullback} does not depend on which extension we choose. More precisely, if $ \widetilde \omega$
is any extension of $\omega$ which is everywhere defined, bounded, measurable and satisfies $\widetilde{\om}(x)\in\La^{k,w_\om}$   at each point, 
then 
\begin{equation} \label{eq:L1_convergence_pullback_other_ext}
 f_\rho^*\widetilde{\omega} \wedge \ga \to f_P^* \omega \wedge \ga   \quad   \text{in $L^s_{\loc}(U)$ with $s = \frac{p}{-w_\om}$.}
 \end{equation}
\end{remark}

\bigskip\bigskip
\begin{remark}
\label{rem_w_component_convergence}
The convergence in \eqref{eq:L1_convergence_pullback} in connection with the condition  \eqref{eq:condition_minimal_weight}  is equivalent to convergence of weight $w$ components  
$$
(f_\rho^*\overline \om)_w\ra (f_P^*\om)_w
$$
for $w\geq w_\om$.  To see this, note that for $\om$ fixed and $w\geq w_\om$,  we may choose a basis $\{\ga_i\}$ of the space of left-invariant forms $\ga\in\Om^{N-k,-\nu-w}(G)$, and this is dual via the wedge product to a basis $\{\al_{w,i}\}$ for the left invariant forms in $\Om^{k,w}(G)$.  Thus \eqref{eq:L1_convergence_pullback} applied to each $\ga_i$ yields convergence 
$$
(f_\rho^*\overline \om)_{w,i}\ra (f_P^*\om)_{w,i}
$$ 
in $L^s_{\loc}$ where the notation $(\be)_{w,i}$ for a form $\be$ is defined by $\be_w=\sum_i(\be)_{w,i}\al_{w,i}$.  In  particular, if $G=G'$ and the weight of $\omega$ is minimal among nonzero forms of degree $k$, 
then all components converge and thus $f_\rho^*\overline \omega \to f_P^*\omega$ in $L^s_{\loc}(U)$.
\end{remark}

\bigskip\bigskip

\begin{remark}    \label{re:natural_exponent_p}
We now comment on the assumptions on the exponent $p$.  The obvious estimate for the pullback is 
$$
|f_P^*\om|(x)\leq C|D_Pf(x)|^{-w_\om}|\om|(f(x))\,.
$$
Therefore, in general, one would expect $p\geq -w_\om$ to be the optimal lower bound on the Sobolev exponent.   However if $w_\om + w_\ga < - \nu$ then some improvement is possible, see Corollary~\ref{co:improved_exponents_approximation} below.   
In the abelian case we have $w_\om = -k$ and it is known that the condition $p \ge k$ is necessary to have $L^1_{\loc}$ convergence of $f_\rho^* \omega$. 
 Typical  counterexamples are given by suitable $0$-homogeneous functions. For example,  if $G=G'= \R^N$ and $\omega =dy_1 \wedge \ldots \wedge dy_N$
 one can take $f = \frac{x}{|x|}$. Then $f^*\omega = 0$, $f \in W^{1,p}(U; \R^N)$ for all $p < N$,  but it is easily seen, e.g. by a degree argument,  that $f_\rho^*\omega$ weak$*$ converges  to the Dirac mass  
 $\mu(B(0,1)) \delta_0$ as $\rho\ra 0$, where we identify top degree forms with measures.    We do not know the optimal exponent $p$ for which 
the conclusion $f_\rho^*\omega \wedge \gamma \to f_P^*\omega \wedge \gamma$ in $L^1_{\loc}(U)$ holds. 
\end{remark}

\bigskip

\begin{proof}[Proof of Theorem~\ref{th:weight_controlled_pullback}]
We will prove the result using the dominated convergence theorem. 
In brief, this is implemented as follows.  Pointwise convergence almost everywhere  follows from the formula
in Lemma~\ref{lem_moll_calc_forms}~\eqref{it:pullback_pre_pansu},
Pansu differentiability a.e. (in an $L^m$ sense) and the fact the 
mollification improves $L^m$-convergence to $C^1$-convergence.
The majorant is obtained from the estimate in  Lemma~\ref{lem_moll_calc_forms}~\eqref{it:bound_pullback}
and the Sobolev-Poincar\'e inequality  which provides
a uniform estimate of the $L^m$ oscillation in terms of the maximal function of  the $p$-th power of   the (horizontal) derivative. 

We begin with some preparations.
Since we only want to prove convergence in $L^s_{\loc}$ we may assume that $f \in W^{1,p}(U;G')$.  By linearity it suffices to verify the theorem for forms $\omega = a \alpha$ where $\alpha$ is a left-invariant form with $w_\al + w_\ga \le - \nu$ and $a$   is a continuous and bounded  function.   We denote by  $\overline a$ the extension of $a$  by zero to $G' \setminus U'$.
Set $w_\alpha = \wt(\alpha)$ and $w_\gamma = \wt(\gamma)$. 
Fix a compact set $  K \subset U$.  
We next show pointwise convergence a.e. in $K$.
Recall that  $U_\rho := \{ x \in U : \dist(x, G \setminus U) > \rho\}.$
For $\rho >0$ small enough we have $K \subset U_\rho$.
By Lemma~\ref{lem_moll_calc_forms}~\eqref{it:pullback_pre_pansu} 
we have for $x \in K$
\begin{align}  \label{eq:pullback_pointwise_convergence}
&(f_\rho^*( \overline \omega)  \wedge \gamma)(x)\\
=&(\overline a\circ f_\rho)(x)    \,    (f_\rho^*\al  \wedge \gamma)(x)    \nonumber \\
=&( \overline a\circ f_\rho)(x)    \, \rho^{-(\nu + w_\al+w_\gamma)}
\big( (\de_{\rho^{-1}}\circ f_x\circ \de_\rho)_1^*\al  \wedge \gamma \big)(e)\,.
\nonumber 
\end{align}
By Theorem~\ref{th:Lp*_pansu_differentiability_new}  and the condition  $\frac1p \le \frac1m + \frac1\nu$
 we have for a.e.\  $x \in K$
the convergence  
$\de_{\rho^{-1}}\circ f_x\circ \de_\rho  \stackrel{L^m_{loc}}{\lra}   
D_P f$.
(Recall that we are using the notation $D_Pf(x)$ to denote a graded Lie algebra homomorphism $\fg\ra \fg'$ and a homomorphism of Carnot groups $G\ra G'$, depending on the context.)   
By Lemma~\ref{le:moll_prop_Lm_bis}  we get
$D(\de_{\rho^{-1}}\circ f_x\circ \de_\rho)_1(e)\ra D_Pf(x)$
as $\rho\ra 0$. 
Moreover by Lemma~\ref{le:moll_prop_Lm}~\eqref{it:convergence_f_rho} we have 
$f_\rho(x)  \to f(x)$ almost everywhere. 

Let $N \subset K$ be a null set  such that for all $x \in K \setminus N$ we have $f_\rho(x) \to f(x)$ and 
$D(\de_{\rho^{-1}}\circ f_x\circ \de_\rho)_1(e)\ra D_Pf(x)$. 
Since $U'$ is open, for each $x\in K\setminus N$  there exist a $\rho_0(x) > 0$ such that $f_\rho(x) \in U'$ for
all $\rho < \rho_0(x)$. Since $\overline a$ is continuous in $U'$ (and agrees there with $a$) it follows that 
$\overline a \circ f_\rho(x) \to a \circ f(x)$ for all $x \in K\setminus N$. Note that this convergence is independent of how we extend
$a$ outside $U'$.

Now if $w_\al+w_\gamma= - \nu $, then
\begin{align}  \label{eq:pointwise_convergence_pull_back}
(f_\rho^* \overline \omega  \wedge \gamma) (x) \ra &  (a\circ f)(x)  \, ((D_Pf(x))^* \al)(x)  \wedge \gamma\\
=&(f_P^*\omega \wedge \gamma)(x)  \nonumber
\end{align}
so we have pointwise convergence in this case.  If $w_\al+w_\gamma< - \nu$, then 
$(f_\rho^*\overline \omega  \wedge \gamma) (x) \ra 0$ as $\rho\ra 0$, while
$$
(f_P^*\omega \wedge \gamma)(x)   =(a\circ f)(x) 
\, \, ((D_P f)(x)^* \alpha)(x) \wedge \gamma\,.
$$
Now by  Lemma~\ref{lem_weight_facts}~\eqref{it:weight_facts_pullback}
we deduce that    $((D_P f)(x)^* \alpha)(x) \wedge \gamma$ is a form of weight strictly
less than $-\nu$ and hence zero. Thus if $w_\alpha + w_\gamma < - \nu$ we have $(f_P^*\omega \wedge \gamma)(x) = 0$.
Hence  we have shown that $(f_\rho^* \overline\omega  \wedge \gamma) (x) \ra (f_P^*\omega \wedge \gamma)(x)$
for a.e. $x \in K$.

By Proposition~\ref{pr:dominated_equiintegrable}, it remains only to show that $|f_\rho^*\overline \omega \wedge \eta|^s$ is equi-integrable for $s = \frac{p}{-\omega_\alpha}$.
If $m \ge 2$ define $q > 1$ by $\frac1q = \frac1m + \frac1\nu$  (if $m=1$, i.e., if $G$ is abelian,  take $q=1$; then 
 \eqref{eq:osc_bound_max_function} below follows directly from the Poincar\'e inequality).
Set   $\psi = | D_h f|^{q}$.
Then $\psi \in L^{\frac{p}{q}}(U)$. 
By the Sobolev-Poincar\'e inequality (\ref{eq:sobolev_poincare_new}) we have for $x \in K$
$$
\rho^{-\frac{\nu}{m}} \osc_m (f,B(x,\rho)) \le  C \rho   \left(  \rho^{-{\nu}} \int_{B(x,\rho)}  \psi  \right)^{\frac1q} = C \rho   \,  \psi_\rho^{\frac1q}(x)
$$
where 
$$ \psi_\rho: = \psi \ast  \rho^{-\nu} 1_{B(0, \rho)}$$
and 
 $$ (f\ast g)(x) := \int_G  f(x y^{-1})  g(y)  \, \mu(dy) = \int_G   f(y) g(y^{-1} x) \, \mu(dy).$$    
Since $\psi \in L^{\frac{p}{q}}_{\loc}(G)$ we have
\begin{equation}
\label{eqn_psi_rho_converges_l_p_q}
 \psi_{\rho} \to \psi \quad \text{in $L^{\frac{p}{q}}(K)$}    
\end{equation}
as $\rho\ra 0$.
Moreover
\begin{equation}  \label{eq:osc_bound_max_function}
 \osc_m (f,B(x,\rho)) \le C \rho^{1 + \frac{\nu}{m}}  \psi^{\frac1q}_\rho(x).
 \end{equation}
 Now let $s =\frac{p}{ -w_\alpha}$. Then by 
Lemma~\ref{lem_moll_calc_forms}~\eqref{it:bound_pullback}
  \begin{align}  \label{eq:pullback_equiintegrability}
|f_\rho^* \overline\omega \wedge \eta|^s(x)   \le  &  \, \|\bar a\|_\infty^s   \, |(f_\rho^*\alpha)(x) \wedge \eta|^s \\
\le &  \, C  \psi_{\rho}^{\frac{- s w_\alpha}{q}}(x) \rho^{-s(\nu+ w_\alpha + w_\gamma)} \,   
\| \alpha\|^s \, \| \gamma\|^s \,  \|\bar a\|_\infty^s     \nonumber \\
\le & \, C    \rho^{-s(\nu+ w_\alpha + w_\gamma)}    \psi_\rho^{\frac{p}{q}}.  \nonumber
\end{align}
In view of \eqref{eqn_psi_rho_converges_l_p_q} the family $\psi_\rho^{\frac{p}{q}}$ is equi-integrable, so \eqref{eq:pullback_equiintegrability} gives the desired equi-integrability of $|f_\rho^* \overline\omega \wedge \eta|^s$.    Note also that the argument used only the fact that the extension $\bar a$ is bounded.  
\end{proof}

\bigskip

The argument above shows easily that we have better convergence results if $\wt(\omega) + \wt(\gamma) < - \nu$.  We summarize these as follows. 

\begin{corollary} \label{co:improved_exponents_approximation}  With the notation and assumptions  of Theorem~\ref{th:weight_controlled_pullback} the following refinements of  \eqref{eq:L1_convergence_pullback} hold.
\ben
\item  If   $\frac1p \le \frac1m + \frac1\nu$, $p \ge -w_\om$ and $w_\om+w_\ga < - \nu$ then
\begin{equation}
\rho^{\nu + w_\om + w_\ga} f_\rho^*\overline \omega \we \ga\to 0  \quad    \text{in $L^s_{\loc}(U)$ with $s = \frac{p}{-w_\om}$.}
\end{equation}
\item 
Set $\beta = -\frac{p}{\nu}(w_\om+w_\gamma+\nu)$. If $\beta < -w_\omega$ then
\begin{equation}
 f_\rho^*\overline\omega\we\ga \to 0    \quad \text{in $L^{\frac{p}{-w_\omega-\beta}}_{\loc}(U)$.}    \end{equation}
 If $\beta \ge -w_\omega$ then
 \begin{equation}
 f_\rho^*\overline\omega\we\ga \to 0    \quad \text{locally uniformly}.   \end{equation}

 \een
\end{corollary}

\bigskip\bigskip

\begin{proof} The first assertion follows directly from the proof of Theorem~\ref{th:weight_controlled_pullback}.
Indeed,  \eqref{eq:pullback_pointwise_convergence}, Pansu differentiability a.e., and
the estimates for the mollification 
 imply that $\rho^{\nu + w_{\om}+w_{\ga}} f_\rho^*\overline\omega\we\gamma \to (a \circ f) f_P^* \alpha \wedge \gamma$ a.e. 
  Moreover $f_P^* \alpha \wedge \gamma =0$ since forms of weight strictly less than $-\nu$ must vanish. 
Regarding equi-integrability,      \eqref{eq:pullback_equiintegrability} yields
$$| \rho^{\nu + w_{\om}+w_{\ga}
} f_\rho^*\overline\omega|^{\frac{p}{-\wt(\omega)}} 
\le C |\psi_\rho|^{\frac{p}{q}}  \to C  |\psi|^{\frac{p}{q}}$$
in $L^1_{\loc}(U)$. Hence the assertion follows from Proposition~\ref{pr:dominated_equiintegrable}.

To prove the second assertion, we note that 
Lemma~\ref{lem_moll_calc_forms}~\eqref{it:bound_pullback} yields
\begin{align}  \label{eq:pullback_equiintegrability_refined}
  |f_\rho^* \overline\omega \wedge \gamma|(x) 
  \le  &
 \,  C 
\Big(    \rho^{-(1+ \frac{\nu}{m})}   
\osc_m (f,B(x,\rho) )   \Big)^{-w_\om}  
 \rho^{  e}
 \end{align} 
 where  where $C$ is a constant independent of $x$ and $\rho$, and
  $$ e= -(w_\omega + w_\gamma+ \nu).$$  
By  \eqref{eq:osc_bound_max_function} and the Poincar\'e-Sobolev inequality we have
\begin{align}
 \osc_m(f,  B(x,\rho)) \le & \,  C \rho^{1 + \frac{\nu}m} \psi_\rho^{\frac1q}(x),  \label{eq:osc1}\\ 
 \osc_m(f,  B(x,\rho)) \le  &  \, C \rho^{1 + \frac{\nu}m- \frac\nu{p}}  \label{eq:osc2}
\|D_h f\|_{L^p(B(x,\rho)}.    
\end{align}
 
Recall that $\beta = ep/  \nu$.
 If $\beta < -w_\omega$ then we take  \eqref{eq:osc1} to the power $-w_\omega-\beta$ and
  \eqref{eq:osc2} to the power $\beta$ to get
 \begin{align}  \label{eq:pullback_equiintegrability_refined}
  |f_\rho^*\overline\omega \wedge \gamma|(x) 
  \le  &
 \,  C 
 \psi_\rho(x)^{\frac{-w_\omega - \beta}{q}}  \, \|D_h f\|_{L^p(B(x,\rho))}^\beta. \, 
\end{align} 
Since $\|D_h f\|_{L^p(B(x,\rho))} \to 0$ as $\rho \to 0$, locally uniformly in $x$, and
$\psi_\rho^{\frac1q}$ converges in $L^p_{\loc}(U)$
it follows that 
$ |f_\rho^*\overline \omega \wedge \gamma| \to 0$ 
in $L_{\loc}^{    \frac{p}{-w_\omega-\beta}}(U)$.
If $\beta \ge -w_\omega$ we take  \eqref{eq:osc2} to the power $-w_\omega$ and get (for $\rho \le 1$)
the estimate
$  |f_\rho^* \overline\omega \wedge \gamma|(x) 
  \le 
 \,  C 
\|D_h f\|_{L^p(B(x,\rho))}^{-w_\omega}$. The assertion follows since the right hand side converges locally uniformly to zero.

\end{proof}

\bigskip

\bigskip

\bigskip

\bigskip

\bigskip

Next we apply the approximation theorem to show that  for certain components the Pansu pullback of differential  forms commutes with exterior differentiation. 
Note that in general the Pansu pullback does not commute with exterior differentiation (see \cite{KMX1}).

\begin{theorem} \label{th:pansupullback_exterior_d}~
Let $G$ be a Carnot group of topological dimension $N$ and homogeneous dimension $\nu$, let $G'$ be a $m$-step Carnot group, and 
$$
f:G\supset U\ra U'\subset G'
$$ 
be a $W^{1,p}_{\loc}$-mapping between open subsets.
Suppose that $\alpha  \in \Omega^{k,w_\al}(G')$ has  continuous and bounded coefficients  such that the weak exterior differential $d\alpha$ 
also has continuous and bounded coefficients. Let $\beta \in \Omega^{N-k-1,w_\be}(G)$ be a  closed left-invariant form. Assume that
\begin{equation}
w_\al+w_\be  = - \nu + 1.
\end{equation}
Then the following assertions hold.

\ben
\item If  $\alpha$ is weakly closed, $p \ge -w_\al$  and $\frac1p \le \frac1m + \frac1\nu$ 
then $f_P^*(\alpha) \wedge \be$  is weakly closed, i.e.
\begin{equation}
\int_{G} f_P^*(\alpha) \wedge \be \wedge d\varphi = 0 \quad \text{for all $\varphi \in C^\infty_c(U)$.}
\end{equation}
\item Assume $\wt(d\al)<w_\al$, and that  $d\alpha = \sum_{s\leq w< w_\al}\om^{(w)}$ is the weight decomposition of $d\alpha$. Assume that $p\geq -s$   and  $\frac1p \le \frac1m + \frac1\nu$.
Then
\begin{equation}
d( f_P^*\alpha \wedge \be) = f_P^*(d\alpha) \wedge \be   \quad \text{in  the sense of distributions,}
\end{equation}
i.e. 
\begin{equation}  \label{eq:distributional_exterior_derivative}
 (-1)^N 
\int_{U}    f_P^*\alpha \wedge \be \wedge d\varphi =    \int_{U}   f_P^*(d\alpha)  \wedge \be  \, \varphi  \quad 
\forall \varphi \in C_c^\infty(U).
\end{equation}
\een
\end{theorem}

\begin{remark}  \label{re:codegree_coweight_1}
If $G = G'$, $k = N-1$ and $w_\al = - \nu +1$ we have $d\alpha \in \Omega^N(G')$ and hence $\wt(d\alpha) = -\nu$. Thus for $p \ge \nu$
we can take $\be \equiv 1$ and we get
\begin{equation}  \label{eq:pullback_adjoint}
d  f_P^*\alpha = f_P^*(d\alpha) \quad \text{if $G=G'$, $k=N-1$, $\wt(\alpha) = - \nu +1$}
\end{equation}
in the sense of distributions. For $2$-step groups this was first shown by Vodopyanov, see \cite{vodopyanov_foundations}.
\end{remark}

\begin{remark} If we use Corollary~\ref{co:improved_exponents_approximation}  then the condition on the exponent 
in the second assertion can be slightly improved if   $s >  - w_\al + 1$. In that case we can replace the condition 
$ p \ge s$ by
$$ p \ge s-\frac{p}\nu (s + w_\beta + \nu) = s- \frac{p}{\nu}(s+1-w_\alpha),$$
or, equivalently, 
\begin{equation}
\frac{s}{p} \le 1+ \frac1\nu(s+1-w_\alpha).
\end{equation}
\end{remark}

\bigskip\bigskip
\begin{proof}[Proof of Theorem~\ref{th:pansupullback_exterior_d}]  Since $\beta$ is closed we have $d (\varphi \beta) = d\varphi \wedge \beta$, and hence $\wt(d(\phi\be))\leq \wt\be-1$.
 Using that  the (weak) exterior derivative commutes with pullback by smooth functions we get
\begin{equation}
\label{eqn_smooth_pullback}
 \int_G    f_\rho^*\alpha \wedge \beta \wedge d\varphi =  (-1)^{N}   \int_G   f_\rho^*(d\alpha) 
  \wedge \beta   \, \varphi  \quad 
\forall \varphi \in C_c^\infty(U).
\end{equation} 
Hence  both  assertions follow by applying   Theorem~\ref{th:weight_controlled_pullback} to both sides of  \eqref{eqn_smooth_pullback}; on the right hand side the theorem is applied to each component of the weight decomposition of  $d\al$ separately.  Note that the condition $\wt(d\al)<w_\al$ 
 ensures that Theorem~\ref{th:weight_controlled_pullback} can be applied. 
\end{proof}

\section{Quasiregular mappings}
\label{sec_quasiregular_mappings}
In this section we review some results from \cite{vodopyanov_foundations} which were stated only for $2$-step Carnot groups, but which extend immediately to general Carnot groups using the Approximation Theorem \ref{th:weight_controlled_pullback}.

In this section we fix a Carnot group $G$ of homogeneous dimension $\nu$, and an open subset $U\subset G$.

\begin{definition}[\cite{vodopyanov_foundations}]
A mapping $f:G\supset U\ra G$ is {\bf quasiregular} (has {\bf bounded distortion}) if   $f\in W^{1,\nu}_{\loc}$ and there is a constant $C$ such that $|D_hf|^\nu\leq C\det D_Pf$ almost everywhere.
\end{definition}

We now fix a quasiregular mapping $f:G\supset U\ra G$.  

Following Reshetnyak \cite{reshetnyak_space_mappings_bounded_distortion,heinonen_holopainen,vodopyanov_foundations}, we exploit the pullbacks of $\nu$-harmonic functions to control quasiregular mappings.
\begin{theorem}
\label{thm_morphism_property}
If $u:G\ra \R$ is a Lipschitz $\nu$-harmonic function, then the composition $u\circ f$ is $\mathcal{A}$-harmonic.  See \cite[Sec. 2]{heinonen_holopainen}, \cite[Subsec. 4.3]{vodopyanov_foundations} for the definition and basic properties of $\mathcal{A}$-harmonic functions.   
\end{theorem}
Note that if $f$ takes values in an open subset $U'\subset G$,  then the theorem holds when $u$ is locally Lipschitz, see below.
\begin{proof}
In the $2$-step case, the proof is  contained in \cite{vodopyanov_foundations}.   This extends to general Carnot groups using the Approximation Theorem.

We give an outline of the steps, to facilitate reading of \cite{vodopyanov_foundations}:
\bit
\item By Remark~\ref{re:codegree_coweight_1},   
if $\om$ is a smooth differential form on $G$ with codegree and coweight $1$, and both $\om$ and $d\om$ are bounded, then 
\begin{equation}
\label{eqn_d_commutes_with_pullback_on_codegree_1_coweight_1}
df_P^*\om=f_P^*d\om
\end{equation} 
distributionally.  
\item If $\Si\subset G$ is a Borel null set, then the (approximate) Pansu differential $D_Pf(x)$ satisfies $\det D_Pf(x) = 0$  for a.e. $x\in f^{-1}(\Si)$ \cite{vodopyanov_P_differentiability}; hence by the bounded distortion assumption in fact $D_Pf(x)=0$ 
  for a.e.\  $x\in f^{-1}(\Si)$.  If $\om$ is a measurable differential form on $G$, and we define $f_P^*\om(x)$ to be zero whenever $D_Pf(x)=0$, then the Pansu pullback $f_P^*\om$ is well-defined almost everywhere. 
\item By an approximation argument (\ref{eqn_d_commutes_with_pullback_on_codegree_1_coweight_1}) remains true if $\om\,,d\om\in L^\infty$, see \cite[Corollaries 2.15, 2.18]{vodopyanov_foundations}.
\item It follows  from Proposition~\ref{pr:composition_by_lip} that the composition $v:=u\circ f$ belongs to $W^{1,\nu}_{\loc}(U)$.
\item To see that $v$ is $\A$-harmonic, it suffices to show that its horizontal differential $d_hv$ satisfies the distributional equation $\de(\A d_hv)=0$ (cf. \cite[(2.11)]{heinonen_holopainen}).  This is equivalent to the vanishing of the distributional exterior derivative of $\star_\A d_hv$, see  \cite[Section 3, Theorem 3.14]{heinonen_holopainen}.  Since $\star_\A d_hv=f_P^*(\star d_hu)$, this follows from (\ref{eqn_d_commutes_with_pullback_on_codegree_1_coweight_1}).
\eit
\end{proof}

\bigskip\bigskip
The composition $u$ of the abelianization map $G\ra G/[G,G]$ with a coordinate function is Lipschitz and $\nu$-harmonic.  Hence by Theorem~\ref{thm_morphism_property} the composition $u\circ f$ is $\A$-harmonic.   Following \cite{reshetnyak_space_mappings_bounded_distortion,BI83,heinonen_holopainen,vodopyanov_foundations},  by applying the Caccioppoli inequality for $\A$-harmonic functions and the Poincare inequality one obtains a number of results, including:
\bit
\item $f\in W^{1,\nu'}$ for some $\nu'>\nu$.
\item $f$ is H\"older continuous, (classically) Pansu differentiable almost everywhere, and maps null sets to null sets.
\item A suitable change of variables formula holds for $f$.
\eit

Since $f$ is continuous, the proof of Theorem~\ref{thm_morphism_property} may be localized in the target:
\begin{corollary}
Suppose the image of $f$ is contained in an open subset $U'\subset G$, and $u:U'\ra\R$ is a locally Lipschitz $\nu$-harmonic function.  Then $u\circ f$ is $\A$-harmonic. 
\end{corollary}

If there exists for some $r>0$ a locally Lipschitz $\nu$-harmonic function $u:B(e,r)\setminus\{e\}\ra (0,\infty)$ such that $\lim_{x\ra e}u(x)\ra\infty$, then the method of Reshetnyak could be applied to show that $f$ is open and discrete, which would have a number of further consequences, see \cite{vodopyanov_foundations} (Theorem 4.11 and the ensuing discussion).  Unfortunately, the existence of such $\nu$-harmonic functions remains an open problem.

\section{Product rigidity}

In this section we show how the results in \cite{KMX1} on product rigiditiy can be  improved by using the improved version 
of the Pullback Theorem, Theorem~\ref{th:pansupullback_exterior_d} and a better choice of forms to be pulled back.

\begin{theorem}[Product rigidity]   \label{th:product_rigidity} 
Let $\{G_i\}_{1 \le i \le n}$, $\{G'_j\}_{1 \le j \le n'}$ be collections of Carnot groups where each $G_i, G'_j$ is nonabelian and does not admit
a nontrivial decomposition of Carnot groups. Let $G = \prod_i G_i$, $G' = \prod_j G'_j$. Set 
$$ K_i : = \{ k \in \{1, \ldots, n\} : G_k \simeq G_i \}$$
and  if $|K_i|\geq 2$ for some $i$, assume that 
\begin{equation}  \label{eq:p_rigidity}
 p \ge \max\{ \nu_i - 1:   |K_i| \ge 2 \}
 \end{equation}
 where $\nu_i$ denotes the homogeneous dimension of $G_i$.   
 Suppose that $f: G \supset U \to G'$ is a $W^{1,p}_{\loc}$-mapping, $U = \prod_i U_i$ is a product of open connected sets $U_i \subset G_i$,
 and the  (approximate) Pansu differential $D_P f(x)$ is an   
  isomorphism for a.e.\ $x\in U$. Then $f$ is a product of mappings, i.e. $n=n'$
 and for some permutation $\sigma:\{1, \ldots, n\} \to \{1, \ldots, n\}$ there are mappings
 $\{f_{\sigma(i)}:U_i \to G'_{\sigma(i)}\}_{1 \le i \le n}$
 such that 
 \begin{equation} \label{eq:product_mapping}
 f(x_1, \ldots, x_n) = \left(f_1(x_{\sigma^{-1}(1)}), \ldots, f_n(x_{\sigma^{-1}(n)})    \right)
 \end{equation}
 for a.e. $(x_1,\ldots,x_n)\in\prod_iU_i$.
\end{theorem}

For many groups the condition  \eqref{eq:p_rigidity} on $p$ can be improved. 

\begin{corollary} \label{co:product_rigidity}
Let $G_i, G'_j, U_i$  be as in  Theorem~\ref{th:product_rigidity}, assume  that $G_i$ is a group of step $m_i$ and set
$$ \bar m = \max \{ m_i : |K_i| \ge 2\}, \quad \bar \nu = \sum_{i : |K_i| \ge 2} \nu_i.$$
Assume further for all $i$ with  $|K_i| \ge 2$ the Lie algebra $\fg_i/ \oplus_{j=3}^{m_i} V_j$ is not a free Lie algebra. 
Then the conclusions of  Theorem~\ref{th:product_rigidity} hold provided  \eqref{eq:p_rigidity}
is replaced by the weaker conditions
\begin{equation} p \ge 2 \quad \text{and} \quad \frac1p  \le  \frac1{\bar m} + \frac{1}{\bar \nu}.
\end{equation} 
\end{corollary}
 For example,  the conclusion of Theorem~\ref{th:product_rigidity} holds for $p=2$ if  all $G_i$ are isomorphic to a higher Heisenberg group $\H_{d_i}$ with $d_i \ge 2$ or to a complex Heisenberg group $\H^\C_{d_j}$ with $d_j \ge 1$. On the other hand,  the assumptions are not satisfied if some $G_i$ is a copy of  the first Heisenberg group $\H_1$ which is a free Carnot group of step $2$. 

\medskip

We  use the  following result from \cite[Prop. 2.5]{Xie_Pacific2013}, see also \cite{KMX1}.
\begin{lemma}
\label{lem_isom_is_product}
Suppose $\fg=\oplus_{i\in I}\fg_i$, $\fg'=\oplus_{j\in I'}\fg_j'$ where every $\fg_i$, and $\fg'_j$ is nonabelian and does not admit a nontrivial decomposition as a direct sum of graded ideals.
Then any graded isomorphism $\phi:\fg\ra \fg'$ is a product of graded isomorphisms, i.e. there is a bijection $\si:I\ra I'$ and for every $i\in I$ there exists a graded isomorphism $\phi_i:\fg_i\ra \fg'_{\si(i)}$ such that for all $i\in I$ we have $\pi_{\si(i)}\circ\phi=\phi_i\circ\pi_i$.
\end{lemma}

We also use the following Fubini-type property of Sobolev maps.

\begin{lemma}   \label{le:fubini_sobolev}
Let $G_1$, $G_2$, and $G'$ be Carnot groups,  let $U_i \subset G_i$ be open sets,  let $\imath_i$ be the injections
$G_i \to G_1 \times G_2$ and let $\pi_i: G_1 \times G_2 \to G_i$ be the corresponding projections.
Let $f$ be the representative of a map in $W^{1,p}(U_1 \times U_2; G')$, let $D_P f$ be a representative
of the 
 (approximate) Pansu differential.
Then the following assertions hold.
\ben
\item For a.e.\  $a \in U_1$ the map $f_a: U_2  \to G'$ defined by $f_a(y) = f(a,y)$ is in the Sobolev space $W^{1,p}(U_2;G')$
and the 
Pansu differential  $D_P f_a$ satisfies 
\begin{equation} \label{eq:fubini_chain}
 D_P f_a  = D_P f(a, \cdot)  \circ \imath_2    \qquad \text{$\mu_{G_2}$-a.e.\ in $U_2$.}
\end{equation} 
\item If, in addition, $U_2$ is connected and  $D_P f \circ \imath_2 = 0$ a.e.\ then there exists a function $\bar f: U_1 \to G'$
such that  $f = \bar f \circ \pi_1$ a.e.
\item   If $G'=G_1\times G_2$, $U_i$ is connected, and for a.e. $x\in U$ we have 
$$
\pi_1\circ D_Pf(x)\circ\imath_2=0\,,\quad\text{and} \quad \pi_2\circ D_Pf(x)\circ\imath_1=0\,,
$$ 
then there exist mappings $\bar f_i:U_i\ra G_i$ such that $f(x_1,x_2)=(\bar f_1(x_1),\bar f_2(x_2))$ for a.e. $(x_1,x_2)\in U_1\times U_2$.
\een
\end{lemma}

\begin{proof} 

To prove (1), first note that for $G' = \R^N$ the Pansu differential (viewed as map on the Lie algebra) is given by $D_P f = D_h f \circ \Pi_1$
a.e., where $D_h f$ is the weak horizontal differential and $\Pi_1$ is the projection to the horizontal subspace. Thus, for $G' = \R^N$, assertion (1)
follows directly from the definition of the weak horizontal derivative and Fubini's theorem. 

For a general Carnot group $G'$ there exists a $p$-integrable function $g:U_1 \times U_2 \to \R$ such that
$|D_h d(z, f(\cdot))| \le g$ a.e. 
 Applying  the result for real-valued maps to the maps $(x,y) \mapsto d(z, f(x,y))$ for all $z$ in a countable
dense subset $D$ of $G'$ we easily conclude that  $|D_h d(z, f_a(\cdot)| \le g_a$ and hence  $f_a \in W^{1,p}(U_2, G')$ for a.e. $a \in U_1$. 

Let  $\Pi_{G'}$  denote the abelianization map. Then we can apply the result for $\R^N$-valued maps to $\Pi_{G'} \circ f$
and we get, for a.e $a \in U_1$, 
$$  \Pi_{G'} \circ D_P f_a = D_P (\Pi_{G'} \circ f_a) = \big(D_P (\Pi_{G'} \circ f)(a, \cdot) \big) \circ \imath_2 = \Pi_{G'} \circ (D_P f \circ \imath_2)(a, \cdot) 
$$   
$\mu_{G_2}$-a.e.\ in $U_2$.
Now if $\Phi, \Psi: G_2 \to G'$  are graded group homomorphism such that $\Pi_{G'} \circ \Phi = \Pi_{G'} \circ \Psi$ then $\Phi = \Psi$.
Hence  \eqref{eq:fubini_chain} holds.

 Assertion (2) is an immediate consequence of assertion (1).

Assertion (3) follows by applying assertion (2) to the compositions $\pi_i\circ f$. 

\end{proof}

\begin{proof}[Proof of  Theorem~\ref{th:product_rigidity}]
The result was established in \cite{KMX1} under the stronger condition $p > \nu(\Pi_i G_i) = \sum_i \nu_i$. We first briefly
recall the argument in \cite{KMX1} and indicate the strategy to obtain the improved condition
 \eqref{eq:p_rigidity}. 
 
 First, Lemma~\ref{lem_isom_is_product} implies that we may assume without loss of generality
 that 
 $n=n'$ and $\fg_i=\fg_i'$ for all $i\in I$, and so there
    is a measurable function $\si:U\ra \perm(\{1, \ldots, n\})$ such that $D_Pf(x)(\fg_i)=\fg_{\si(x)(i)}$ for a.e. $x\in U$. 
  Moreover $\sigma(i) = i$ if $|K_i| = 1$.  Hence it follows by applying  Lemma~\ref{le:fubini_sobolev}(3) repeatedly   that we may assume without loss of generality that  
  \begin{equation}  \label{eq:no_trivial_factors}
  |K_i| \ge 2 \quad \forall i=1, \ldots, n.
  \end{equation}
The main point is to show that there exists a constant permutation $\bar \sigma$ such that $\si = \bar \si$ a.e.
Then, using again Lemma~\ref{le:fubini_sobolev}, we see  $f$ has the desired product structure.

To prove that $\si$ is constant
a.e.\ we argue as follows. We choose closed left-invariant forms $\alpha$ 
such that the pullback $f_P^*\alpha$ can 'detect' the permutation $\sigma$. Then we use the Pullback Theorem to deduce that
for suitable left-invariant forms $\be$,   we have 
$d f_P^*\alpha \wedge \be = 0$ in the sense of distributions 
and conclude  that $\sigma$ is constant. In \cite{KMX1} we use for $\alpha$ the 
volume forms $\omega_i$ of the factors $G_i$. Then  Lemma~\ref{lem_isom_is_product}  implies that
$f_P^*\omega_i = \sum_{j \in K_i} a_j \omega_j$  
where the $a_j$ are integrable functions. Moreover, for each
$x$, exactly one of the functions $a_j$ is different from zero, namely $a_{\sigma^{-1}(i)}$ 
Let $j' \in \{1, \ldots, n\}$ and $l \ne j'$. For $X \in V_1 \cap \fg_{l}$  we apply the Pullback theorem with  the closed codegree $N_{j'} + 1$  
test  forms  
$\beta =i_X \omega_l \wedge   \bigwedge_{i' \notin\{j',l\}} \omega_{i'}$ and easily
conclude  that $a_{j'}(x_1, \ldots, x_n)$ depends only on $x_{j'}$. Then  one  easily deduces  that $\sigma = \bar \sigma$ almost everywhere.

The requirement  $p > \nu(\Pi_i G_i)$ in \cite{KMX1} comes from the hypotheses of the Pullback Theorem \cite[Theorem 4.5]{KMX1}. Using the improved Pullback Theorem, 
Theorem~\ref{th:pansupullback_exterior_d},  we see that this argument works
if $p \ge \max_i \nu_i$. 
The second condition in Theorem~\ref{th:pansupullback_exterior_d}, namely $\frac1p \le  \frac1m + \frac1\nu$ is then automatically
satisfied since 
\begin{equation} \label{eq:bound_m}
m \le \max_i \nu_i - 1.
\end{equation}
To get the improved condition 
 \eqref{eq:p_rigidity} we will apply the pullback theorem not to the volume forms of each factor, but to a codegree one form on each factor,
 with weight $-\nu_i +1$.

To set the notation,  we assign to  a form  $ \alpha \in \Omega^*(G_i)$ the form $\pi_i^* \alpha  \in \Omega^*(G)$ where  $\pi_i: G \to G_i$ is the projection map.
Note that $\pi_i^* \alpha$ is closed if and only if $\alpha$ is closed. We  will usually write  $\alpha$ also for the form  $\pi_i^* \alpha$ if no confusion can occur. 
Similar we identify a vector field $X \in TG_i$ with a vectorfield in $TG$ through the push-forward by the canonical injection $G_i \to G$.

For $i\in \{1, \ldots, n\}$ let  $\vol_{G_i}$ denote the volume form in $G_i$ and  let $Y \in V_1(G_i) \setminus \{0\}$. 
By Cartan's formula and the biinvariance of $\vol_{G_i}$ the $N_i-1$ form
$i_Y \vol_{G_i}$ is closed.  Let  $\alpha_i = i_Y \vol_{G_i}$.  
Then the $N_i-1$ form  $\alpha_i$ is left-invariant,  closed and has weight $-\nu_i +1$.

In view of   \eqref{eq:no_trivial_factors} we have  $p \ge \nu_{i}-1$. 
For $j \in K_i$ let  $X_{j,k}$, $k=1, \ldots \dim V_1(G_{i})$  be  a  basis of $V_1(G_{j})$. 
Then $ i_{X_{j,k}}\vol_{G_j}$, $k=1, \ldots \dim V_1(G_{i})$ is a basis of the left-invariant  forms
on $G_{j}$ with degree $N_i-1$ and weight $-\nu_i +1$. 
 Since pullback by a graded isomorphism preserves degree and weight we have
$$f_P^*\alpha_{i} =  \sum_{j \in K_i} \sum_{k=1}^{\dim V_1 (G_{i})}  a_{j,k}   \,  i_{X_{j,k}} \vol_{G_j}$$
with $a_{j,k} \in L^1_{\loc}(U)$. 
Set 
$$ a_j :=(a_{j,1},   \ldots, a_{j,  \dim V_1(G_{i})})$$
and  
 $$ E_j = \{ x \in U : \sigma^{-1}(x) (i)  = j \}.$$   
Then $U \setminus \bigcup_{j  \in  K_i} E_j$ is a null set. Since $D_P f(x)$ is a graded automorphism for a.e.\ $x \in U$,  we have, 
for all $j \in K_i$, 
\begin{equation}   \label{eq:disjoint_pullback}
a_j \ne 0 \quad \text{a.e.\ in $E_j$}, \qquad a_j = 0 \quad \text{a.e.\ in $U \setminus E_j$}.
\end{equation}

We next show that
\begin{equation}  \label{eq:product_distributional_derivative}
Z a_{j'} = 0 \quad \text{in  distributions for all $ j' \in K_i, \,  Z \in \oplus_{l \ne j'} \fg_{l}$.}
\end{equation}
To prove \eqref{eq:product_distributional_derivative}, 
let $\theta_{j',k'}$ be basis of left-invariant one-forms which vanish on  
   $\oplus_{l = 2}^s V_l(G_{j'})$ 
which is dual to the basis $X_{j',k}$ of $V_{1}(\fg_{j'})$ 
, i.e. $\theta_{j',k'}(X_{j',k}) = \delta_{k k'}$.
Note that the forms $\theta_{j',k'}$ are closed.
For  $l  \in \{1, \ldots, n\}  \setminus \{j'\}$, and  $X \in V_1(G_{l})$ consider the closed form
$$ \beta =  \theta_{j',k'}   \wedge i_{X} \vol_{G_{l}} \wedge (\La_{i'\neq j',l} \vol_{G_{i'}}).  
$$
Then, for a.e.\ $x \in U$, 
$$ (D_Pf)^*(x) \alpha_i \wedge \beta =  \pm a_{j',k'}  \,  i_X \vol_G.$$

In view of   \eqref{eq:bound_m}
and the assumption    $p \ge \nu_{i} - 1$
(recall that we may assume  \eqref{eq:no_trivial_factors})
 we get from
the Pullback Theorem, 
Theorem~\ref{th:pansupullback_exterior_d},
$$
0 = \int_U f_P^*\alpha \wedge \beta \wedge d\varphi =\pm  \int_U a_{j',k'}   \, \,  X\varphi \,  \,   \vol_G
$$
for all $\varphi \in C_c^\infty(U)$. 
Since $V_1\cap \fg_{l}$ generates $\fg_{l}$ as a Lie algebra, we
see that \eqref{eq:product_distributional_derivative} holds.

It follows from \eqref{eq:product_distributional_derivative} that $a_j(x) = a_j(x_j)$. Thus  \eqref{eq:disjoint_pullback} implies that
  $\chi_{E_j}(x)  = \chi_{E_j}(x_j)$ for all $j \in K_i$.
  Since $\sum_{j \in K_i} \chi_{E_j} = 1$ a.e.\ there exists one $j_0$ such that 
  $\chi_{E_{j_0}} = 1$ almost everywhere. Thus $\sigma^{-1}(i) = j_0$ almost everywhere. 
  Summarizing, we have shown that for all $i$  the function $\sigma^{-1}(x) (i)$ is constant almost everywhere.
  Hence $\sigma$ is constant almost everywhere.
 \end{proof}

 \bigskip
 
 The proof of   Corollary~\ref{co:product_rigidity} is very similar to the proof of Theorem~\ref{th:product_rigidity}.
 The main change is that instead of the closed  codegree one forms $i_{Y} \vol_{G_i}$ of weight $-\nu_i + 1$ we pull back certain closed
 two-forms of weight $-2$ in $G_i$. To identify suitable two-forms we use the setting in \cite{kmx_rumin}.
 For a Carnot  algebra $\fg = \oplus_{i=1}^s V_i$,  let $\Lambda^{1}_v\fg$ denote the space of one-forms which vanish on
 the first layer $V_1$  
 and let
 $I^*\fg \subset \Lambda^*\fg$ be the differential ideal generated  by $\Lambda^{1}_v\fg$.
 Thus 
 \begin{equation}
 I^*\fg = \Span \{   \alpha \wedge \tau + \beta \wedge d\eta : \alpha, \beta \in \Lambda^*\fg, \tau, \eta \in \Lambda^{1}_v\fg      \}. 
 \end{equation}
 The set of $k$-forms in $I^*\fg$            is denoted by    $I^k\fg = I^*\fg \cap \Lambda^k\fg$.
 We define $J^*\fg$ to be the annihilator $\ann(I^*\fg)$ of $I^*\fg$, i.e.,
 \begin{equation}
 J^*\fg = \{ \alpha \in \Lambda^*\fg : \alpha \wedge \beta = 0 \quad  \forall \beta \in I^*\fg \}.
 \end{equation}

We will use the following facts about $I^*\fg$ and $J^*\fg$ which easily follow from  exterior algebra  and the
formula for $d\alpha$. For the convenience of the reader we  include a  proof after the proof of 
Corollary~\ref{co:product_rigidity}.

\begin{proposition}  \label{pr:IJ} Let $\fg = \oplus_{i=1}^s V_i$ be a Carnot algebra of dimension $N$, homogeneous dimension $\nu$
and step $s \ge 2$. Let $G$ be the corresponding Carnot group.
Then the following assertions hold.
\ben 
\item \label{it:jnmk} 
For all $0\leq k\leq N$ we have 
$$
J^{N-k}\fg=\ann(I^k\fg)\cap \La^{N-k}\fg
=\{\al\in \La^{N-k}\fg\mid \al\we\be=0\,,\; \forall\be\in I^k\fg\}.
$$
\item  \label{it:IJ1}  If $I^k\fg = \Lambda^k\fg$ then $J^k\fg = \{0\}$. If $I^k\fg \ne \Lambda^k\fg$ then the wedge product induces a nondegenerate pairing
$$ \Lambda^k\fg/ I^k\fg \times J^{N-k}\fg \overset{\wedge}{\longrightarrow} \Lambda^N \fg \simeq \R.$$
In particular,  $\dim \Lambda^k\fg/ I^k\fg = \dim J^{N-k}\fg$ 
and for each basis $\{ \tilde \alpha_i\} $  of $ \Lambda^k\fg/ I^k\fg$ there exists a dual basis $ \{ \gamma_j\}$
of $J^{N-k}\fg$ such that $\tilde \alpha_i \wedge \gamma_j = \delta_{ij} \vol_G$.
\item  \label{it:IJ2} $J^*\fg$ is a differential ideal, i.e. $\alpha \in J^k\fg \Longrightarrow d\alpha \in J^{k+1}\fg$.
\item \label{it:IJ3} If $\fg'$ is another Carnot algebra and $\Phi: \fg \to \fg'$ is a graded isomorphism then $\Phi^*(I^k\fg') = I^k\fg$
and $\Phi^*$ induces an isomorphism from $\Lambda^k \fg'/ I^k\fg'$ to $\Lambda^k\fg/ I^k\fg$. 
\item \label{it:IJ4} If $\gamma  \in J^{k}\fg \setminus \{0\}$ then $\gamma$ is homogeneous  with coweight equal to its codegree, i.e. $\wt(\al)=N-k-\nu$. 
\item  \label{it:IJ5} If $\gamma \in J^k\fg$ then $d\gamma = 0$.
\een
\end{proposition}

 The main new ingredient in the proof of Corollary~\ref{co:product_rigidity} is the following simple observation.
 For a Carnot algebra $\fg = \oplus_{i=1}^s V_i$ with $s \ge 3$ let $\pi_{1,2}$ denote the projection
 $\fg \to V_1 \oplus V_2$.  We define $ \tilde \fg := \fg/ \oplus_{j=3}^s V_j$ as the algebra $V_1 \oplus V_2$ with
 bracket $[X,Y]_\sim = \pi_{1,2}[X,Y]$. Then $\tilde \fg$ is a Carnot algebra. If $s=2$, we set $\tilde \fg = \fg$.
 
 \begin{proposition} \label{pr:good_two_forms}
 If $\fg/ \oplus_{j=3}^s V_j$ is not a free Carnot algebra then 
 $I^2(\fg) \ne \Lambda^2(\fg)$. Moreover
 \begin{equation}  \label{eq:I2_horizontal}
 \Lambda^2\fg/ I^2\fg \simeq  (\Lambda^1_h \fg \wedge \Lambda^1_h \fg )/ (I^2\fg \cap (\Lambda^1_h \fg \wedge \Lambda^1_h \fg ))
 \end{equation}
 where $\Lambda^1_h\fg$ denotes the space of horizontal one-forms, i.e. one-forms which vanish on $\oplus_{i=2}^s V_i$. 
\end{proposition}

\begin{proof} Since $\Lambda^1\fg = \Lambda^1_h \fg \oplus \Lambda^1_v\fg$ we have
$$ \Lambda^2 \fg = (\Lambda^1_h \fg \wedge \Lambda^1_h \fg)  \oplus (\Lambda^1_v \fg \wedge \Lambda^1_h \fg)  \oplus 
(\Lambda^1_v \fg \wedge \Lambda^1_v \fg ).$$
Since the second and third summand on the right hand side are contained in $I^2\fg$ we get  \eqref{eq:I2_horizontal}.

Now assume that $\fg/ \oplus_{j=3}^s V_j = V_1 \oplus V_2$ is not a free Carnot algebra. Let $\{X_i\}_{1 \le i \le \dim V_1}$
be a basis of $V_1$.  Then there exist  coefficients
$\{a_{i,j}\}_{1 \le i < j \le \dim V_1}$ which do not all vanish such that 
$$ \sum_{1 \le i < j \le \dim V_1} a_{i,j} [X_i, X_j]_\sim = 0.$$
Since the $a_{i,j}$ do not all vanish there exists $\gamma \in \Lambda^1_h \wedge \Lambda^1_h$ such that
$$  \gamma(\sum_{i < j} a_{i,j} X_i \wedge X_j) \ne 0.$$
We claim that $\gamma \notin I^2\fg$. Otherwise there exist $\alpha \in \Lambda^1$ and $\tau, \eta \in \Lambda^1_v$ such that
\begin{align*}
0 \ne  \,( \alpha \wedge \tau + d\eta)\left(\sum_{i < j} a_{i,j} X_i \wedge X_j\right) 
=  - \eta\left( \sum_{i < j} a_{i,j} [X_i, X_j]_\sim \right) = 0.
\end{align*}
This contradiction concludes the proof.
\end{proof}

\bigskip
\begin{proof}[Proof of Corollary~\ref{co:product_rigidity}]
  Again we may assume without loss of   generality that $|K_i| \ge 2$ for all $i$, that $n=n'$ and that $\fg'_i = \fg_i$.
  Then $G'$ is a step $\bar m$ group.

  By Propositions~\ref{pr:IJ} and~\ref{pr:good_two_forms} there exist horizontal two-forms  $\alpha_{j,k} \in \Lambda_h^1 \wedge \Lambda_h^1$ 
  such that $\alpha_{j,k} + I^2\fg_j$ is a basis  of $\Lambda^2\fg_j/ I^2\fg_j$ and there exist  dual bases
  $\gamma_{j,k'}$ of $J^{N_j-2}\fg_j$ such that
  \begin{equation}   \alpha_{j,k} \wedge \gamma_{j,k'} = \delta_{kk'} \vol_{G_j}.
  \end{equation}
  Note also that forms in  $ \Lambda_h^1 \wedge \Lambda_h^1$ are closed since forms in $\Lambda_h^1$ are closed.
 
 Now we can proceed as in the proof of Theorem~\ref{th:product_rigidity}.
 Let $i \in \{1, \ldots, n\}$. Then, for a.e.\ $x \in U$ we have  
 $(D_P f)^*(x) \alpha_{i,1} \in \Lambda^1_h(\fg_{\sigma^{-1}(i)}) \wedge \Lambda^1_h(\fg_{\sigma^{-1}(i)})$.
 Thus
$$
f_P^*\alpha_{i,1}
 = \sum_{j \in K_i}  \sum_k  (a_{j,k}   \alpha_{j,k} + \beta_k) \,  $$
with $a_{j,k} \in L^1_{\loc}(U)$ and  $\beta_k \in L^1_{\loc}(U; I^2\fg_j)$. 
Set
$$ a_j :=(a_{j,1},   \ldots, a_{j,  \dim( \Lambda^2(\fg_j)/ I^2\fg_j)})$$
and
 $$ E_j = \{ x \in U : (\sigma(x))^{-1} (i)  = j \}.$$
Then $U \setminus \bigcup_{j  \in  K_i} E_j$ is a null set. 
By Proposition~\ref{pr:IJ}~\eqref{it:IJ3} we have, for a.e. $x \in U$,  $(D_P f)^*(x) \alpha_{i,1} \notin I^2\fg_{\sigma^{-1}(i)}.$    
Thus,
for all $j \in K_i$, 
\begin{equation}   \label{eq:disjoint_pullback_bis}
a_j \ne 0 \quad \text{a.e.\ in $E_j$}, \qquad a_j = 0 \quad \text{a.e.\ in $U \setminus E_j$}.
\end{equation}

We next show that  
\begin{equation}  \label{eq:product_distributional_derivative_bis}
Z a_j = 0 \quad \text{in  distributions for all  $j\in K_i,  \, Z \in \oplus_{j' \ne j} \fg_{j'}$.}
\end{equation}

To prove \eqref{eq:product_distributional_derivative_bis}, let 
 $j' \in K_i$, $l  \in \{1, \ldots, n\}  \setminus \{j'\}$,  $X \in V_1(G_{l})$, and  consider the closed form
$$ \beta =   \gamma_{j',k'} \wedge i_{X} \vol_{G_{l}}      \wedge \La_{i'\neq j',l} \vol_{G_{i'}}.
$$
Then $\beta$ is a closed form of degree $N-3$ and weight $-\nu + 3$, where $N$ and $\nu$ are the topological
and homogeneous dimension of $G$, respectively.
Moreover
$$ \alpha_{j,k} \wedge \beta =  \pm \delta_{jj'}  \, \delta_{kk'}  \,  i_X  \vol_G.$$
Finally, 
$\alpha_{i,1}$ is a closed left-invariant two-form of weight $-2$.   Thus the conditions on $p$ in Corollary~\ref{co:product_rigidity}
allow us to apply the 
 Pullback Theorem, 
Theorem~\ref{th:pansupullback_exterior_d},  and we get
$$
0 = \int_U f_P^*\alpha_{i,1} \wedge \beta \wedge d\varphi =\pm  \int_U a_{j',k'}   \, \,  X\varphi\,  \,   \vol_{G}.
$$
for all $\varphi \in C_c^\infty(U)$, all $j' \in K_i$, and all  $k'$. 
As in the proof  Theorem~\ref{th:product_rigidity}  we conclude that $a_{j,k}(x_1, \ldots, x_n)$ depends only on $x_j$
and that $\sigma$ is constant almost everywhere.
\end{proof}

\medskip

Looking back, we see that the arguments in the proofs of Theorem~\ref{th:product_rigidity} and 
Corollary~\ref{co:product_rigidity} are exactly analogous. The only difference is that we use different
forms in the ideals $I^*$ and $J^*$ as forms to be pulled back and as test forms.
 Indeed, note that $I^1\fg_i = \Lambda^1_v\fg_i$  
and thus $\Lambda^1\fg_i/I^1\fg_i \simeq \Lambda^1_h\fg_i$ and $J^1\fg_i = \{ i_X \vol_{G_i} : X \in \Lambda^1_h\fg_i\}$.
Thus in the proof of  Theorem~\ref{th:product_rigidity} we pull back a form in $J^1\fg_i$ and use test forms of the type 
  $\Lambda^1\fg_{j'}/I^1\fg_{j'}    \wedge  i_X \vol_{G_l} \wedge  \La_{i' \ne \{j',l\}} \vol_{G_{i'}}$, 
  while for the proof 
  of Corollary~\ref{co:product_rigidity} we pull back forms in $\Lambda^2\fg_i/I^2\fg_i$ and use test forms of the type
  $J^2\fg_{j'}  \wedge i_X \vol_{G_l} \wedge  \La_{i' \ne \{j',l\}} \vol_{G_{i'}}$.

\bigskip

We finally provide a proof of Proposition~\ref{pr:IJ} for the convenience of the reader.

\begin{proof}[Proof of Proposition~\ref{pr:IJ}]

\eqref{it:jnmk} 
Since $I^k\fg\subset I^*\fg$ we have $\ann(I^k\fg)\cap \La^{N-k}\fg\supseteq \ann(I^*\fg)\cap \La^{N-k}\fg$.  To establish the opposite inclusion, choose $\al\in \ann(I^k\fg)\cap \La^{N-k}\fg$, and $\be\in I^j\fg$ for some $0\leq j\leq N$.  
If $k<j\leq N$ we have $\al\we\be=0$ since $\deg\al+\deg\be>N$.   If $0\leq j\leq k$ and  $\ga\in \La^{k-j}\fg$ we have $
\be\we\ga\in I^k\fg$ since $I^*\fg$ is an ideal, so 
\begin{equation}
\label{eqn_al_be_ga}
(\al\we\be)\we\ga=\al\we(\be\we\ga)=0
\end{equation}
because $\al\in\ann(I^k\fg)$ by assumption.   The pairing  $\La^{N-(k-j)}\fg\times\La^{k-j}\fg\ra \La^N\fg$ is nondegenerate and $\ga\in \La^{k-j}\fg$ was arbitrary, so \eqref{eqn_al_be_ga} implies that $\al\we\be=0$.  Since $\be \in \La^j\fg$ was arbitrary, we have $\al\in \ann(I^j\fg)$.  Since $j$ was arbitrary we get $\al\in \cap_{0\leq j\leq N}\ann(I^j\fg)=\ann(I^*\fg)$.

\eqref{it:IJ1}  
Since the wedge product $\La^k\fg\times\La^{N-k}\fg\ra \La^N\fg\simeq \R$ is a nondegenerate pairing,  \eqref{it:IJ1} follows from \eqref{it:jnmk} and the fact that for any nondegenerate pairing $E\times E'\stackrel{b}{\ra}\R$ of finite dimensional vector spaces and any subspace $W\subset E$, there is a  nondegenerate pairing
$(E/W)\times W^\perp\ra \R$ induced by $b$, where $W^\perp:=\{e'\in E'\mid b(e,e')=0\,,\;\forall e\in W\}$.

\eqref{it:IJ2} This follows from the fact that $I$ is a differential ideal and the graded Leibniz rule.

\eqref{it:IJ3} This follows from the facts that $\Phi^*(\Lambda^1_v\fg') = \Lambda^1_v\fg$ and that $d$ commutes
with pullback by $\Phi$. 

\eqref{it:IJ4} Let $\{X_{i,j}\}$, $i = 1, \ldots, s$, $j = 1, \ldots, \dim V_i$  be a graded basis of $\fg$, i.e., $X_{i,j} \in V_i$.
Let $\theta_{i,j}$ be the dual basis of one-forms. Then  the forms $\theta_{i,j}$ are homogeneous with  $\wt(\theta_{i,j}) = -i$.
Moreover
\begin{align*} \Lambda^1_h\fg = & \,  \Span \{ \theta_{1,j} : 1 \le j \le \dim V_1\}, \\
 \Lambda^1_v\fg =  & \, \Span \{ \theta_{i,j} : i \ge 2, 1 \le j \le \dim V_i \}.
 \end{align*}
 Set $\tau = \La_{i \ge 2, 1 \le j \le \dim V_i} \th_{i,j}$.   
  It is easy to see that for $k \ge N - \dim V_1$
 every $\gamma \in J^k\fg$ is of the form
 $$ \gamma = \alpha \wedge \tau, \quad \text{with $\alpha \in \La^{k - (N- \dim V_1)} (\Lambda_h^1\fg)$,}$$
 and $J^k\fg = \{0\}$ if $k < N- \dim V_1$. Thus every non-zero element of $J^k\fg$ is a homogeneous form
 with weight $-\nu + (N-k)$.
 
 \eqref{it:IJ5} Let $\gamma \in J^k\fg$ and assume that $d\gamma \ne 0$. By properties \eqref{it:IJ2} and \eqref{it:IJ4} the form $\gamma$ is homogeneous and 
 \begin{equation} \label{eq:weights_dJk} \wt(d\gamma) = \wt(\gamma) - 1.
 \end{equation}
  On the other hand for $\gamma \in \Lambda^k\fg$ we have
 $$ d\gamma(X_0,  \ldots, X_k) =  \sum_{0 \le i < j \le k} \gamma([X_i, X_j], X_0, \ldots, \hat X_i, \ldots, \hat X_j, \ldots, X_k) $$   
where $\hat X_i$ denotes that $X_i$  is omitted (see, for example, \cite[Lemma 14.14]{Michor}).
 If $\gamma$ is homogeneous  and $d\gamma \ne 0$ it follows that $d\gamma$ is homogeneous and
 $\wt(d\gamma) = \wt(\gamma)$. This contradicts   \eqref{eq:weights_dJk}. Hence $d\gamma = 0$. 
 \end{proof}

\section{Complexified Carnot algebras}

In \cite{KMX1} it was shown that under suitable conditions nondegenerate Sobolev maps of a complexified Carnot group are
automatically  holomorphic or antiholomorphic. In this section show these results with improved conditions on the Sobolev exponent.

We first recall the setting in \cite{KMX1} to which we refer for further details.
Let $H$ be a Carnot group of topological dimension $N$ and homogeneous dimension $\nu$. Let $\fh$ be the  corresponding  Carnot algebra. 
Let $\fg$ denote  the complexified Carnot algebra, i.e. $\fg = \fh^\C$
 equipped with the grading $\fg=\oplus_jV_j^\C$.
  The corresponding Carnot group $G$ has topological dimension 
  $2N$ and homogeneous dimension $2\nu$. 
   We now denote by $J$ the almost complex structure on $G$ 
   coming from $\fg$;  it follows from the Baker-Campbell-Hausdorff formula that $J$ is integrable, i.e.\ $(G,J)$ is a complex manifold, and the group operations are holomorphic.  
   Also, complex conjugation $\fg\ra \fg$ is induced by a unique graded automorphism $G\ra G$, since $G$ is simply-connected.

\begin{theorem}  \label{th:holomorphic} Let $U \subset G$ be a connected open subset, let $p \ge \nu$, let $f \in W^{1,p}_{\loc}(U, G)$ and assume that 
(approximate) Pansu differential $D_P f (x)$  is either  a $J$-linear 
 isomorphism or a $J$-antilinear graded isomorphism for a.e.\ 
$x \in U$. Then $f$ is holomorphic or antiholomorphic (with respect to the complex structure $J$).
\end{theorem}

\begin{corollary}    \label{co:holomorphic} Let $U \subset G$ be open, let $p \ge \nu$, let $f \in W^{1,p}_{\loc}(U, G)$.
  Suppose that any graded isomorphism $\fg \to \fg$ is either $J$-linear or $J$-antilinear and that $D_P f(x)$ is an isomorphism
  for a.e.\ $x \in U$.  Then $f$ is holomorphic or antiholomorphic (with respect to the complex structure $J$).
\end{corollary}

The condition that any graded isomorphism $\fg \to \fg$ is either $J$-linear or $J$-antilinear is in particular satisfied 
for the complexified Heisenberg algebras $\fh^\C_m$, see \cite[Section 6]{reimann_ricci}.

\begin{proof}[Proof of Theorem~\ref{th:holomorphic}] The result is proved in \cite{KMX1} under the stron\-ger condition
$p >  \text{homogeneous dimension of $G$} = 2 \nu$.
 The key step in the proof is to show that $D_P f$ cannot switch 
between a $J$-linear and a $J$-antilinear map. 
To show this we use the Pullback Theorem to prove  that the the pullback of
the top degree holomorphic form cannot oscillate between a holomorphic and an anti-holomorphic form.
Since the top degree holomorphic form has weight $-\nu$,  
 the improved version of the Pullback Theorem,  Theorem~\ref{th:pansupullback_exterior_d},
 gives this result already under the weaker condition $p \ge \nu$.
Note that $G$ is a group of step $m$ with $m < \nu$ so the condition $\frac1p \le \frac1m + \frac1{\bar \nu}$, where $\bar \nu = 2 \nu$ is the 
homogeneous dimension of $G = H^\C$,  is automatically satisfied.

Thus, under the assumption $p \ge \nu$, we still may asssume that $D_P f$ is $J$-linear a.e.\  (the case
that $D_P f$ is $J$-antilinear a.e.\  being analogous). Let $\pi_{G} : G \to G/[G,G]$ denote the abelianization map.
By Remark~\ref{re:weak_derivative_abel},  the map $\pi_{G} \circ f$ belongs to $W^{1,p}_{\loc}$ and 
for each horizontal vectorfield  $X$ the weak derivative satisfies  $X (\pi_{G} \circ f)(x) = D_P f(x) X$
for a.e.\ $x \in U$. Since $D_P f(x)$ is $J$-linear,  the horizontal anticonformal derivatives
$\bar Z (\pi_{G} \circ f)$ vanish. It follows easily that $\pi_{G} \circ f$ is holomorphic (see, e.g., \cite{KMX1}).
In particular $\pi_{G} \circ f$ is smooth and hence $|D_P f(x)| := \max \{ |D_P f(x) X|_{\fg} : |X|_{\fg} \le1, X \in V_1\}$
is locally bounded. By  \eqref{eq:optimal_g},  it follows that $f \in W^{1, \infty}_{\loc}(U;G)$. 
Thus the assertion follows from the result in  \cite{KMX1} for $p > 2 \nu$.

\end{proof}

\bigskip
\begin{proof}[Proof of Corollary~\ref{co:holomorphic}] This follows immediately from Theorem~\ref{th:holomorphic}.
\end{proof}

\appendix

\section{$L^{p*}$-Pansu-differentiability}
\label{se:W1p_differentiability}

Here we give another  proof of the  the following result by Vodopyanov.

\begin{theorem}[$L^{p*}$ Pansu differentiability a.e.,  \cite{vodopyanov_differentiability_2003}, Corollary 2]
\label{th:pansu_diff_Lpstar}
 Let $ U \subset G$ be open, let $1 \le p < \nu$ 
and define $p^*$ by $\frac1{p^*} = \frac1p - \frac1\nu$.
Let $f \in W^{1,p}(U;G')$. For $x \in U$ consider the rescaled maps
$$ f_{x,r} = \delta_{r^{-1}} \circ \ell_{f(x)^{-1}} \circ f \circ \ell_x \circ \delta_r.$$
Then, for a.e. $x \in U$,  there exists  a graded group homomorphism $\Phi: G \to G'$ such that
\begin{equation}  \label{eq:Lpstart_pansu_differentiability}
 f_{x,r} \to \Phi  \quad \text{in $L^{p^*}_{loc}(G;G')$ as $r \to 0$.} 
 \end{equation}
 \end{theorem}

 We write $$D_P f(x) = \Phi$$
 for the $L^{p*}$ Pansu derivative
 and we use the same notation to denote the corresponding graded  Lie algebra homomorphism given by $D\Phi(e)$.

 \bigskip
 \begin{remark}  \label{re:g_equal_norm_DPf_app}
  It follows easily from  Theorem~\ref{th:pansu_diff_Lpstar}  that for all $z \in G'$ 
 the functions $u_z:= d'(z, f(\cdot))$ satisfy 
 \begin{equation}  \label{eq:optimal_g_app}
 \text{$|D_h u_z|  \le |D_P f|$  a.e., where} 
  \end{equation}
 \begin{equation}
 |D_P f(x)|  = \max \{ |D_P f(x) X|_{V'_1} : \, X \in V_1, \, |X|_{V_1} \le 1 \}.
 \end{equation}
Here  $| \cdot |_{V_1}$ and $| \cdot |_{V'_1}$ denote the norms induced by the scalar product on the first layer of $\fg$ and $\fg'$,
respectively. Thus the condition   \eqref{eq:bound_D_h_metric}  
in Definition~\ref{de:sobolev_carnot_new} holds with $g = |D_P f|$. 
We provide a proof after the proof of Theorem~\ref{th:pansu_diff_Lpstar}.
 \end{remark}

\bigskip
\begin{remark}
\label{rem_step_2_of_proof}
It follows from Step 2 of the proof of Theorem~\ref{th:pansu_diff_Lpstar} that $f_{x,r}$ converges to $\Phi$ in the stronger sense that the horizontal distributional 
derivatives of the composition $\pi_{G'}\circ f_{x,r}$ converge in $L^p_{\loc}$ to a constant function, where $\pi_{G'}:G'\ra G'/[G',G']$ is the abelianization map; more precisely, for a.e. $x\in U$, 
and  a basis  $X_i$ of the space of left-invariant  horizontal vectorfields
\begin{equation}
X_i(\pi_{G'}\circ f_{x,r})\lra g_i(x) \quad \text{in $L^p_{loc}$} \quad \text{where $g_i =  X_i(\pi_{G'}\circ f)$}
\end{equation}
is the weak horizontal derivative of $\pi_{G'} \circ  f$  in direction $X_i$. It follows from  \eqref{eq:Lpstart_pansu_differentiability} that  $\pi_{G'}\circ f_{x,r} \to \pi_{G'} \circ \Phi$ in 
$L^{p}_{loc}$ and thus 
$$ g_i(x) = X_i (\pi_{G'} \circ \Phi) \quad \text{on $G$}  $$
Evaluating the right hand side at $e \in G$ and viewing $D_P f(x)$ as a map from $\fg$ to $\fg'$ given by $D\Phi(e)$
we get 
$$ g_i(x) =  D_P f(x) X_i.$$
Here we used the fact that $D_P f(x)$ maps the first layer of $\fg$ to the first layer of $\fg'$ and 
that $D \pi_{G'}(e)$ is the identity
map on the first layer of $\fg'$. 
It follows that for every  horizontal vectorfield $X$ on $G$ the
weak derivative of $ (\pi_{G'} \circ f)$ in direction $X$ satisfies
\begin{equation}  \label{weak_derivative_abel_pansu}
 X(\pi_{G'}\circ f)(x) =  D_P f(x) X \quad \text{for a.e.\  $x$.}
\end{equation}
\end{remark}

\bigskip\bigskip

We give the proof using the distributional definition of Sobolev spaces,  see Definitions~\ref{de:sobolev_scalar_new} 
and~\ref{de:sobolev_carnot_new}.
The  proof uses  only the Poincar\'e-Sobolev inequality (which easily implies compactness of the Sobolev
 embedding, see Appendix~\ref{se:app_compact}) and a characterization of group homomorphism by their abelianization 
 (see Lemma~\ref{le:lift_of_L_maps} below).  Thus 
  the proof applies verbatim if one instead uses the definition of Sobolev spaces by upper gradients, if one replaces the function 
  $g$ in Definition~\ref{de:sobolev_carnot_new} by an upper gradient;  see \cite[Theorem 9.1.15]{HKST} for the Poincar\'e-Sobolev inequality
  for Sobolev spaces defined by upper gradients. For the convenience of the reader a short discussion of Sobolev spaces defined
  by upper gradient is given in Appendix~\ref{se:sobolev}.

 The strategy of the proof is as follows. Denote by $\pi_{G'} : G' \to G'/[G', G'] \simeq V_1$ the abelianization map.
 By Proposition~\ref{pr:composition_by_lip} we have $\pi_{G'}   \circ f \in W^{1,p}(U;G'/[G', G'])$. 
 Let $x$ be a Lebesgue point of $f$,  the function  $g$ in Definition~\ref{de:sobolev_carnot_new} and of  the weak derivatives $g_i = X_i (\pi_{G'} \circ f)$. 
 
 \ben
 \item By the compact Sobolev embedding, Theorem \ref{th:compact_sobolev_app}~\eqref{it:compact_sobolev_app2},
   a subsequence $f_{x,r_j}$ converges to 
 a Lipschitz map
 $\hat f$ in $L^{p^*}_{loc}$ with $\hat f(e) = e$.
 \item The whole sequence $\pi_{G'} \circ f_{x,r}$ converges to a linear map $\hat u$, 
 i.e., the weak horizontal derivatives of $\hat u$ are constant.
 \item If $F: G \to G'$ is Lipschitz, $F(e) = e$  and $\pi_{G'} \circ  F$ is a linear map, then 
 $F$ is  a group homorphism and $F$ is uniquely determined by $\pi_{G'} \circ F$; see Lemma~\ref{le:lift_of_L_maps} below.
 \item Uniqueness of the limit implies that the full sequence $f_{x,r}$ converges in $L^{p*}_{loc}$. 
 \een

 We begin by recalling some properties  of  the abelianization map.  
 
 \begin{proposition}   \label{pr:abelianization}
 Let $G$ be a Carnot group,  equipped with the Carnot-Caratheodory distance,  with graded Lie algebra $\fg = \oplus_{i=1}^m \fg_i$.
  Then the abelianization  
     homomorphism $\pi=\pi_G : G \to G/[G,G]$ has the following properties.
  \ben 
  \item  \label{it:abelianization1}  The map $\pi$ is graded, i.e., $\pi(\delta_r g) = r \pi(g)$, and the restriction of $d\pi(g)$
  to the horizontal subspace of $T_g G$ is an isomorphism onto $G/ [G,G]$. Thus for every 
  $Y \in G/[G,G]$ there exists a unique horizontal vectorfield $Z$ on $G$ with $d \pi(g) Z(g) = Y$.
 Moreover $Z$ is left-invariant.
  \item   \label{it:abelianization2}  The map $\pi$ is $1$-Lipschitz if $G/[G,G]$ is equipped with the induced metric 
  $d_\pi(\pi(a), \pi(b)) := \min_{g, g' \in [G,G]} d_{CC, G}(ag,bg')$. Moreover, the induced metric
  comes from a  norm: $d_{\pi}(a',b') = |a'-b'|$.
  \een
  \end{proposition}
  
  \begin{proof}  \eqref{it:abelianization1}  \quad By definition, the commutator subgroup $[G,G]$ is  the closure of  the set of finite products
 of commutators $[xy] = x^{-1} y^{-1} xy$. In particular $[G,G]$  is a normal subgroup. If $G$ is a Carnot group 
 then $[G,G]$ is a Lie subgroup with $T_e [G,G] = [\fg, \fg] = \oplus_{j=2}^m \fg_j$. 
 Thus $\ker d\pi(e) = \oplus_{j=2}^m \fg_j$ and hence 
  the restriction of $d\pi(e)$
  to the horizontal subspace $\fg_1$  is an isomorphism onto $G/ [G,G]$ (note that we can identifty
  $T_g G/ [G,G]$ and $G/ [G,G]$ since $G/ [G,G]$ is abelian). In particular $d\pi(e)$  is graded and hence $\pi$ is graded.
 Since $\pi$ is a group homomorphism we have  $d\pi(g) \circ (\ell_g)_* = d\pi(e)$ and the remaining
 assertions easily follow from this identity. 
 
 \eqref{it:abelianization2}   \quad The fact that $\varphi$ is $1$-Lipschitz follows directly from the definition of the 
 induced metric. Since the Carnot-Caratheodory distance is left-invariant, so is the induced metric $d_\pi$.
 Thus $d_\pi(a,b) = N(a-b)$ since $G/ [G,G]$ is abelian. The fact that $\pi$ is graded implies the $N$ is $1$-homogeneous
 and hence a norm. 
  \end{proof}

 To show that the limit map $\hat f$ is Lipschitz and to ensure that the normalisation $f_{x,r}(e) =e$ implies
 that $\hat f(e) = e$ we use the following facts, which are a  simple consequence of the Poincar\'e inequality.
 
 \begin{proposition} \label{pr:quantitative_lebesgue_point} Let $X'$ be a complete, separable metric space.
Let $U \subset G$ be open, let $f \in W^{1,p}(U; X')$ and let  
$g: U \to \R$ be the common bound for the weak derivatives of  the maps $x \mapsto d(z, f(x))$  in Definition~\ref{de:sobolev_carnot_new}.  
\ben
\item 
Suppose that $x$ is $p$-Lebesgue point of $f$ and $g$.
Then there exists a constant $C_x < \infty$ such that, for all $0 < s < \dist(x, U)$, 
\begin{equation}  \label{eq:pointwise_oscillation}
  \av_{B(x,s)} (d')^p(f(y), f(x))  d \mu \le C_x s^p
\end{equation}
\item  \label{it: quantitative_lebesgue_point2}  If $g$ is bounded by $L$ a.e\  on $B(x,R)$ then $f$ has a representative
which is $CL$-Lipschitz in $B(x,R/5)$. 
\een
\end{proposition}

\begin{proof}
(1) Consider the local maximal function of $g^p$ at $x$:
$$ M_x = M_{x,R} := \sup_{0 < r < R}  
 \av_{B(x,r)}  g^p(y) \, \mu(dy).$$
Since $x$ is a $p$-Lebesgue point  of $g$ we have $M_x < \infty$.
By the Poincar\'e inequality,  for every $r \in (0,R]$ there exists a  point $f_r \in X'$ such that
\begin{equation}  \label{eq:poincare_campanato}
  \av_{B(x,r)}  (d')^p(f(y), f_r)  \, \mu(dy) \le C r^p M_x.
  \end{equation}
By the triangle inequality we have $d'(f_r, f_{r/2}) \le d'(f(y), f_r) + d'(f(y), f_{r/2})$.
Integrating this estimate over $B(x,r/2)$ and using that $\mu(B(x,r)) = 2^\nu \mu(B(x,r/2))$ we deduce that
$$ d'(f_r, f_{r/2}) \le C r M_x^{1/p}.$$
Thus  the sequence $i \mapsto f_{2^{-i} r}$ is a Cauchy sequence and hence has a limit $f_0$
and  $d'(f_r, f_0) \le C r M_x^{1/p}$. Now $x$ is a $p$-Lebesgue point of $f$ and it follows from
\eqref{eq:poincare_campanato} that $f_0 = f(x)$. This gives the desired estimate with $C_x = C M_x$. 

(2) Since $g$ is bounded we have $M_{x,s}  \le  L^p$ for all $x \in B(x,R/5)$ and all $s \le \frac45 R$.
Let $x_1, x_2 \in B(x,R/5)$ be $p$-Lebesgue points of $f$
and set $r = d'(x_1,x_2)$. Then $r \le \frac25 R$. 
Application of assertion (1) with $C_x = C  M_x$ 
    gives for $i=1,2$
$$    
 \av_{B(x_i,2r)} (d')^p(f(y), f(x_i))  d \mu \le C L^p r^p.$$
Now one can integrate the inequality $d'(f(x_1), f(x_2)) \le d'(f(y), f(x_1)) + d'(f(y), f(x_2))$
over $B(x_1,r) \subset B(x_1, 2r) \cap B(x_2,2r)$ and use that   the measure of the ball $B(x_i, s)$  is
$\mu(B(x_i, s)) = s^\nu$ to deduce  that    
$$ d'(f(x_1), f(x_2)) \le C L r$$
Since the Lebesgue points are dense in $B(x,R/5)$ there exists a unique $CL$-Lipschitz function $\bar f$
on $B(x,R/5)$ which agrees with $f$ at all Lebesgue points. 
\end{proof}

 \begin{proof}[Proof of Theorem~\ref{th:pansu_diff_Lpstar}] Since the assertion is
 local we may assume that $U$ is bounded.
 Let $g: U \to \R$ be the common bound for the weak derivatives 
 of  the maps $x \mapsto d'(z, f(x))$   as  in Definition~\ref{de:sobolev_carnot_new}. 
Let  $u = \pi_{G'} \circ f$.
Since $G'/ [G', G']$ is a linear space,  $\pi_{G'}$ is $1$-Lipschitz and $U$ is bounded, it follows from Proposition~\ref{pr:composition_by_lip} that $u \in W^{1,p}(U; G'/ [G', G'])$. For $j=1, \ldots, K$ let  $g_j = X_j u$ denotes  the weak derivatives.  Let $x$ be a Lebesgue point for $f$, $g$ and $g_1, \ldots g_K$.  Fix a ball $B(e,R)$.
 
 Step 1:  There exists a subsequence $r_j \to 0$ such that $f_{x,r_j} \to \hat f$ in $L^{p*}(B(e,R), G')$.
 Moreover $ \hat f \in W^{1,p}(B(e,R);G')$   and $\hat f$ has a representative which is Lipschitz in $B(e,R/5)$ and satisfies $\hat f(e) = e$.

 Set 
 $$ G_{x,r}  := g \circ \ell_x \circ \delta_r$$
 and let  $z \in G'$. It follows directly from the behaviour of the Carnot-Caratheodory metric
 on $G'$ under left-translation and dilation that  $|D_h  d'(z, f_{x,r}(\cdot)| \le G_{x,r}$
 in $B(e,R)$  (as long as  $R r < \dist(x, G\setminus U)$).

 Since $x$ is a Lebesgue point of $g$ the sequence $G_{x,r}$ converges to a constant:
 \begin{equation}  \label{eq:convergence_upper_gradient_f}
  G_{x,r}  \to g(x)  \quad \text{in $L^p(B(e,R))$}.
  \end{equation}
 
 Applying   \eqref{eq:pointwise_oscillation} to $f$ and $B(x, \rho r)$  and unwinding definitions we see 
 that  for all $0 < \rho \le R$: 
 \begin{equation}    \label{eq:lebesgue_f_quantitative}
  \av_{B(e, \rho)}  (d')^p(f_{x,r}(y), e) \, d\mu \le C_x  \rho^p.
  \end{equation}
 Taking $\rho = R$ and using in addition   \eqref{eq:convergence_upper_gradient_f}
 we can apply the compact Sobolev embedding, 
 Theorem \ref{th:compact_sobolev_app}~\eqref{it:compact_sobolev_app2}.
    Thus there exists a subsequence $r_j \to 0$ and a map $\hat f \in W^{1,p}(B(e,R);G')$  
    such that 
    $$ f_{x,r_j} \to \hat f   \quad \text{in $L^{p*}(B(e,R);G')$ as $j  \to \infty$.} $$
    
    In particular $d'(z, f_{x,r_j}(\cdot)) \to d'(z, \hat f(\cdot))$, for all $z \in G'$,  and it follows from 
    the $L^p$ convergence of $G_{x,r}$ that $|D_h d'(z, \hat f(\cdot))| \le g(x)$.

    Thus by Proposition~\ref{pr:quantitative_lebesgue_point}~\eqref{it: quantitative_lebesgue_point2}  the map
       $\hat f$ has a representative which  is $C g(x)$-Lipschitz in 
   $B(e, R/5)$. 
   Passing to the limit in  \eqref{eq:lebesgue_f_quantitative} we see that   
    \begin{equation}    \label{eq:lebesgue_f_quantitative_limit}
  \av_{B(e, \rho)}  (d')^p(\hat f(y), e) \, d\mu \le C_x  \rho^p.
  \end{equation}
  Thus a   representative of $\hat f$  which   is $C g(x)$-Lipschitz in 
   $B(e, R/5)$ satisfies $\hat f(e) = e$.

 Step 2: The functions  $\pi_{G'}  \circ f_{x,r}$  converge  in $L^{p^*}(B(e,R); G'/[G', G']))$ to 
 a linear map $\hat u$, i.e. $\hat u(e) = 0$ and the (weak) horizontal derivatives of $\hat u$
 are constant. 
 
 Set $u = \pi_{G'} \circ f$ and define $u_{x,r}$ like $f_{x,r}$,
 i.e.,  $u_{x,r} = \delta_r^{-1} \circ \ell_{{u(x)}^{-1}} \circ u \circ \ell_{x} \circ \delta_r$.
Since the target is abelian this can actually be written in the more conventional form:
 $$  u_{x,r}(y) = \frac{u(x \delta_r y) - u(x)}{r}.$$
 
 Since $\pi_{G'}$ is a graded group homomorphism 
 it follows that $\pi_{G'} \circ f_{x,r} = u_{x,r}$.
 Since $\pi_{G'}$ is a Lipschitz map we get
     $$ u_{x,r_j} \to \hat u  = \pi_{G'} \circ \hat f  \quad \text{in $L^{p*}(B(e,R);G'/[G',G'])$    as $j  \to \infty$.}$$
In particular   $\hat u$ has a  representative which is $C g(x)$-Lipschitz in $B(0, R/5) $ and  satisfies 
 $ \hat  u(e) = 0$.
 
 In addition, we have 
 $$ X_j u_{x,r} = g_j \circ \ell_x \circ \delta_r$$
 and thus $X_j u_{x,r}$ converges to a constant:
 $$ X_j u_{x,r} \to   g_j(x)  \quad \text{in $L^p(B(e,R))$.}$$
 It follows directly from the definition of weak derivatives that $\hat u$  has  constant weak horizontal
derivatives which are given by $X_j \hat u = g_j(x)$. 
Since $\hat u(e) = 0$
it follows (e.g., from the Poincar\'e inequality)  $\hat u$ is uniquely determined by $g_j(x)$. 
Thus the whole sequence $u_{x,r}$ converges to $\hat u$.

 Step 3: Conclusion.\\
 Apply Steps 1 and 2 on balls $B(e,R)$ with $R = 1, 2, \ldots$. For each $R$ choose a further subsequence
 in Step 1 and finally choose a diagonal sequence. Thus we find a single sequence $r_k \to 0$ such that
 $$ f_{x,r_k} \to \hat f  \quad \text{in $L^{p*}_{loc}(G;G')$}.  $$
 Moreover $|D_h d'(z, \hat f(\cdot))|$ is bounded by the constant $g(x)$ for all $z \in G'$.

 Thus  $\hat f$ has a Lipschitz representative and we have already shown that
 this representative satisfies  $\hat f(e) = e$. 
 In combination with Step 2 we see that
 $$ \pi_{G'} \circ \hat f = \hat u$$
 where $\hat u$ has constant horizontal derivatives.
 Now Lemma~\ref{le:lift_of_L_maps} below implies that  $\hat f$ is a graded group homomorphism
 and $\hat f$ is uniquely determined by $\hat u$. 
 Uniqueness implies that the full sequence $f_{x,r}$ converges to $\hat f$ in $L^{p^*}_{loc}(G)$.
  \end{proof}

\bigskip\bigskip  
  \begin{proof}[Proof of   \eqref{eq:optimal_g_app}]
  Fix $z \in G'$ and write $u = u_z$. Theorem~\ref{th:pansu_diff_Lpstar} and Remark~\ref{rem_step_2_of_proof} also apply to the map $u: G \to \R$ and thus for a.e.\ $x \in U$ there
   exists a linear map $L_x: G \to \R$  (i.e. $L_x(e) = 0$ and the horizontal derivatives of $L$ are constant) such that 
  the maps $ u_{x,r}(y) := r^{-1}(u(x \delta_r y) - u(x))$ satisfy      
  $$ u_{x,r} \to L_x \quad \text{in $L^{p^*}_{\rm loc}(G)$}   \quad \text{and} \quad 
   X_i L_x \equiv (X_i u)(x).$$
   Here $X_i u$ denotes the weak  horizontal derivatives. 
  On the other hand the triangle inequality for $d'$ and the behaviour of $d'$ under left-translation and dilation
  imply that 
  $$ |u_{x,r}(y)| = \frac{|d'(z, f(x) \delta_r f_{x,r}(y)) - d'(z, f(x))|}{r} \le d(e,  f_{x,r}(y)).$$
  Passing to the limit $r_j \to 0$ we see that $|L_x(y)| \le d(e, D_P f(x)(y))$ for a.e.\  $x\in U$. 
  Taking $y = \exp tX$ and letting $t \to 0$ we  obtain  \eqref{eq:optimal_g_app}.
\end{proof}

\bigskip\bigskip
To show that  the map  $\hat f$   constructed in the proof of Theorem~\ref{th:pansu_diff_Lpstar}  is a graded group homomorphism,
we introduce some notation.
Let $X'$ be a finite-dimensional, normed, linear space. We say that a map $u:  G \to X'$ is affine if for every 
left-invariant horizontal vector 
 field $X$ the weak derivative $X u$ is constant. We say that a map $f:  G \to G'$ 
is an $L$-map if $\pi_{G'} \circ f = L$ and $L$ is affine. 

\begin{lemma} \label{le:lift_of_L_maps} Let $f$ and $f'$ be Lipschitz $L$-maps.
\ben
\item If $f$ and $f'$ agree at one point than $f \equiv f'$.
\item If $f(e)=e$ then $f$ is a graded group homomorphism.
\een
\end{lemma}

\begin{proof}  
We first show that the second assertion is an immediate consequence of the first. Let $f$ be an $L$-map with $f(e) =e$.  Since $\pi_{G'}$ is a graded group homomorphism the maps
$F_a:= \ell_{ (f(a))^{-1}} \circ f \circ \ell_a$ and $F^r:= \delta_{r^{-1}} \circ f \circ \delta_r$ are also 
$L$-maps and $F_a(e) = F^r(e) =e$. Thus $F_a = F^r =f$ and hence $f$ is a graded group homomorphism. 

Note that the closure of the group generated by $\exp \fg_1$ is $G$. Thus to prove the 
first assertion it suffices to show the following implication:
\begin{equation}  \label{eq:lift_unique_along_curves}
f(x_0) = f'(x_0)  \quad \Longrightarrow   f(x_0 \exp X) = f'(x_0 \exp X)  \quad \forall X \in \fg_1.
\end{equation}
Let $Y = X L$ be the constant horizontal derivative of $L = \pi \circ f = \pi \circ f'$.
Consider the curves 
$$ \gamma(t) = f( x_0 \exp tX), \quad \eta(t) = f'(x_0 \exp tX).$$
Since $t \mapsto x_0 \exp tX$ is a horizontal  Lipschitz curve in $G$ and since  $f$ and $f'$ are
Lipschitz, the curves  $\gamma$ and $\eta$ are rectifiable curves in $G'$   
(where $G'$ is equipped with the Carnot-Carath\'eodory metric)
and hence differentiable a.e. 
with horizontal derivative, see  \cite[Proposition 4.1]{pansu}.      

 Moreover
$$ X (\pi_{G'} \circ \gamma) =  X (\pi_{G'} \circ \eta) =  \frac{d}{dt} L(x_0 \exp tX)  = Y.$$
By  Proposition~\ref{pr:abelianization}~\eqref{it:abelianization1} there exists a unique horizontal vectorfield $Z$ on $G'$ such that
$d \pi'(g') Z(g') = Y$. Thus both $\gamma$ and $\eta$ are integral curves of $Z$. 
Moreover $Z$ is left-invariant and hence smooth. Since $\gamma(0) = \eta(0)$ it follows that
$\gamma \equiv \eta$. Taking $t=1$ we get  \eqref{eq:lift_unique_along_curves}.
\end{proof}

\section{Compact Sobolev embeddings}  \label{se:app_compact}

Here we give a proof of the compact Sobolev embedding. 
For scalar-valued maps it is observed in  \cite{garofalo_nhieu_1996} that the compactness of the  Sobolev embedding is
an immediate consequence of the Poincar\'e-Sobolev inequality. The same reasoning applies to Sobolev maps with values in metric spaces and we provide the details for the convenience of the reader.

\begin{theorem}  \label{th:compact_sobolev_app}
Let $U \subset G$ be open, let $X'$ be a metric space.  Suppose that every closed ball in $X'$ is compact. 
 Let $1 \le p < \nu$ and define $p^*$ by
$$\frac{1}{p*}  = \frac1p - \frac1\nu.$$
Let $f_k \in W_{loc}^{1,p}(B(x,r),X')$ and let
$g_k \in L^p(B(x,r))$  be a common bound for
 the weak derivatives of  the maps $x \mapsto d(z, f_k(x))$  as  in Definition~\ref{de:sobolev_carnot_new}.
Assume that, for some $a \in X'$
\begin{equation}
\sup_k   \| d'(f_k(\cdot), a) \|_{L^p(B(x,r)} + \| g_k\|_{L^p(B(x,r))} < \infty.
\end{equation}
Then:
\begin{enumerate}
\item   \label{it:compact_sobolev_app1}
there exists a subsequence $f_{k_j}$ and  a map $f_\infty$ such that 
\begin{equation} f_{k_j} \to f_\infty  \quad \text{in $L^q(B(x,r);X')$ for all $q<p^*$,}
\end{equation}
i.e.,
\begin{equation}
d'( f_{k_j}(\cdot), f_\infty(\cdot))  \to 0  \quad \text{in $L^q(B(x,r))$ for all $q<p^*$.}
\end{equation}
\item     \label{it:compact_sobolev_app2}
If, in addition,  the sequence $|g_k|^p$ is equiintegrable 
then the convergence also holds  in $L^{p^*}_{loc}(B(x,r))$.
\end{enumerate}
\end{theorem}

Recall that  a sequence of  $L^1$ functions $h_k: U \to \R$ is equiintegrable if there exists a function $\omega :[0, \infty) \to [0, \infty)$ with $\lim_{t \to 0} \omega(t) = 0$
such that for all measurable sets $A \subset U$ one has
$\int_A |h_k| \, d\mu \le  \omega(\mu(A))$. Note that if $g_k \to g$ in $L^p(U)$ then $|g_k|^p$ is equiintegrable.

For the proof we use the following simple covering result.
\begin{proposition} \label{pr:covering} Let $X$ be a metric space with a doubling measure. 
Then there exists a constant $C(X)$ with the following property.
Let $\mathcal{F}$ be a family of balls of fixed radius $s> 0$ in $X$. Then there exists
a disjointed subfamily $\mathcal G$ such that
\begin{equation}   \label{eq:cover_basic}
\bigcup_{B \in \mathcal F}   B \subset  \bigcup_{B \in \mathcal G} 5 B
\end{equation}
where $5 B$ denotes the concentric ball of five  times the radius.
 Moreover each point is contained in  at most $C(X)$ of the balls $5B$ with $B \in \mathcal G$:
 \begin{equation}  \label{eq:finite_overlap}
  \sum_{B \in \mathcal G} 1_{5B} \le C(X).
  \end{equation}
\end{proposition}

\begin{proof}
 The existence of a subfamily $\mathcal G$ is classical, see e.g. \cite{HKST_2001}, Theorem 1.2.
To show  \eqref{eq:finite_overlap} consider $x \in X$ and  let $\mathcal G_x = \{ B \in \mathcal G : x \in 5 B\}$.
Let $B(a,s)$ be a ball in $\mathcal G_x$. Then $d(x,a) \le 5 s$ and hence $B(a,s)  \subset B(x, 6 s) \subset B(a,11 s)$.
Since $\mu$ is doubling there exist a $C(X)> 0$ such that $\mu(B(x,6s) \le  \mu(B(a,11 s) \le C(X) B(a,s)$.
Since the balls in $\mathcal G_x$ are disjoint it follows that the number of ball in $\mathcal G_x$ is bounded by $C(X)$. 
\end{proof}

\begin{proof}[Proof of Theorem~\ref{th:compact_sobolev_app}]
  We first show (1).   The argument is essentially the same as in the scalar-valued case, see for example
  the proof of Theorem 1.28 in \cite{garofalo_nhieu_1996}.
  By the Poincar\'e inequality the functions $d(f_k(\cdot),a)$ are uniformly  bounded in $L^{p*}(B(x,r))$. 
It thus suffices to show convergence in $L^p(B(x,r);X')$.
Actually it suffices to show $L^p$ convergence for all $B(x,r')$ with $r' <r$ since the $L^p$ norm in $B(x,r) \setminus B(x,r')$
is small, uniformly in $k$, if $r'$ is close to $r$. Indeed,
\begin{align*} & \,  \int_{B(x,r) \setminus B(x,r')} {d'}^p(f_k(y), a) \, \mu(dy)\\
 \le   & \,  \big( \mu( B(x,r) \setminus B(x,r'))  \big)^{1- \frac{p}{p^*}}   \, \, 
 \| d'(f_k(\cdot, a)\|_{L^{p*}(B(x,r)}^p
\end{align*}
and $ \mu( B(x,r) \setminus B(x,r')) = r^\nu - (r')^{\nu}$.

Let $r' < r$. To show  convergence of a subsequence in $L^p(B(x,r');X')$ 
it suffices to show that for every $\eps >0$ there exist a compact subset $K$  of $L^p(B(x,r');X')$
such that $\sup_k \dist(f_k, K) \le \eps$.   We will take $K$ as a set of functions which are piecewise
constant on a fixed partition of $B(x,r')$.   If the partition is taken sufficiently fine 
then the Poincar\'e inequality will guarantee that all members of the sequence are $\eps$-close to $K$. 

Set $$M := \sup_k   \|d(f_k(\cdot), a) \|_{L^p(B(x,r)} + \| g_k\|_{L^p(B(x,r))}.$$
Let $j$ be an integer with $j^{-1} < \frac{1}{10}( r - r')$. 
By Proposition~\ref{pr:covering} there exist disjoint balls $B(x_i, j^{-1})$ such that the balls
$B(x_i, 5 j^{-1})$ cover $B(x,r')$, $B(x_i, 5 j^{-1}) \subset B(x,r)$ and 
each point is contained in at most $C(G)$ of the  balls $B(x_i, 5 j^{-1})$. 
Since the balls $B(x_i, j^{-1})$ are disjoint, the collection of balls is finite. 
Define a partition of $B(x,r')$ recursively by
$$ A_1 = B(x_1,5 j^{-1}) \cap B(x,r'),  \quad A_{i+1} = B(x_{i+1}, 5j^{-1}) \cap B(x,r') \setminus \bigcup_{k=1}^i A_i.$$

Let 
$$ \tilde K_j = \{ f:  B(x,r') \to X' : \text{for all $i$ the map  $f$ is constant on $A_i$}, \forall i \} .$$
Now  the Poincar\'e inequality implies that there exist $h_k \in \tilde  K_j$   such that
\begin{align*}
& \,  \int_{A_i} {d'}^p(f_k(y), h_k(y))\,  \mu(dy) \le    \int_{B(x_i, 5 j^{-1})} {d'}^p(f_k(y), h_k(y))\,  \mu(dy) \\
 \le & \, C j^{-p}   \int_{B(x_i, 5 j^{-1})} g_k^p \,  \mu(dy).
 \end{align*}
Summing over $i$ we see that
$$ \dist_{L^p}(f_k,  \tilde K_j) \le  C(G)^{1/p}  C  j^{-1} M.$$
Now set $K_j = \{ h \in \tilde K_j : \| d'(h(\cdot), a) \|_{L^p(B(x,r'))} \le M +1 \}$
 and choose $j$ so large that
$  C(G)^{1/p}  C  j^{-1} M \le \eps < 1$.
Then $K_j$ is compact  in $L^p(B(x',r); X')$  (since the functions in $K_j$ take only finitely many values and the values stay in a bounded set) and
$$ \dist_{L^p}(f_k,   K_j) \le \dist_{L^p}(f_k,  \tilde K_j) \le \eps.$$
This finishes the proof of the first assertion.

To prove the second assertion,  let $A_i \subset U$ and  $ \tilde K_j$ be as before. Let $\omega: [0, \infty) \to [0, \infty)$ be
the function in the definition of equi-integrability. The Sobolev-Poincar\'e inequality yields   
\begin{align*}
& \,  \osc_{p^*}^{p^*}( f_k,A_i)    \le C \| g_k\|_{L^p(B(x_i, 5 j^{-1}))}^{p^*} \\
\le & \, C \omega^{\frac{p^*}{p}  -1}\big(\mu(B(x_i, 5 j^{-1}))\big) \, \,  \| g_k\|_{L^p(B(x_i, 5 j^{-1}))}^{p} 
\end{align*}
Summation over $i$ yields
$$ \dist^{p^*}_{L^{p*}(B(x,r')} (f_k, \tilde  K_j) \le   \omega^{\frac{p^*}{p}  -1}(5 ^\nu j^{-\nu})  C(G) C  M^p  \quad \forall  k \in \N$$
Thus given $\eps \in (0,1]$ there exists a $j$ such that 
$$  \dist_{L^{p*}(B(x,r')} (f_k,  \tilde K_j)  \le \eps \quad \forall k \in \N.$$
By the Sobolev-Poincar\'e inequality we have $$M' := \sup_k \| d(f_k(\cdot), a)\|_{L^{p*}(B(x,r))} < \infty.$$
Thus setting
$$ K_j = \{ h \in \tilde K_j : \| h \|_{L^{p^*}(B(x,r'))} \le M' +1 \}$$
we see that  $K_j$ is compact and 
$ \dist_{L^{p*}}(f_k,   K_j) = \dist_{L^{p*}}(f_k,  \tilde K_j) \le \eps$.
\end{proof}

\section{Sobolev spaces defined by upper gradients}~  \label{se:sobolev}

In this section we  first recall the definition of a (weak) upper gradient, the Poincar\'e-Sobolev inequality for 
maps which possess a  $p$-integrable $p$-weak upper gradient, and the stability of $p$-weak upper gradients under $L^p$ convergence. 

 We then show that if the domain $U$  is an open subset of a Carnot group,
 then  a map is in  $W^{1,p}(U;X')$ in the sense of Definition~\ref{de:sobolev_carnot_new} if 
 and only if it has a representative which has $p$-integrable $p$-weak upper gradient.
While this can be shown by combining standard arguments in the field,
we are not aware of a specific reference for this result. The corresponding result for scalar-valued functions defined
on open subsets of Euclidean space can be found, for example, in \cite[Theorem 7.4.5]{HKST}.

\bigskip

\bigskip

\bigskip

A comprehensive introduction to the  Sobolev spaces defined via (weak) upper gradients is given in the book \cite{HKST} by Heinonen, Koskela, 
Shanmugalingam and Tyson and we closely follow their exposition.

\subsection{Weak upper gradients and the Poincar\'e-Sobolev inequality}
Let $X = (X,d, \mu)$ be a metric measure space, i.e.  a separable metric space $(X,d)$  with a nontrivial locally finite Borel regular (outer) measure $\mu$.
A curve $\gamma: I \to X$ is a continuous map from an interval $I \subset \R$ to $X$. We say that $\gamma$ is compact or open 
if $I$ is compact or open. We define the length of a compact curve $\gamma:[a,b] \to X$  as the supremum of the numbers
$\sum_{i=1}^k d(\gamma(t_i), \gamma(t_{i-1}))$ where  the supremum is taken over all choices $t_0, \ldots t_k$ with 
$a = t_0 < t_1 < \ldots < t_k = b$ and all $k \in \N$. For a noncompact curve the length is defined as the supremum of the length of all compact subcurves.
 A curve if rectifiable if it has finite length and locally rectifiable if all compact subcurves have finite length. For a rectifiable curve $\gamma$
 we denote by $\gamma_s$ its arc-length parametrization. 
 We say that a rectifiable curve $\gamma:[a,b] \to X$  is absolutely continuous  if for every $\eps > 0$ we can find a $\delta >0$
 such that $\sum_{i=1}^k d(\gamma(b_i), \gamma(a_i)) < \eps$ whenever $(a_i, b_i) \subset [a,b]$ are non-overlapping intervals
 with $\sum_{i=1}^k (b_i - a_i) < \delta$. 
 For a rectifiable curve $\gamma: I \to X$ and a Borel function $\rho:X\ra [0, \infty]$ we define the integral
$\int_\gamma \rho\, ds$  by $\int_0^{\mathrm{length}(\gamma)}\rho \circ \gamma_s(t) \, dt$.  
If $\gamma$ is locally rectifiable $\int_\gamma \rho\, ds$ is defined as the supremum of the integrals over all compact subcurves.

 For $p \ge 1$ the  $p$-modulus of a  family $\Gamma$ of   curves
 is defined by  
\begin{align*} \modp(\Gamma)
 := \inf \Bigg\{        & \,    \int_X \rho^p \, d\mu :   \rho:  X \to [0, \infty] \text{ Borel},     
      \\ \int_\gamma \rho\, ds \ge 1  
           &     \text{ for all  locally rectifiable $\gamma \in \Gamma$.}  \Bigg\}
 \end{align*}
 We call the  Borel functions $\rho$ with $ \int_\gamma \rho\, ds \ge 1$ for all $\gamma \in \Gamma$ admissible densities. 
 Every family of non-locally rectifiable  curves has modulus zero and every family which contains a constant curve has modulus $\infty$. 
 Moreover the modulus is countably subadditive. 
 
  We say that a  family of curves is $p$-exceptional if it has $p$-modulus zero. 
We say that a property holds for $p$-a.e.  curve if there exists a curve family $N$ of zero $p$-modulus such that the property holds for all 
which do not belong to $N$. 
 A set $E$ is $p$-exceptional if
the $p$-modulus of the family of all  nonconstant (rectifiable)  curves which meet $E$ is zero.     
We denote the family of all nonconstant compact rectifiable curves by $\Gamma_{\rm rec}$.

We give the definition of a $p$-weak upper gradient directly in the setting of metric-space-valued maps.

\begin{definition}[\cite{HKST}, Section 6.2, p.\  152]     \label{de:upper_gradient} Let $U \subset G$ be open,  let $1 \le p < \infty$ and 
let $(X', d')$ be a metric space.   Let $g: U \to [0, \infty]$  be  a  Borel function. We say that  $g$ is $p$-weak upper gradient of a map
$f: U \to X'$ if for $p$-a.e. rectifiable curve $\gamma:[a,b] \to U$ 
\begin{equation}  \label{eq:upper_gradient_scalar}
d'\big( f(\gamma(b)),   f(\gamma(a)) \big) \le \int_{\gamma}  g \, ds
\end{equation}
\end{definition}

If $g$ is $p$-integrable then one easily sees that  $\int_\ga g\,ds<\infty$ on $p$-a.e. compact curve. 
This yields the following result.

\begin{proposition}[\cite{HKST}, Proposition 6.3.2]  \label{pr:upper_gradient}
Suppose that the Borel function $g: U \to [0, \infty]$ is a $p$-integrable $p$-weak upper gradient of $f: U \to X'$.
Then $p$-a.e.\  every compact rectifiable curve $\gamma$ in $U$ has the following property: $g$ is integrable 
on $\gamma$ and the pair $(f,g)$ satifies the upper gradient inequality 
 \eqref{eq:upper_gradient_scalar} on $\gamma$ and each of its compact subcurves. In particular every map
 $f: U \to X'$ that has a $p$-integrable $p$-weak upper gradient is absolutely continuous on $p$-a.e. compact curve
 in $X$. 
\end{proposition}

To state the Poincar\'e-Sobolev inequality we recall that for a measurable set $A \subset G$ and a map $f: A \to X'$ we defined the $L^p$ oscillation by 
\begin{equation}
 \osc_p(f,  A)  := \inf_{a \in X'}  \left(   \int_{A} d^p(f(x), a)  \mu(dx)   \right)^{1/p}.
 \end{equation}

\begin{theorem}[\cite{HKST}, Thm. 9.1.15]  Let $U \subset G$ be open, let $X'$ be a metric space. Let $1 \le p < \nu$ and define $p^*$ by
$$\frac{1}{p*}  = \frac1p - \frac1\nu.$$
Let $f \in W^{1,p}(U,X')$ and let $g \in L^p(U)$ be a $p$-weak upper gradient. 
Then for every ball $B(x,r) \subset U$  
\begin{equation}  \label{eq:sobolev_poincare}
 \osc_{p^*}(f,   (B(x,r)) \leq C \|g\|_{L^p(B(x,r))}
\end{equation}
and
\begin{equation} \label{eq:p_poincare}
 \osc_{p}(  f, (B(x,r)) \le C r  \|g\|_{L^p(B(x,r))}.
\end{equation}
\end{theorem}

\begin{proof} First note that $G$ supports a $p$-Poincar\'e inequality (for scalar-valued functions)  in the sense of 
(1.3) in \cite{HKST}. A short proof is due to Varapoulos \cite{varopoulos}, see also 
  \cite[Proposition 11.17]{hajlasz_koskela_sobolev_met_poincare},  \cite[p.\ 461]{saloff_coste} or \cite[Corollary 1.6.]{garofalo_nhieu_1996}.
Note also that a $p$-weak upper gradient  $g$ can be approximated in $L^p$ by upper gradients $g_k$
(see Lemma 6.2.2. in \cite{HKST}). 

Now \eqref{eq:sobolev_poincare} follows from  Theorem 9.1.15 in  \cite{HKST} and 
the isometric embedding of $X'$ into the Banach space $\ell^{\infty}(X')$ of bounded functions on $X'$
 given by
$i(x) = d'(\cdot, x) - d'(\cdot, a)$ where $a$ is an arbitrary point in $X'$. 
Indeed,  Theorem 9.1.15
in  \cite{HKST} implies that there exists an element $h$ of $\ell^\infty(X')$ such that

$$ \left(  \int_{B(x,r)}  | (i \circ f)(y) - h|^{p^*}  \, \mu(dy)  \right)^{1/p^*} 
\le C \|g\|_{L^p(B(x,r))}.$$
Hence there exists a $\bar y \in B(x,r)$ 
s.t.  $|i \circ f(\bar y) - h|^{p^*}   \mu( B(x,r))  \le  C \|g\|^{p^*}_{L^p(B(x,r))}$.
Thus by the triangle inequality
$$ \left(  \int_{B(x,r)}  | (i \circ f)(y) - (i \circ f)(\bar y)|^{p^*}|  \, \mu(dy)  \right)^{1/p^*} 
\le 2 C \|g\|_{L^p(B(x,r))}.$$
This implies  \eqref{eq:sobolev_poincare} since $i$ is an isometric embedding.

Finally,  \eqref{eq:p_poincare} follows from   \eqref{eq:sobolev_poincare} 
 and H\"older's inequality since $\mu(B(x,r)) =   c  r^\nu$. 
\end{proof}

A key feature of the $p$-weak upper gradient is that  it  is stable under  $L^p$ convergence
in the following sense.
\begin{lemma}   \label{le:stability_p_weak_upper_gradient} Let $(X',d')$ be a complete  metric space.
Let $f_k: U \to X'$ be maps  and let  $g_k: U \to [0, \infty]$ be Borel functions and assume that $g_k$ is a $p$-weak upper gradient of $f_k$.
Assume further that  there exist a map
$f: U \to X'$ and a Borel function  $g: U \to [0, \infty]$ such that $d'(f_k, f) \to 0$ in $L^p(U)$ and $g_k \to g$ in $L^p(U)$. 
Then there exists a subsequence $f_{k_j}$ with the following property. The set $E$ where $f_{k_j}$ does not converge
is a  $\mu$-null set, the set of curves $\ga$ such that $f_{k_j}\circ\ga$ does not converge pointwise is a $p$-exceptional set  and if we define
\begin{equation}
\bar f(x) =
 \lim_{j \to \infty} f_{k_j}(x)  \quad \text{if $x \in U \setminus E$,}
\end{equation}
and extend $\bar f$ arbitrarily in $E$
 then $g$ is a $p$-weak upper gradient of $\bar f$ and
$\bar f = f$ a.e.
\end{lemma}

The proof of this result uses   two standard arguments. The first is Fuglede's lemma which can be seen as a counterpart of Fubini's theorem in metric measure spaces.

\begin{lemma}[Fuglede's lemma] 
 \label{th:convergence_ae_path} Let $X$ be a metric measure space and suppose that $f_k: X \to \R$ is a sequence 
 of Borel functions which converges in $L^p(X)$ to a Borel function $f$. 
Then there exists a subsequence $f_{k_j}$ such that 
$$\lim_{j \to \infty}  \int_\gamma |f_{k_j}- f|^p   = 0$$
for $p$-a.e. all rectifiable curves $\gamma$ in $X$.   
\end{lemma}

\begin{proof} This  follows directly from the definition of the modulus of a curve family
if we choose the subsequence such that
$$ \int_E |f_{k_j} - f|^p \, dx < 2^{-pj - j},$$
see, for example,  \cite{V71}, Thm. 28.1 or \cite{HKST}, Chapter 5.2.
\end{proof}

\bigskip\bigskip
The second  standard tool  is an improvement of  a.e.\ properties to properties which hold away from 
a $p$-exceptional set once we have a $p$-integrable $p$-weak upper gradient
(see  \cite[Lemma 6.3.5, Corollary 6.3.6]{HKST} for closely related results and arguments).
We state the result  for a metric measure space $X$.  It applies equally  to open subsets  $U \subset X$
considered as metric measure spaces with the induced metric and measure.

\begin{proposition} \label{pr:improve_to_pae}
Let $X$ be a metric measure space. Then the following assertions hold.
\ben
\item  \label{it:improve_to_pae_set} If $E \subset X$ is  a $\mu$-nullset then $p$-a.e.\  curve has zero length in $E$, 
i.e. $\mathcal L^1( \{ t \in [0, \mathrm{length}(\gamma)] : \gamma_s(t) \in E\}) = 0$ where $\gamma_s$ denotes
the arclength parametrization.
\item  \label{it:improve_to_pae_equality} Suppose $f: X \to X'$ has a $p$-integrable $p$-weak upper gradient. Assume that  there 
exists  $c \in X'$ such that
 $f = c$ $\mu$-a.e. (or assume  $X' = \R$  and $f \ge a$ $\mu$-a.e.).
Then there exists a $p$-exceptional set  $E$
such that  $f = c$ in $X \setminus E$ (or $ f \ge a$ in $X\setminus E$). 
\item \label{it:improve_to_pae_convergence} Suppose that the maps $f_j : X \to X'$ have $p$-integrable $p$-weak upper gradients $g_j$. Assume that $f_j \to f$ 
$\mu$-a.e. and that 
exists a Borel function $g$ such that $g_j \to g$ in $L^p(X)$. Then there exists a $p$-exceptional set  $E$ such that $f_j$ 
converges in $X \setminus E$. 
\een
\end{proposition}

\begin{proof} \eqref{it:improve_to_pae_set}.  Let $E' \supset E$ be a Borel null set. Then the assertion follows from the definition of the 
$p$-modulus if we consider the admissible function $\rho$ which is $\infty$ on $E'$ and zero elsewhere.

 \eqref{it:improve_to_pae_equality}. Assume $f = c$ $\mu$-a.e. and let $E$ be the set where $f \ne c$.  
 By assertion \eqref{it:improve_to_pae_set} we have $f \circ \gamma_s = c$ 
 $\mathcal L^1$ a.e. for $p$-a.e. curve. Since $f$ is absolutely continuous on $p$-a.e. curve, it follows that $p$-a.e.
 curve does not meet $E$. The same reasoning applies if $X' = \R$ and $f \ge a$ a.e.
 
  \eqref{it:improve_to_pae_convergence}. Let $E$ be set where the sequence  $f_j$ does not converge. It follows from Fuglede's lemma 
  that,
  for $p$-a.e.\ rectifiable curve $\gamma$,  we have $g_j \circ \gamma_s \to g \circ \gamma_s$ in $L^1([0, \mathrm{length}(\gamma)])$. 
  Thus the functions $f_j \circ \gamma_s$ are equicontinuous on $p$-a.e. curve  and one concludes as for assertion  \eqref{it:improve_to_pae_equality}.
\end{proof}

\subsection{Upper gradients and weak derivatives}
In this subsection we provide a proof  that a map is in the Sobolev space $W^{1,p}(U;X')$  defined by weak derivatives if and only if 
it has a represenative which possesses a $p$-integrable $p$-weak upper gradient. 
Specifically we show the following results.

\begin{proposition}    \label{pr:weak_scalar_upper_gradient} Let $U \subset G$ be open and let $1 \le p < \infty$. 
 Let  $u: U \to \R$ be in $L^p(U)$.
Then the following two assertions are equivalent:
\ben
\item $u \in W^{1,p}(U)$;
\item  $u$ has a  representative $\bar u$ which has a $p$-integrable  $p$-weak upper gradient $g$.
\een
Moroever, if $u \in W^{1,p}(U)$   then every Borel respresentative of $|D_h u|$ is a $p$-weak upper gradient. 
Conversely if $g$ is a $p$-weak upper gradient then $|D_h u| \le g$ a.e.
\end{proposition}

\begin{proposition}  \label{pr:upper_gradient_Xprime}
Let $U \subset G$ be open and let $1 \le p < \infty$. Let  $X'$ be a complete separable metric space and let
 $f: U \to X'$ be a measurable  function such that $d(f(\cdot), a) \in L^p(U)$ for some $a \in X'$.
Then the following three assertions are equivalent:
\ben
\item   \label{it:upper_distributional_X1}
There exists a representative $\bar f$ of $f$ which  has a  $p$-integrable $p$-weak subgradient $g$;
\item  \label{it:upper_distributional_X2}  for every Lipschitz function $\varphi: X' \to \R$ the function  $ \varphi \circ f  -  \varphi(a)$ is in $W^{1,p}(U)$;
\item  \label{it:upper_distributional_X3} $f \in W^{1,p}(U;X')$;
\een
Moreover if the above assertions hold, $\bar f$ is as in (\ref{it:upper_distributional_X1}), and $\bar g  \in L^p(U)$ is a Borel representative of the function $g$ in 
Definition~\ref{de:sobolev_carnot_new}
then $\bar g$ is  a $p$-weak upper gradient for $\bar f$. Conversely if $\bar g$ is a $p$-weak upper gradient of $\bar f$ then
we can take  $g = \bar g$ in  Definition~\ref{de:sobolev_carnot_new}.

The same conclusions holds if one replaces  $L^p(U;X')$ and $W^{1,p}(U;X')$ by 
$L^p_{loc}(U;X')$ and $W^{1,p}_{loc}(U;X')$, respectively.
\end{proposition}

 Similar equivalences using absolute continuity along a.e.\ horizontal curve $t \mapsto
a \exp(t X_j)$ rather than absolute continuity for $p$-a.e. rectifiable curve have been studied by 
Vodopyanov  \cite[Proposition 3, p.\ 674] {vodopyanov_bounded_distortion}.

\bigskip
\begin{proof}[Proof of  Proposition~\ref{pr:weak_scalar_upper_gradient}]  
To show the  implication 
 (1) $\Longrightarrow$ (2)  we first note that smooth functions are dense in $W^{1,p}(U)$.
 This was proved by  Friedrichs \cite{friedrichs_1944}  (in local coordinates)
  who observed   that for a $C^1$ vectorfield $X$ 
  the  commutator $ \tilde J_\eps = [X, J_\eps]$, where $J_\eps$ is the usual (Euclidean) mollification, satisfies $\tilde J_\eps u \to 0$ in 
  $L^p_{loc}$ for $u \in L^p$;  see
   also Thm.\  1.13 and Thm.\  A.2 in 
   \cite{garofalo_nhieu_1996}).

   Now let $u_k$ be a sequence of smooth functions such that $u_k \to u$ and $X_i u_k \to h_i$ in $L^p(U)$ where $h_i$ are the weak
    horizontal derivatives of $u$. Then $g_k := |D_h u_k|$ is  an upper gradient of $u_k$ and $g_k \to  (\sum h_i^2)^{1/2} = |D_h u|$. 
    Let $g$ be a Borel representative of $|D_h u|$.
   Then  it follows from Lemma~\ref{le:stability_p_weak_upper_gradient}  that $u$ has a representative  $\bar u$ such that $g$ is a $p$-weak upper gradient of $\bar u$.

For the converse implication one uses essentially absolute continuity on a.e.\  curve  $t \mapsto a \exp(tX_j)$ and Fubini's theorem.
For the convenience of the reader we sketch some details.
By a partition of unity it suffices to show that the weak derivatives exists in a small neighbourhood of any point in $U$.
Let $X$ be a left-invariant vectorfield.
Let $B' \subset \R^{N-1}$ be a (small) ball around $0$ and consider a smooth surface $\Psi: B' \to G$ which is transversal to $X$. 
Then $\Phi(t, x') = \Psi(x') \exp tX$ defines a smooth diffeomorphism of $(-\delta, \delta) \times B'$ to its image if $\delta > 0$ is small enough.
Moreover  $\partial_1 \Phi = X \circ \Phi$. Since the Haar measure $\mu$ is biinvariant,  the pull-back measure $\Phi^* \mu$ is invariant under 
translation in $x_1$ direction, i.e.
$ \Phi^* \mu = dx_1 \otimes \mu'$.  Consider the curves $\gamma_{x'}(t) = \Phi(t, x') = \Psi(x') \exp tX$ and the family
$\Gamma_E = \{ \gamma_{x'} : x' \in E\}$.  
Using Fubini's theorem one easily checks that $\modp( \Gamma_E) = 0$ implies $\mathcal L^{N-1}(E) = 0$
(or, equivalently,  $\mu'(E) =0$). 

Set $\tilde u = u \circ \Phi$. 
Then $t \mapsto \tilde u(t, x')$ is absolutely continuous for $\mathcal L^{N-1}$-a.e.\  $x'$ and 
$| \tilde u(b, x')  - \tilde u(a,x')| \le \int_a^b \tilde g(t, x') \, dt$ where $\tilde g = g \circ \Phi$.  Set $\tilde u = u \circ \Phi$. It is then easy to show that  
the difference
quotients $\Delta^s \tilde u ;= s^{-1} ( u(t+s,x') -  u(t, x'))$ are controlled by a family of one-dimensional convolutions of $g$
and hence  a subsequence $s_j \downarrow 0$  converges weakly in $L^p_{loc}$ (for $p=1$ use the Dunford-Pettis theorem) to a function $h \in L^p$ with
$|h| \le g$ a.e. Then $h$ is a weak derivative of $\tilde u$, i.e. $\int \tilde u \, \partial_1 \varphi  \,  \, \Phi^* \mu = - \int h \,  \varphi  \,  \, \Phi^* \mu$. 
Unwinding definitions,  we see that $h = \tilde h \circ \Phi^{-1}$ is the desired weak derivative $X u$. Moreover $|h| \le g$  a.e. from which we deduce
$|D_h u| \le g$ by considering a countable dense family of left-invariant unit vector fields. 
\end{proof}

\begin{proof}[Proof of    Proposition~\ref{pr:upper_gradient_Xprime}].  The assertion  essentially follows from  Theorem 7.1.20 in \cite{HKST} upon using the
isometric embedding of $(X',d')$ into the Banach space $V = \ell^\infty(X')$ of bounded functions on $X'$. We give a self-contained proof for the convenience of
the reader.

 We only give the argument for $W^{1,p}(U;X')$. The version for $W^{1,p}_{loc}$ is then deduced easily.

 \eqref{it:upper_distributional_X1} $\Longrightarrow$  \eqref{it:upper_distributional_X2}.  \quad Note that
  $(\Lip \varphi) \, \, g$ is a $p$-weak upper gradient of   $\varphi \circ \bar f$ and that  $\varphi \circ \bar f = \varphi \circ f$ almost everywhere.
 Thus the implication follows from 
 Proposition~\ref{pr:weak_scalar_upper_gradient}.
 
  \eqref{it:upper_distributional_X2} $\Longrightarrow$  \eqref{it:upper_distributional_X3}.  \quad This is clear since the map
  $y \mapsto  d'(y, z)$  is $1$-Lipschitz.

  \eqref{it:upper_distributional_X3} $\Longrightarrow$  \eqref{it:upper_distributional_X1}. \quad 
Set $u_z(x) = d'(z, f(x))$. Let $D \subset Z$ be a countable dense subset. 
By the definition of $W^{1,p}(U; X')$ there exists a Borel function $\bar g$ 
  such that $|D_h u_z|  \le \bar g$ almost everywhere.  
  By Proposition~\ref{pr:weak_scalar_upper_gradient} 
  for each $z \in D$ there exists a   representative $\bar u_z$ such that $\bar g$ is a $p$-weak upper gradient
  of $\bar u_z$. The main point is to show that there exist
  a $p$-exceptional set $E$ and a map $\bar f: U \setminus E \to X'$ such that
  \begin{equation}  \label{eq:property_barf}
 d'(z, \bar f(x)) = \bar u_z(x) \quad \forall x \in U \setminus E,  \quad \forall z \in D
  \end{equation}
 Then the upper gradient inequality for the function $\bar u_z$ implies that
  $$ | d'(z, \bar f (\gamma_s(t))) - d'(z, \bar f(\gamma_s(s)))| \le \int_s^t  \bar g \circ \gamma_s \, d\mathcal L^1 $$
   for all $z \in D$ and $p$-a.e. curve $\gamma$. 
   If $z_k \to z$ in $Z$ then the functions $d(z_k, \cdot)$ converge uniformly  to $d(z, \cdot)$. Thus the inequality holds for all
   $z \in Z$. Taking $z = \bar f(\gamma_s(t))$ we see that $\bar g$ is $p$-weak upper gradient of $f$. 
   
   To construct $\bar f$, note that the definition of $u_z$ and the triangle inequality imply that
 \begin{equation*}   
 \inf_{z \in D} u_z = 0,  \quad   
 \forall z,z' \in D \quad  d'(z,z') -  u_{z'} \le u_z \le d'(z,z') +  u_{z'}.
  \end{equation*}
  Since $\bar u_z$ agrees with $u_z$ a.e., it follows from 
   and Proposition~\ref{pr:improve_to_pae}~\eqref{it:improve_to_pae_equality} 
   that there exists
   a $p$-exceptional set $E$ such that 
  \begin{eqnarray}   \inf_{z \in D} u_z &=&  0 \quad  \text{in $U \setminus E$}    \nonumber \\
  d'(z,z') - \bar u_{z'} &\le& \bar u_z \le d'(z,z') + \bar u_{z'}  \quad  \text{in $U \setminus E$} 
 \label{eq:uz_zprime_pae}
   \end{eqnarray}
  for all $z, z' \in D$. 
   We claim that for all $x \in  U \setminus E$ there exists a unique $\bar z = \bar z(x) \in Z$ such that
   \begin{equation}  \label{eq:property_barz}
   d'(z, \bar z) = \bar u_z(x)  \quad \forall z \in D
   \end{equation}
   Fix $x \in U\setminus E$. By definition of $\underline u$ 
   there exist $z_k \in  D$ such that $\bar u_{z_k}(x) \to 0$. Moreover $d'(z_k, z_l) \le \bar u_{z_k}(x) + \bar u_{z_l}(x)$.
   Hence $z_k$ is a Cauchy sequence. Thus $z_k \to \bar z$ and taking $z' = z_k$ 
   in  \eqref{eq:uz_zprime_pae}   we get \eqref{eq:property_barz}.
   If also $d'(z,  \tilde z) = \bar u_z(x)$ for all $z \in D$ then we get $d'(z, \bar z) = d(z, \tilde z)$ for $z \in D$. Thus $\tilde z = \bar z$. 
   
   We now define $\bar f(x) = \bar z(x)$ for $x \in U \setminus E$. Since $\bar u_z(x) = u_z(x) = d'(z, f(x))$ for a.e.\ $x \in U$ it follows
   from the uniqueness for $\bar z$ that $\bar f = f$ a.e.  
\end{proof}

\bibliography{product_quotient}
\bibliographystyle{amsalpha}
\end{document}